\documentclass[11pt,a4paper]{article}%
\usepackage[centertags]{amsmath}
\usepackage{amsfonts}
\usepackage{amssymb}
\usepackage{amsthm}
\usepackage{epsfig}
\usepackage{setspace}
\usepackage{ae}
\usepackage{eucal}
\usepackage[usenames]{color}%
\usepackage{graphicx}

\theoremstyle{plain}
\newtheorem{thm}{Theorem}[section]
\newtheorem{lem}[thm]{Lemma}
\newtheorem{cor}[thm]{Corollary}

\theoremstyle{definition}
\newtheorem{defi}[thm]{Definition}

\theoremstyle{remark}

\newtheorem{rem}[thm]{Remark}

\begin{document}

\title{On global schemes for highly degenerate Navier Stokes equation systems}
\author{J\"org Kampen }
\maketitle

\begin{abstract} 
First order semi-linear coupling of scalar hypoelliptic equations of second order leads to a natural class of incompressible Navier Stokes equation systems, which encompasses systems with variable viscosity and essentially Navier Stokes equation systems on manifolds. We introduce a controlled global solution scheme which is based on a) local contraction results in function spaces with polynomial decay of some order at spatial infinity related to the polynomial growth factors of standard a priori estimates of densities and their derivatives for hypoelliptic diffusions of H\"{o}rmander type (cf. \cite{KS}), and on b) a controlled equation system where we discuss variations of the scheme we considered in \cite{KB3}. Global regularity of the controlled velocity functions and the control function is obtained. We supplement our notes on global bounds of the Leray projection term and related controlled Navier Stokes equation schemes in \cite{KNS,K3, KB2, KB3, KIF2}. Some arguments for linear upper bounds for the control function are added. 
\end{abstract}


\section{Introduction}
The classical incompressible Navier-Stokes equation in $n$-dimensional Euclidean space for the velocity $\mathbf{v}=(v_1,\cdots ,v_n)^T$ and the scalar pressure $p$, with initial data $\mathbf{h}=(h_1,\cdots ,h_n)^T$ and with viscosity $\nu>0$, i.e., the equation
\begin{equation}\label{nav}
\left\lbrace\begin{array}{ll}
    \frac{\partial\mathbf{v}}{\partial t}-\nu \Delta \mathbf{v}+ (\mathbf{v} \cdot \nabla) \mathbf{v} = - \nabla p, \\
    \\
\nabla \cdot \mathbf{v} = 0,~~~~t\geq 0,~~x\in{\mathbb R}^n,\\
\\
\mathbf{v}(0,.)=\mathbf{h},
\end{array}\right.
\end{equation}
has its special form due to Galilei invariance in flat space. As outlined in \cite{BS} this symmetry fixes the highly constrained structure of the equation, especially the coefficient of the nonlinear convection term. Although there are rather natural generalisations of the Navier Stokes equation model on Riemannian manifolds, there is some freedom of choice concerning the description of the coupling of the velocity field to the curvature in such cases, where this choice can be determined by other types of symmetries which fit with the manifold considered, e.g. Killing symmetry for spheres. Here we just note that these  and other phenomena, such as the fact that realistic modelling of fluids sometimes requires variable viscosity, motivate generalisations of the classical incompressible Navier-Stokes equation. We are concerned with such a generalisation where we start with the classical Navier Stokes equation in its equivalent Leray projection form, i.e., the equation system
\begin{equation}\label{Navleray}
\left\lbrace \begin{array}{ll}
\frac{\partial v_i}{\partial t}-\nu\sum_{j=1}^n \frac{\partial^2 v_i}{\partial x_j^2} 
+\sum_{j=1}^n v_j\frac{\partial v_i}{\partial x_j}=\\
\\ \hspace{1cm}\int_{{\mathbb R}^n}\left( \frac{\partial}{\partial x_i}K_n(x-y)\right) \sum_{j,k=1}^n\left( \frac{\partial v_k}{\partial x_j}\frac{\partial v_j}{\partial x_k}\right) (t,y)dy,\\
\\
\mathbf{v}(0,.)=\mathbf{h}.
\end{array}\right.
\end{equation}
Recall that the restriction of incompressibility reduces to the condition of incompressibility of the initial data for (\ref{Navleray}), i.e., the condition
\begin{equation}
\mbox{div}~\mathbf{v}(0,.)=\mbox{div}~\mathbf{h}=0.
\end{equation}
Recall furthermore that the pressure in (\ref{nav}) is determined by the solution of 
(\ref{Navleray}) in the form
\begin{equation}\label{pp}
p(t,x)=-\int_{{\mathbb R}^n}K_n(x-y)\sum_{j,k=1}^n\left( \frac{\partial v_k}{\partial x_j}\frac{\partial v_j}{\partial x_k}\right) (t,y)dy,
\end{equation}
where 
\begin{equation}
K_n(x):=\left\lbrace \begin{array}{ll}\frac{1}{2\pi}\ln |x|,~~\mbox{if}~~n=2,\\
\\ \frac{1}{(2-n)\omega_n}|x|^{2-n},~~\mbox{if}~~n\geq 3\end{array}\right.
\end{equation}
is the Poisson kernel. We are interested in $n\geq 3$, although our considerations may be applied in the case $n=2$ with modifications related to the different growth behavior and the different singularity of the Laplacian kernel and its first order derivatives in that case. We mention that $|.|$ denotes the Euclidean norm and $\omega_n$ denotes the area of the unit $n$-sphere. Next we shall generalize the Navier-Stokes equation system in its Leray projection form having in mind the global scheme we considered in \cite{KB2,KB3, KNS, K3}. We recall the main idea of the scheme in its most simple form as it was discussed in \cite{KB3} in order to indicate some differences to the generalized systems considered here.
One difference is that the equation systems considered include systems with highly degenerated second order coefficients. Concerning the solution scheme, the main difference is then related to an additional spatial polynomial growth factor in the H\"{o}rmander or Kusuoka-Stroock estimates of the densities of approximating equations (cf. our discussion below).
The global scheme considered in \cite{KB3} is based on local contraction results in strong norms and on the choice of dynamically defined control functions which ensure that spatial polynomial decay of a certain order is inherited from time step to time step and which helps in order to show that the solution is bounded in strong norms over time. For a linear upper bound of the Leray projection term we need less, as we indicated in \cite{KB3} and argue here more specifically. Let us reconsider these ideas from a slightly different point of view. We assume step size one in transformed coordinates by the time transformation $t=\rho_l \tau$, where the time step size $\rho_l$ will be small in general. The subscript $l$ in $\rho_l$ indicates that the time step size may be dependent on the time step number $l$. However, there are some obvious restrictions concerning the dependence on the time step number $l$ if we want to have a global scheme. For some versions of our scheme  with more sophisticated control functions even a time step size with an uniform lower bound can be chosen. Having computed the functions $v^{\rho,l-1}_i(l-1,.)$ 
for $1\leq i\leq n$ and $l\geq 1$, where $v^{\rho,0}_i(l-1,.):=h_i(.)$ for $1\leq i\leq n$, we consider the Leray projection form of the incompressible Navier-Stokes equation at each time step $l\geq 1$ on the domain $[l-1,l]\times {\mathbb R}^n$, i.e., the equation
\begin{equation}\label{qparasystnav2l}
\left\lbrace \begin{array}{ll}
\frac{\partial v^{\rho,l}_i}{\partial \tau}-\rho_l\nu\sum_{j=1}^n \frac{\partial^2 v^{\rho,l}_i}{\partial x_j^2} 
+\rho_l\sum_{j=1}^n v^{\rho,l}_j\frac{\partial v^{\rho,l}_i}{\partial x_j}\\
\\
=\rho_l\int_{{\mathbb R}^n}\sum_{i,j=1}^n \left( \frac{\partial v^{\rho,l}_i}{\partial x_j}\frac{\partial v^{\rho,l}_j}{\partial x_i}\right) (\tau,y)\frac{\partial}{\partial x_i}K_n(x-y)dy,\\
\\
\mathbf{v}^{\rho,l}(l-1,.)=\mathbf{v}^{\rho,l-1}(l-1,.).
\end{array}\right.
\end{equation}  
In order to have a global scheme the time step size should be at least $\rho_l\sim\frac{1}{l}$. Some polynomial decay assumption and regularity assumption on the data is useful in order to prove that the scheme is global. For $n\geq 3$ and for the classical model with constant viscosity (or even for classical Navier Stokes equations on manifolds) a condition of form $v^{\rho,l-1}_i(l-1,.)\in H^m\cap C^m$ for an integer $m$ with $m>\frac{1}{2}n$ is an appropriate choice in order to prove convergence of the local scheme to a classical solution via local contraction estimates. For the generalized highly degenerate model of this paper with diffusions satisfying a H\"{o}rmander condition we shall need stronger conditions of polynomial decay. The reasoning is quite similar in both cases. For the generalisations considered in this paper, polynomial decay assumptions (along with regularity assumptions) are very useful as they can be combined with Kusuoka-Stroock estimates for the densities related to the part of the operator which satisfies the H\"{o}rmander condition. Let us consider the simple Navier-Stokes equation model first.
The local solution of the incompressible Navier-Stokes equation in Leray-projection form is constructed via a functional series
\begin{equation}\label{funciv}
v^{\rho ,l }_i=v^{\rho,l-1}_i+\sum_{k=1}^{\infty} \delta v^{\rho, l,k}_i, 1\leq i\leq n,
\end{equation} 
where $v^{\rho, l,0}_i:=v^{\rho,l-1}_i$, and where for the most simple scheme $v^{\rho,l,1}_i$ solves 
 \begin{equation}\label{scalparasystlin10v}
\left\lbrace \begin{array}{ll}
\frac{\partial v^{\rho,l,1}_i}{\partial \tau}-\rho_l \nu\sum_{j=1}^n \frac{\partial^2 v^{\rho,l,1}_i}{\partial x_j^2}=\\
\\ 
-\rho_l\sum_{j=1}^n v^{\rho,l-1}_j(l-1,.)\frac{\partial v^{\rho,l-1}_i}{\partial x_j}\\
\\
+\rho_l\int_{{\mathbb R}^n}\sum_{j,m=1}^n \left( \frac{\partial v^{\rho,l-1}_j}{\partial x_m}\frac{\partial v^{\rho,l-1}_m}{\partial x_j}\right) (l-1,y)\frac{\partial}{\partial x_i}K_n(x-y)dy,\\
\\
{\bf v}^{\rho,l,1}(l-1,.)={\bf v}^{l-1}(l-1,.).
\end{array}\right.
\end{equation} 
Furthermore, the functional increments $\delta v^{\rho,k+1,l}_i=v^{\rho,k+1,l}_i-v^{\rho,k,l}_i,~1\leq i\leq n$ solve
\begin{equation}\label{deltaurhok0}
\left\lbrace \begin{array}{ll}
\frac{\partial \delta v^{\rho,k+1,l}_i}{\partial \tau}-\rho_l \sum_{j=1}^n \frac{\partial^2 \delta v^{\rho,k+1,l}_i}{\partial x_j^2}\\
\\
=-\rho_l\sum_{j=1}^n v^{\rho,k-1,l}_j\frac{\partial \delta v^{\rho,k,l}_i}{\partial x_j}
-\rho_l\sum_j\delta v^{\rho,k,l}_j\frac{\partial v^{\rho,k,l}}{\partial x_j}+\\ 
\\
\rho_l\int_{{\mathbb R}^n}K_{n,i}(x-y){\Big (} \left( \sum_{j,m=1}^n\left( v^{\rho,k,l}_{m,j}+v^{\rho,k-1,l}_{m,j}\right)(\tau,y) \right)  \delta v^{\rho,k,l}_{j,m}(\tau,y) {\Big)}dy,\\
\\
\mathbf{\delta v}^{\rho,k+1,l}(l-1,.)= 0,
\end{array}\right.
\end{equation}
and where $\delta v^{\rho,1,l}_j= v^{\rho,1,l}_j-v^{\rho,0,l}:=v^{\rho,1,l}_j-v^{\rho,l-1}_i(l-1,.)$. Note that $\delta v^{\rho,1,1}_j= v^{\rho,1,1}_j-h_j$ at the first time-step.
In \cite{KB3} we have shown that the functional series $\left( v^{\rho,l,k}_i\right)_{k\geq 1}$ converges to a local solution for an appropriate choice of the time step size $\rho_l$ in strong $C^0\left( [l-1,l], H^{m}\right)$, or $C^1\left( [l-1,l], H^{m}\right)$-norms via contraction estimates (supremum with respect to time). These contraction estimates can be based on Gaussian a priori estimates for densities, Young inequalities, Fourier transforms, and standard estimates for products in $H^m$. 

Now let us reconsider controlled Navier-Stokes equation systems, where we consider a variation of the scheme in (cf. \cite{KB3}) from a slightly different point of view in preparation of natural generalisations aimed at in this paper.
For a regular control function $\mathbf{r}=\left(r_1,\cdots ,r_n \right)^T:[0,\infty)\times {\mathbb R}^n\rightarrow {\mathbb R}^n$ the equation for the controlled velocity function 
\begin{equation}
\mathbf{v}^{r}:=\mathbf{v}+\mathbf{r},
\end{equation}
in original coordinates becomes  
\begin{equation}\label{Navleray2}
\left\lbrace \begin{array}{ll}
\frac{\partial v^r_i}{\partial t}-\nu\sum_{j=1}^n \frac{\partial^2 v^r_i}{\partial x_j^2} 
+\sum_{j=1}^n v^r_j\frac{\partial v^r_i}{\partial x_j}=
+\frac{\partial r_i}{\partial t}\\
\\
-\nu\sum_{j=1}^n \frac{\partial^2 r_i}{\partial x_j^2} 
+\sum_{j=1}^n r_j\frac{\partial v^r_i}{\partial x_j}+\sum_{j=1}^n v^r_j\frac{\partial r_i}{\partial x_j}-\sum_{j=1}^n r_j\frac{\partial r_i}{\partial x_j}
\\
\\+\int_{{\mathbb R}^n}\left( \frac{\partial}{\partial x_i}K_n(x-y)\right) \sum_{j,k=1}^n\left( v^r_{k,j}v^r_{j,k}\right) (t,y)dy\\
\\
-2\int_{{\mathbb R}^n}\left( \frac{\partial}{\partial x_i}K_n(x-y)\right) \sum_{j,k=1}^n\left( v^r_{k,j}r_{j,k}\right) (t,y)dy\\
\\
-\int_{{\mathbb R}^n}\left( \frac{\partial}{\partial x_i}K_n(x-y)\right) \sum_{j,k=1}^n\left( r_{k,j}r_{j,k}\right) (t,y)dy,\\
\\
\mathbf{v}^r(0,.)=\mathbf{h}.
\end{array}\right.
\end{equation}
This equation for $v^r_i\in C^{1,2}\left(\left[0,\infty\right)\times {\mathbb R}^n\right) , 1\leq i\leq n$ may be solved for an appropriate control function space $R$ such that the summand  $$v_i\in C^{1,2}\left(\left[0,\infty\right)\times {\mathbb R}^n\right)$$ for $1\leq i\leq n$ is a global classical solution of the incompressible Navier Stokes equation. The idea is to choose at the beginning of each time step $l$ control functions $r^l_i:[l-1,l]\times {\mathbb R}^n\rightarrow {\mathbb R},~1\leq i\leq n$ such that the controlled function becomes bounded on this domain while - in the most simple case - the increment of the control function is bounded by a constant which is fixed, i.e., especially independent of the time step number $l$. This leads to a global linear bound of the control functions and a global bound of the controlled velocity functions $v^{r}_i$. Even in this most simple case we can then conclude that there exist global classical solutions.

The construction is done time-step by time step on domains $\left[l-1,l\right]\times {\mathbb R}^n,~l\geq 1$, where for $1\leq i\leq n$ the restriction of the control function component $r_i$ to $\left[l-1,l\right]\times {\mathbb R}^n$ is denoted by 
$r^l_i$. The local functions $v^{r,\rho,l}_i$ with $v^{r,\rho,l}_i(\tau,x)=v^{r,l}_i(t,x)$ are defined inductively on $\left[l-1,l\right]\times {\mathbb R}^n$ along with the control function $r^l$ via the Cauchy problem for
\begin{equation}
\mathbf{v}^{r,\rho,l}:=\mathbf{v}^{\rho,l}+\mathbf{r}^l.
\end{equation}
We mention here that the control functions $r^l_i$ are chosen at every time step $l\geq 1$. This means that we can analyse the local behavior by analysis of functional sequences where the only reference to the control function is with respect to the initial data $v^{r,\rho,l-1}_i(l-1,.)$. 

Here, $\mathbf{v}^{\rho,l}=\left(v^{\rho,l}_1,\cdots ,v^{\rho,l}_n \right)^T$ is the time transformed solution of the incompressible Navier Stokes equation (in Leray projection form) restricted to the domain
$\left[l-1,l\right]\times {\mathbb R}^n$. Note that the local solution function at time-step $l\geq 1$, i.e., 
\begin{equation}
\mathbf{v}^{r,\rho,l}=\left(v^{\rho,l}_1+r^l_1,\cdots v^{\rho,l}_n+r^l_n\right)^T,
\end{equation}
satisfies the equation   
\begin{equation}\label{Navleraycontrolledlint}
\left\lbrace \begin{array}{ll}
\frac{\partial v^{r,\rho,l}_i}{\partial \tau}-\rho_l\nu\sum_{j=1}^n \frac{\partial^2 v^{r,\rho,l}_i}{\partial x_j^2} 
+\rho_l\sum_{j=1}^n v^{r,\rho,l}_j\frac{\partial v^{r,\rho,l}_i}{\partial x_j}=\\
\\
\frac{\partial r^l_i}{\partial \tau}-\rho_l\nu\sum_{j=1}^n \frac{\partial^2 r^l_i}{\partial x_j^2} 
+\rho_l\sum_{j=1}^n r^l_j\frac{\partial v^{r,\rho,l}_i}{\partial x_j}\\
\\
+\rho_l\sum_{j=1}^n v^{r,\rho,l}_j\frac{\partial r^l_i}{\partial x_j}-\rho_l\sum_{j=1}^n r^l_j\frac{\partial r^l_i}{\partial x_j}
\\
\\+\rho_l\int_{{\mathbb R}^n}\left( \frac{\partial}{\partial x_i}K_n(x-y)\right) \sum_{j,k=1}^n\left( \frac{\partial v^{r,\rho,l}_k}{\partial x_j}\frac{\partial v^{r,\rho,l}_j}{\partial x_k}\right) (\tau,y)dy\\
\\
-2\rho_l\int_{{\mathbb R}^n}\left( \frac{\partial}{\partial x_i}K_n(x-y)\right) \sum_{j,k=1}^n\left( \frac{\partial v^{r,\rho,l}_k}{\partial x_j}\frac{\partial r^l_j}{\partial x_k}\right) (\tau,y)dy\\
\\
-\rho_l\int_{{\mathbb R}^n}\left( \frac{\partial}{\partial x_i}K_n(x-y)\right) \sum_{j,k=1}^n\left( \frac{\partial r^l_k}{\partial x_j}\frac{\partial r^l_j}{\partial x_k}\right) (\tau,y)dy,\\
\\
\mathbf{v}^{r,\rho,l}(l-1,.)=\mathbf{v}^{r,\rho,l-1}(l-1,.).
\end{array}\right.
\end{equation}

At each time step $l\geq 1$ the functions $v^{r,\rho,l-1}_i(l-1,.)$ and $r^{l-1}_i(l-1,.)$ are defined in a regular space by inductive assumption, and we are free to choose the functions $r^l_i$ at the next time step $l\geq 1$ within a regular function space $R$ with the restriction that $r^l_i(l-1,.)=r^{l-1}_i(l-1,.)$. 
There are several possibilities for defining the control functions $r^l_i$ in order to get an upper bound of the Leray projection term. We shall construct a bounded solution with sophisticated control functions and linearly bounded solutions with less sophisticated control functions. However, concerning simplicity of the control function leads to a linear bound with respect to time for the velocity functions. Note that this is sufficient for existence of global classical solutions. The price to pay for the simplicity of the control function is that we need a refinement of the contraction, and we shall consider alternatives where this is not the case. We have two types of simple control functions. One simple type of control functions is based on the following idea. Assume inductively that we can realize a certain growth behavior with respect to time up to the time step number $l-1$ of the form 
\begin{equation}
D^{\alpha}_xv^{r,\rho,l-1}_i(l-1,.)\sim \sqrt{l-1} \mbox{ for }|\alpha|\leq m
\end{equation}
at the beginning of some time step $l\geq 1$ (inductive assumption). Now assume that we have constructed the control function up to time $l-1\geq 0$ such that
\begin{equation}
v^{r^{l-1},\rho,l}_i(l-1,.):=v^{\rho,l}_i(l-1,.)+r^{l-1}_i(l-1,.)
\end{equation}
are the initial data of time step $l\geq 1$. Let $v^{r^{l-1},\rho,l,1}_i$ and $\delta v^{r^{l-1},\rho,l,k}_i$ be solutions of the uncontrolled equations (\ref{scalparasystlin10v}) with data $v^{r^{l-1},\rho,l}_i(l-1,.)$ and (\ref{deltaurhok0}).
The local contraction result
\begin{equation}
{\big |}\delta v^{r^{l-1},\rho,l,k}_i{\big |}_{C^0\left( [l-1,l]\times H^{m}\right) }\leq \frac{1}{2}{\big |}\delta v^{r^{l-1},\rho,l,k-1}_i{\big |}_{C^0\left( [l-1,l]\times H^{m}\right)}
\end{equation}
for $m\geq 2$, and for all $1\leq i\leq n$ ensures (as can be shown easily) that the limit 
\begin{equation}\label{funciv2}
\mathbf{v}^{r^{l-1},\rho ,l }=\mathbf{v}^{r^{l-1},\rho,l-1}+\sum_{k=1}^{\infty} \delta \mathbf{v}^{r^{l-1},\rho,l,k}
\end{equation} 
of the corresponding local functional series represents a local solution of the incompressible Navier Stokes equation on the domain $[l-1,l]\times {\mathbb R}^n$. For a time step size $\rho_l$ of order
\begin{equation}
\rho_l\sim \frac{1}{l},
\end{equation}
we shall reconsider below the argument that a global scheme can be defined. For appropriate inductively defined control functions classical representations of the linear approximations $v^{r,\rho,l,1}_i$ and of the increments $\delta v^{r,\rho,l,k}_i,~k\geq 1$ in terms of convolutions with the fundamental solution of a heat equation with viscosity $\rho_l$ show that
\begin{equation}\label{firstincrg0}
D^{\alpha}_x\delta v^{r,\rho,l,1}_i\sim 1,~\mbox{ for }|\alpha|\leq m
\end{equation}
and 
\begin{equation}\label{secondincrgr0}
D^{\alpha}_x\delta v^{r,\rho,l,k}_i\sim \left( \frac{1}{\sqrt{l}}\right)^{k-1},~\mbox{ for }|\alpha|\leq m \mbox{ and }~k\geq 2.
\end{equation}
We shall consider details of the proof below, even in a more general situation. Our choice of the control functions $r^l_i$ is related to the observations (\ref{firstincrg0}) and (\ref{secondincrgr0}), and motivate our definition of a control functions $r^l_i$ (or a part of the control function) in \cite{KB1} and \cite{KB2}, where we defined
\begin{equation}\label{deltarl0}
\delta r^l_i=r^l_i-r^{l-1}_i(l-1,.)=-\delta v^{r,\rho,l,1}_i.
\end{equation}
This implies that we have
\begin{equation}\label{funciv3}
\begin{array}{ll}
v^{r,\rho ,l }_i=v^{r^{l-1},\rho,l-1}_i+\delta r^l_i+\sum_{k=1}^{\infty} \delta v^{r^{l-1},\rho, l,k}_i\\
\\
=v^{r,\rho,l-1}_i+\sum_{k=2}^{\infty} \delta v^{r^{l-1},\rho, l,k}_i
,~1\leq i\leq n.
\end{array}
\end{equation}
Note: since we choose $\delta r^l_i$ once at time step $l\geq 1$ before we compute the higher order term we may compute $v^{r,\rho ,l }_i$ as follows. At time step $l\geq 1$ we start with the the data $v^{r,\rho,l-1}_i$ and determine functions $v^{\rho,l,1}_i$ 
\begin{equation}\label{scalparasystlin10v23}
\left\lbrace \begin{array}{ll}
\frac{\partial v^{\rho,l,1}_i}{\partial \tau}-\rho_l \nu\sum_{j=1}^n \frac{\partial^2 v^{\rho,l,1}_i}{\partial x_j^2}=\\
\\ 
-\rho_l\sum_{j=1}^n v^{\rho,l-1}_j(l-1,.)\frac{\partial v^{\rho,l-1}_i}{\partial x_j}\\
\\
+\rho_l\int_{{\mathbb R}^n}\sum_{j,m=1}^n \left( \frac{\partial v^{\rho,l-1}_j}{\partial x_m}\frac{\partial v^{\rho,l-1}_m}{\partial x_j}\right) (l-1,y)\frac{\partial}{\partial x_i}K_n(x-y)dy,\\
\\
{\bf v}^{\rho,l,1}(l-1,.)={\bf v}^{r,l-1}(l-1,.).
\end{array}\right.
\end{equation} 
With a slight abuse of our notation so far the functional increments $\delta v^{\rho,k+1,l}_i=v^{\rho,k+1,l}_i-v^{\rho,k,l}_i,~1\leq i\leq n$ then solve the same equation as in (\ref{deltaurhok0})
\begin{equation}\label{deltaurhok023}
\left\lbrace \begin{array}{ll}
\frac{\partial \delta v^{\rho,k+1,l}_i}{\partial \tau}-\rho_l \sum_{j=1}^n \frac{\partial^2 \delta v^{\rho,k+1,l}_i}{\partial x_j^2}\\
\\
=-\rho_l\sum_{j=1}^n v^{\rho,k-1,l}_j\frac{\partial \delta v^{\rho,k,l}_i}{\partial x_j}
-\rho_l\sum_j\delta v^{\rho,k,l}_j\frac{\partial v^{\rho,k,l}}{\partial x_j}+\\ 
\\
\rho_l\int_{{\mathbb R}^n}K_{n,i}(x-y){\Big (} \left( \sum_{j,m=1}^n\left( v^{\rho,k,l}_{m,j}+v^{\rho,k-1,l}_{m,j}\right)(\tau,y) \right)  \delta v^{\rho,k,l}_{j,m}(\tau,y) {\Big)}dy,\\
\\
\mathbf{\delta v}^{\rho,k+1,l}(l-1,.)= 0.
\end{array}\right.
\end{equation}
Then we may choose the increment of the control function $\delta r^l_i=r^l_i-r^l_i(l-1,.)$
and 
\begin{equation}\label{funciv3simpl}
\begin{array}{ll}
v^{r,\rho ,l }_i
=v^{r,\rho,l-1}_i+\delta v^{r,\rho,l,1}_i+\delta r^l_i+\sum_{k=2}^{\infty} \delta v^{\rho, l,k}_i
,~1\leq i\leq n.
\end{array}
\end{equation}
We have to note that the increments $\delta v^{\rho, l,k}_i, k\geq 2$ are not the same as before, although they seem to be determined by identical equations (\ref{deltaurhok023}) and (\ref{deltaurhok0}). However, these equations are not identical since for $k+1=2$ the initial data $v^{r,l,0}_i=v^{r,\rho,l-1}_i(l-1,.)$ enter into the equation, and this has certainly an effect for the higher order terms as well. However these considerations simplifies the analysis. We do not have to establish local contraction results for the whole controlled system (\ref{Navleraycontrolledlint}) but only for the original type of Navier Stokes equations.
Indeed, this way of construction makes it possible to do the local analysis of the higher order terms analogously as for the local scheme - only the initial data are different at each time step.

\begin{rem}
We have to mention another ambiguity in notation here. At time step $l\geq 1$ we understand
\begin{equation}\label{vrl-1}
v^{r,\rho ,l-1 }_i=v^{\rho,l-1}_i+r^{l}_i=v^{\rho,l-1}_i+r^{l-1}_i+\delta r^l_i
\end{equation}
where the left side of the equation (\ref{vrl-1}) at time step $l-1$ is $v^{r,\rho,l-1}_i(l-1,.)=v^{\rho,l-1}_i+r^{l-1}_i$. Disambiguation is clear if a certain time step $l$ is fixed. 
\end{rem}

Now from (\ref{funciv3}) we have
\begin{equation}\label{funciv3}
\begin{array}{ll}
v^{r,\rho ,l }_i
=v^{r,\rho,l-1}_i+\sum_{k=2}^{\infty} \delta v^{r,\rho, l,k}_i\\
\\
\sim \sqrt{l-1}+\frac{1}{\sqrt{l}}
\end{array}
\end{equation}
for all $1\leq i\leq n$. This implies that
\begin{equation}
v^{r,\rho ,l }_i\sim \sqrt{l}~~
\mbox{${\Bigg (}$ or}~
\left( v^{r,\rho ,l }_i\right) ^2\sim l{\Bigg )},
\end{equation}
and heritage of this property renders the scheme global. We think in terms of algorithms if we define $\delta r^l_i$ as in (\ref{deltarl0}). Whatever choice is made for $\delta r^l_i$ it is an important property of the choice just made that on the original time scale $t=\rho_l\tau$ we have in original time 
\begin{equation}
 r^l_i\left(\sum_{m=1}^l \rho_l,. \right) \sim l,
\end{equation}
where the property $\rho_l\sim \frac{1}{l}$ ensures that we have a linear bound on a transformed time scale which is still global. This reasoning implies that there is a global linear bound of the Leray projection term on this transformed time scale.

For models with constant viscosity or Navier-Stokes equation models on manifolds the control functions defined in \cite{K3,KNS} are an alternative choice. Let us remark why the choice made there is not suitable for highly degenerate Navier Stokes equation systems. Essentially the choice in \cite{K3,KNS} is of the form
\begin{equation}\label{choicek3}
 \delta r^l_i(\tau,x)=\int_{l-1}^l\int_{{\mathbb R}^n}\phi^l_i(s,y)G_l(\tau-s,x-y)dyds,
 \end{equation}
where $G_l$ is the fundamental solution of
\begin{equation}
\frac{\partial p}{\partial \tau}-\rho_l\Delta p=0
\end{equation}
on $[l-1,l]\times {\mathbb R}^n$, and 
\begin{equation}
\phi^l_i=\phi^{l,v}_i,
\end{equation}
and
\begin{equation}\label{sourcek3}
\phi^{l,v}_i(\tau,.)=-\frac{v^{r,\rho,l-1}_i(l-1,.)}{C} \mbox{ for }\tau\in \left(l-1,l\right].
\end{equation}
The idea then is that for small time step size $\rho_l$ the value of the convoluted source term in (\ref{choicek3}) is close to the source function in (\ref{sourcek3}) and the value of the later source function has no time step size factor $\rho_l>0$. Since all other terms in the equation for $v^{r,\rho,l}_i$ have the small time step size $\rho_l$ as a coefficient, the convoluted source term value dominates all the other value which determine the growth of the increment $\delta v^{r,\rho,l}_i$ at time step $l\geq 1$. The definition in (\ref{sourcek3}) with its minus sign 'stabilizes' the dynamics of the controlled scheme in the sense that for all time step numbers $l\geq 1$ we get
\begin{equation}
\sup_{x\in {\mathbb R}^n}{\big |}v^{r,\rho,l-1}_i(l-1,.){\big |}\leq C\Rightarrow~\sup_{x\in {\mathbb R}^n}{\big |}v^{r,\rho,l}_i(l,.){\big |}\leq C.
\end{equation}
Depending on the time step size we get a similar growth control for all $m\geq 2$ and all multiindices $\alpha$ with $|\alpha|\leq m$
\begin{equation}
\sup_{x\in {\mathbb R}^n}{\big |}D^{\alpha}_xv^{r,\rho,l-1}_i(l-1,.){\big |}\leq C~\Rightarrow~\sup_{x\in {\mathbb R}^n}{\big |}D^{\alpha}_xv^{r,\rho,l}_i(l,.){\big |}\leq C.
\end{equation}
Well with a initial choice $r^0_i$ which may be chosen to be $r^0_i=0$, and the choice of the increment of the control function in (\ref{choicek3}) this implies immediately
\begin{equation}
r^l_i=r^0_i+\sum_{m=1}^l\delta r^l_i\leq \sum_{m=1}^l\frac{C}{C}=l,
\end{equation}
such that we have a linear global bound of the control function (at least). This implies the existence of a global linear bound of the value function
\begin{equation}
v^{\rho,l}_i=v^{r,\rho,l}_i-r _i,
\end{equation}
and this implies the existence of a global regular solution. This reasoning cannot be applied in this simple form to the model class of highly degenerate Navier Stokes equation systems considered in this paper because the related H\"{o}rmander type estimates involve a spatial polynomial growth factor with respect to the spatial variables, and we need the inheritance of polynomial decay of the value functions in time in order to ensure the scheme is a global one.

For the highly degenerate Navier Stokes equation models considered in this paper we may consider extended control functions. 
We may define the equation for the control functions $r^l_i$ such that a solution for $r^l_i$ leads to some source terms on the right side of of (\ref{Navleraycontrolledlint}), which are then constructed time-step by time-step such that they control the growth of the controlled velocity function and the control function itself. These source terms serve as 'growth consumption terms' for this controlled velocity function and sometimes for the control function itself, and are denoted by $\phi^{l,v}_i,~1\leq i\leq n$ and $\phi^{l,r}_i,~1\leq i\leq n$. For $1\leq i\leq n$ we may define
\begin{equation}\label{phiv}
\phi^{l,v}_i:\left[ l-1,l\right] \times {\mathbb R}^n\rightarrow {\mathbb R}
\end{equation}
as a consumption term for the growth of the local controlled velocity function $v^{r,\rho,l}_i$, and sometimes we may define for each $1\leq i\leq n$ a continuous extension of the function in (\ref{phiv}) as a growth consumption function for the control function $r^l_i$, i.e. a function
\begin{equation}\label{phir}
\phi^{l,r}_i:\left[ l-1,l\right] \times {\mathbb R}^n\rightarrow {\mathbb R}
\end{equation}
which will control the growth of the local control function $r^l_i$ it self. Next on the right side of (\ref{Navleraycontrolledlint}) we have functions $v^{r,\rho,l}_i$ which are not known at the beginning of the construction time step $l\geq 1$. However, we have the data $v^{r,\rho,l-1}_i(l-1,.)$ which will turn out to be close enough to the functions $v^{r,\rho,l}_i$ on the domain $[l-1,l]\times {\mathbb R}^n$ in order to do some relevant  growth estimates. Accordingly, the consumption functions $\phi^{l,v}_i$ and $\phi^{l,r}_i$ may be defined in terms of this information which we know at the beginning of time step $l\geq 1$, and which may determine the local growth consumption  function $\phi^l_i$ in the  equation (\ref{controllint}) below which is derived from the right side of (\ref{Navleraycontrolledlint}) - if we take the direct approach.
We emphasize that in the following at time step $l\geq 1$ function names with superscript $l-1$ are to be understood as functions evaluated at time $\tau=l-1$, i.e., for problems at time step $l\geq 1$ on the domain $[l-1,l]\times {\mathbb R}^n$ we use
\begin{equation}
v^{r,\rho,l-1}_i\equiv v^{r,\rho,l-1}_i(l-1,.)~\mbox{ and }r^{l-1}_i=r^{l-1}_i(l-1,.)
\end{equation}
as synonyms.
A direct approach would lead to a nonlinear equation of the form
\begin{equation}\label{controllint}
\left\lbrace \begin{array}{ll}
\frac{\partial r^l_i}{\partial \tau}-\rho_l\nu\sum_{j=1}^n \frac{\partial^2 r^l_i}{\partial x_j^2} 
+\rho_l\sum_{j=1}^n r^l_j\frac{\partial v^{r,\rho,l-1}_i}{\partial x_j}\\
\\
+\rho_l\sum_{j=1}^n v^{r,\rho,l-1}_j\frac{\partial r^l_i}{\partial x_j}+\rho_l\sum_{j=1}^n r^l_j\frac{\partial r^l_i}{\partial x_j}\\
\\+\rho_l\int_{{\mathbb R}^n}\left( \frac{\partial}{\partial x_i}K_n(x-y)\right) \sum_{j,k=1}^n\left( \frac{\partial v^{r,\rho,l-1}_k}{\partial x_j}\frac{\partial v^{r,\rho,l-1}_j}{\partial x_k}\right) (l-1,y)dy\\
\\
-2\rho_l\int_{{\mathbb R}^n}\left( \frac{\partial}{\partial x_i}K_n(x-y)\right) \sum_{j,k=1}^n\left( \frac{\partial v^{r,\rho,l-1}_k}{\partial x_j}(l-1,y)\frac{\partial r^l_j}{\partial x_k}(\tau,y)\right)dy\\
\\
-\rho_l\int_{{\mathbb R}^n}\left( \frac{\partial}{\partial x_i}K_n(x-y)\right) \sum_{j,k=1}^n\left( \frac{\partial r^l_k}{\partial x_j}\frac{\partial r^l_j}{\partial x_k}\right) (\tau,y)dy=\phi^l_i,\\
\\
\mathbf{r}^l(l-1,.)=\mathbf{r}^{l-1}(l-1,.).
\end{array}\right.
\end{equation}
If the source term $\phi^l_i$ is chosen appropriately, then the control system in (\ref{controllint}) is a possible construction, since we can solve such equations which are similar to the original Navier Stokes equation locally. Another direct possibility, preferable from an algorithmic point of view, is a linearisation. This turns out to be sufficient and is certainly preferable from a constructive and from an algorithmic point of view. A simple choice may be the equation
\begin{equation}\label{controllint2}
\left\lbrace \begin{array}{ll}
\frac{\partial r^l_i}{\partial \tau}-\rho_l\nu\sum_{j=1}^n \frac{\partial^2 r^l_i}{\partial x_j^2} 
+\rho_l\sum_{j=1}^n r^{l-1}_j\frac{\partial v^{r,\rho,l-1}_i}{\partial x_j}\\
\\
+\rho_l\sum_{j=1}^n v^{r,\rho,l-1}_j\frac{\partial r^{l-1}_i}{\partial x_j}+\rho_l\sum_{j=1}^n r^{l-1}_j\frac{\partial r^{l-1}_i}{\partial x_j}
\\
\\+\rho_l\int_{{\mathbb R}^n}\left( \frac{\partial}{\partial x_i}K_n(x-y)\right) \sum_{j,k=1}^n\left( \frac{\partial v^{r,\rho,l-1}_k}{\partial x_j}\frac{\partial v^{r,\rho,l-1}_j}{\partial x_k}\right) (l-1,y)dy\\
\\
-2\rho_l\int_{{\mathbb R}^n}\left( \frac{\partial}{\partial x_i}K_n(x-y)\right) \sum_{j,k=1}^n\left( \frac{\partial v^{r,\rho,l-1}_k}{\partial x_j}(l-1,y)\frac{\partial r^{l-1}_j}{\partial x_k}(l-1,y)\right)dy\\
\\
-\rho_l\int_{{\mathbb R}^n}\left( \frac{\partial}{\partial x_i}K_n(x-y)\right) \sum_{j,k=1}^n\left( \frac{\partial r^{l-1}_k}{\partial x_j}\frac{\partial r^{l-1}_j}{\partial x_k}\right) (l-1,y)dy=\phi^l_i,\\
\\
\mathbf{r}^l(l-1,.)=\mathbf{r}^{l-1}(l-1,.).
\end{array}\right.
\end{equation}  
As we get local contraction results for the local higher order correction terms $\delta v^{r,\rho,l,k}_i,~k\geq 2$ it makes sense to define the control function such that it compenates the first increments $\delta v^{r,\rho,l,1}_i=v^{r,\rho,l,1}_i-v^{r,\rho,l-1}_i(l-1,.)$, as we discussed above. Since the control function $r^l_i$ is chosen once at each time step and with the notation as in (\ref{scalparasystlin10v23}) and (\ref{deltaurhok023}) we have
\begin{equation}
\delta v^{r,\rho,l,k}_i=\delta v^{\rho,l,k}_i \mbox{ for }k\geq 2,
\end{equation}
if we understand that in a slight abuse of notation- as we discussed above-
\begin{equation}
\delta v^{\rho,l,k}_i=\delta v^{r^{l-1},\rho,l,k}_i
\end{equation}
This has the advantage that we can do the local analysis essentially without the control function, because at time step $l$ it appears only in the data $v^{r,\rho,l-1}_i(l-1,)$ and $r^{l-1}_i(l-1,.)$- this will change the higher correction terms $\delta v^{r^{l-1},\rho,l,k}_i$ in general but the form of the equation which determine them is the same as in the uncontrolled case. We then choose the increment $\delta r^l_i$ such that the increment of the controlled velocity function is controlled.
 We did that in \cite{KB3} and consider which kind of control functions may be chosen. Since we are interested in growth control with respect to time it is natural to consider the equation for the incremental functions 
\begin{equation}
\delta r^l_i:=r^l_i-r^{l-1}_i(l-1,.).
\end{equation}
If we look at the direct approach, then equation (\ref{controllint2}) becomes
\begin{equation}\label{controllint3}
\left\lbrace \begin{array}{ll}
\frac{\partial \delta r^l_i}{\partial \tau}-\rho_l\nu\sum_{j=1}^n \frac{\partial^2 \delta r^l_i}{\partial x_j^2} 
=-\rho_l\Delta r^{l-1}_i(l-1,.)-\rho_l\sum_{j=1}^n r^{l-1}_j\frac{\partial v^{r,\rho,l-1}_i}{\partial x_j}\\
\\
-\rho_l\sum_{j=1}^n v^{r,\rho,l-1}_j\frac{\partial r^{l-1}_i}{\partial x_j}-\rho_l\sum_{j=1}^n r^{l-1}_j\frac{\partial r^{l-1}_i}{\partial x_j}
\\
\\-\rho_l\int_{{\mathbb R}^n}\left( \frac{\partial}{\partial x_i}K_n(x-y)\right) \sum_{j,k=1}^n\left( \frac{\partial v^{r,\rho,l-1}_k}{\partial x_j}\frac{\partial v^{r,\rho,l-1}_j}{\partial x_k}\right) (l-1,y)dy\\
\\
+2\rho_l\int_{{\mathbb R}^n}\left( \frac{\partial}{\partial x_i}K_n(x-y)\right) \sum_{j,k=1}^n\left( \frac{\partial v^{r,\rho,l-1}_k}{\partial x_j}(l-1,y)\frac{\partial r^{l-1}_j}{\partial x_k}(l-1,y)\right)dy\\
\\
+\rho_l\int_{{\mathbb R}^n}\left( \frac{\partial}{\partial x_i}K_n(x-y)\right) \sum_{j,k=1}^n\left( \frac{\partial r^{l-1}_k}{\partial x_j}\frac{\partial r^{l-1}_j}{\partial x_k}\right) (l-1,y)dy+\phi^l_i,\\
\\
\delta \mathbf{r}^l(l-1,.)=\mathbf{0}.
\end{array}\right.
\end{equation} 
Note that except for $\phi^l_i$ all the terms on the right side of (\ref{controllint3}) have a factor $\rho_l$ which represents a small time step size at time step $l\geq 1$. For the higher order correction terms of the converging functional series contraction results show that the source term $\phi^l_i$ can dominate these corrections. Furthermore, except for first term on the right side of the first equation in (\ref{controllint3}) all terms on the right side of the first equation in (\ref{controllint3}) involve products of controlled velocity functions $v^{r,\rho,l-1}_m$ and control functions $r^l_m$ and/or spatial derivatives of these functions. In order to design a scheme which preserves a certain degree of polynomial decay it is useful to have representations for the approximating increments $\delta v^{r,\rho,l,k}_i$ and $\delta r^{l}_i$  which involve only such products.
Next we discuss a list of control functions. Each of them can be used in order to prove that the corresponding controlled Navier-Stokes equation scheme is global. We do not repeat the control functions of the direct approach above, but let is remark that the list of simple control functions below (simple compared to the direct approach) lead to schemes which are closely linked to the direct approach, i.e., the difference is small at each time step $l\geq 1$ of the scheme due to the small time step size $\rho_l>0$. Especially, if we add a source function $\phi^l_i$ in the following list of control functions for the classical model, then we convolute it with the 'Gaussian' $G_l$ and this means that in the equation for the controlled velocity functions $v^{r,\rho,l}_i$ the terms
\begin{equation}
\frac{\partial}{\partial \tau}r^l_i-\rho_l\Delta r^l_i 
\end{equation}
equals a source term plus additional function terms with coefficient $\rho_l$.
The direct approach is a little cumbersome to write down, and we do not need so much formal complexity. We mention only several possibilities of control functions which can be used in the case of the more general equation system below as well. The control functions are defined with respect to the classical Navier Stokes equation model, but they list for the generalised highly degenerate model can be obtained by replacement of the 'Gaussian' density $G_l$ by the H\"{o}rmander density $G^l_H$ (cf. below).
\begin{itemize}
 \item[i)] We may define for $(\tau,x)\in (l-1,l]\times {\mathbb R}^n$
 \begin{equation}
 \delta r^l_i(\tau,x)=-\delta v^{r,\rho,l,1}_i(\tau,x)+\int_{l-1}^{\tau}\phi^l_i(s,y)G_l(\tau-s,x-y)dyds,
 \end{equation}
where $G_l$ is the fundamental solution of
\begin{equation}
\frac{\partial p}{\partial \tau}-\rho_l\Delta p=0
\end{equation}
on $[l-1,l]\times {\mathbb R}^n$, and 
\begin{equation}
\phi^l_i=\phi^{l,v}_i+\phi^{l,r}_i,
\end{equation}
along with
\begin{equation}
\phi^{l,r}_i(\tau,.)=0 \mbox{ for }\tau\in \left(l-1,l\right],
\end{equation}
and
\begin{equation}
\phi^{l,v}_i(\tau,.)=-\frac{v^{r,\rho,l-1}_i(l-1,.)}{C} \mbox{ for }\tau\in \left(l-1,l\right].
\end{equation}
\item[ia)] A variation of i) is the choice
\begin{equation}
 \delta r^l_i(\tau,x)=\int_{l-1}^{\tau}\phi^l_i(s,y)G_l(\tau-s,x-y)dyds,
 \end{equation} 
 with $\phi^l_i$ as in i). This possibility is denoted as a subitem because it works for the situation of Navier Stokes equation with constant viscosity or the classical Navier Stokes equation on manifolds, but it does not work for highly degenerate Navier Stokes equations considered in this paper in general.
\item[ii)] As we explained above we may also consider the simplified control function
\begin{equation}
 \delta r^l_i=-\delta v^{r,\rho,l,1}_i
 \end{equation}
\item[iii)]In \cite{KB3} we defined for $(\tau,x)\in [l-1,l]\times {\mathbb R}^n$
 \begin{equation}
 \delta r^l_i(\tau,x)=-\delta v^{r,\rho,l,1}_i(\tau,x)+\int_{l-1}^{\tau}\phi^l_i(s,y)G_l(\tau-s,x-y)dyds,
 \end{equation}
where $G_l$ is the fundamental solution of
\begin{equation}
\frac{\partial p}{\partial \tau}-\rho_l\Delta p=0
\end{equation}
on $[l-1,l]\times {\mathbb R}^n$, and 
\begin{equation}
\phi^l_i=\phi^{l,v}_i+\phi^{l,r}_i,
\end{equation}
along with
\begin{equation}
\phi^{l,r}_i(\tau,.)=-\frac{r^{l-1}_i(l-1,.)}{C^2} \mbox{ for }\tau\in \left(l-1,l\right],
\end{equation}
and
\begin{equation}
\phi^{l,v}_i(\tau,.)=-\frac{v^{r,\rho,l-1}_i(l-1,.)}{C} \mbox{ for }\tau\in \left[l-1,l\right].
\end{equation}
In \cite{KB3} we discussed this scheme for $C>1$.
 \item[iv)] We may also a scheme as in iii), but with the weights exchanged, i.e.,
 \begin{equation}
\phi^{l,r}_i(\tau,.)=-\frac{r^{l-1}_i(l-1,.)}{C} \mbox{ for }\tau\in \left(l-1,l\right],
\end{equation}
and
\begin{equation}
\phi^{l,v}_i(\tau,.)=-\frac{v^{r,\rho,l-1}_i(l-1,.)}{C^2} \mbox{ for }\tau\in \left[l-1,l\right].
\end{equation}
or a scheme without the summand  $\phi^{l,r}_i$. The proof that the scheme is global changes accordingly.
\item[v)] We mention also a fifth possibility which illustrates which terms in our representations have to be controlled. It is sufficient to define
\begin{equation}\label{scalparasystlin10vhercrgen2}
\begin{array}{ll}
\delta r^l_i(\tau,x)=-\int_{{\mathbb R}^n}v^{\rho,l-1}_i(l-1,y)G_l(\tau,x;s,y)dy
+v^{\rho,l-1}_i(l-1,x)
\end{array}
\end{equation}
For the general model with H\"{o}rmander diffusion discussed below the related control function is
\begin{equation}\label{scalparasystlin10vhercrgen2}
\begin{array}{ll}
\delta r^l_i(\tau,x)=-\int_{{\mathbb R}^n}v^{\rho,l-1}_i(l-1,y)G^l_H(\tau,x;s,y)dy
+v^{\rho,l-1}_i(l-1,x)
\end{array}
\end{equation}
\end{itemize}

The simple choices ia) and ii) lead to global linear bounds of the Leray projection term. The choice ii) was essentially considered in \cite{KNS,K3} and the choice ia) was considered in \cite{KB2, KB3}. A local contraction result with respect to strong norms ensures that the time-local functional series
\begin{equation}
v^{r,\rho,l}_i=v^{r,\rho,l-1}_i(l-1,.)+\sum_{k=1}^{\infty}\delta v^{r,\rho,l,k}_i
\end{equation}
converges and provides an upper bound for the growth $\delta v^{r,\rho,l}_i(l,.)=v^{r,\rho,l}_i(l,.)-v^{r,\rho,l-1}_i(l-1,.)$. This growth scales with the time step size $\rho_l$ while the source term in ia)
\begin{equation}
\int_{l-1}^{\tau}\left( -\frac{v^{r,\rho,l-1}_i(l-1,.)}{C}\right) G_l(\tau-s,x-y)dyds
\end{equation}
has no factor $\rho_l$ and is close to the integrand $\left( -\frac{v^{r,\rho,l-1}_i(l-1,.)}{C}\right)$ as $\rho_l$ becomes small. If the modulus $|v^{r,\rho,l-1}_i(l-1,.)|$ becomes larger or equal to $C$ then this damping term dominates the growth of $|\delta v^{r,\rho,l}_i(l,.)|$. Similar for strong norms. The effect is that for a certain step size $\rho_l$ an upper bound of the controlled value function can be established of the form
\begin{equation}
{\big |}v^{r,\rho,l}_i{\big |}_{C^0\left([l-1,l]\times H^{m}\right) }\leq C
\end{equation}
for some $m\geq 2$.
As the control function $r^l_i$ have a linear bound then (because $\delta r^l_i\sim 1$) we get a linear global bound for the velocity functions themself. The idea of the control function in ii) is different. It focuses on the idea to get a global linear bound of the Leray projection term. The idea is that the local contraction result may be refined such that
with the special choice of $\delta r^l_i$ as in ii) the growth of
\begin{equation}\label{funciv3}
\begin{array}{ll}
v^{r,\rho ,l }_i=v^{r,\rho,l-1}_i+\sum_{k=2}^{\infty} \delta v^{r,\rho, l,k}_i
,~1\leq i\leq n,
\end{array}
\end{equation}
is bounded by some constant $\sim\sqrt{l}$ while the control function growth again linearly with respect to the time step number. The analysis can be done if we interpret the scheme in (\ref{funciv3}) as a scheme which involves the control function, but we mention that in our notation with (\ref{scalparasystlin10v23}) and (\ref{deltaurhok023}) we may write 
\begin{equation}\label{funciv3simplified}
\begin{array}{ll}
v^{r,\rho ,l }_i=v^{r,\rho,l-1}_i+\sum_{k=2}^{\infty} \delta v^{\rho, l,k}_i
,~1\leq i\leq n,
\end{array}
\end{equation}
which simplifies the analysis a bit (again for the increments on the right side we suppressed the upper inde $r^{l-1}$). Note that in the limit a representation (\ref{funciv3simplified}) leads to the same result as a limit in (\ref{funciv3}) even if we interpret the latter as a representation of the direct approach.

The preceding considerations lead to a linear bound of the Leray projection term and hence to global existence. The other alternatives are refinements. The possibility 
in iii) is a refinement of ia). We shall see that we can still get an upper bound for the controlled value function $v^{r,\rho,l}_i$ while the additional summand can ensure that the control function it self has an upper bound which is independent of the time step number $l\geq 1$. This leads to a global uniform upper bound. We shall have a closer look at this below. 

All these ideas can be applied with some additional modifications to a more general class of models. Next we shall define this more general class of equation systems.
Recall that an operator $L$ with $C^{\infty}$ coefficients, and defined on an open set $\Omega\subseteq {\mathbb R}^n$ is called hypoelliptic if any distribution $u$ on $\Omega$ which solves $Lu=f$ for some $f\in C^{\infty}$ is itself in $C^{\infty}$. Here, the function space $C^{\infty}$ denotes the set of smooth functions as usual. This definition applies to scalar equations, but we may generalize this and define similar concepts for vector-valued equation straightforwardly. However, in this paper we are only interested in a nonlinear coupling of linear second order equations where the linear second order part (including the first order terms) satisfies a hypoellipticity condition. 
For positive  natural numbers $m,n$ consider a matrix-valued function 
\begin{equation}\label{vcoeff}
x\rightarrow (v^q_{ji})^{n,m}(x),~1\leq j\leq n,~0\leq i\leq m
\end{equation}
 on ${\mathbb R}^n$, and $m$ smooth vector fields 
\begin{equation}\label{vvec}
V_{i}=\sum_{j=1}^n v_{ji}(x)\frac{\partial}{\partial x_j},
\end{equation}
where $0\leq i\leq m$. H\"{o}rmander showed in the scalar case that
a density exists if 
the following condition is satisfied:
for all $x\in {\mathbb R}^n$  we have
\begin{equation}\label{Hoer}
H_x={\mathbb R}^n,
\end{equation}
where
\begin{equation}\label{Hoergenx}
\begin{array}{ll}
H_x:=\mbox{span}{\Big\{} V_{i}(x), \left[V_{j},V_{k} \right](x),
\left[ \left[V_{j},V_{k} \right], V_{l}\right](x),\\
\\
\cdots |1\leq i\leq m,~0\leq j,k,l,\cdots \leq m {\Big \}}.
\end{array}
\end{equation}
Here $\left[.,.\right]$ denotes the Lie bracket of vector fields as usual. More precisely, H\"{o}rmander showed that (given $1\leq q\leq n$) the distributional Cauchy problem
\begin{equation}
	\label{hoer1}
	\left\lbrace \begin{array}{ll}
		\frac{\partial u}{\partial t}=\frac{1}{2}\sum_{i=1}^mV_{i}^2u+V_{0}u\\
		\\
		u(0,x;y)=\delta_y(x),
	\end{array}\right.
\end{equation}
has a smooth solution on $(0,\infty)\times {\mathbb R}^n$. Here $\delta_y(x)=\delta(x-y)$ is the Dirac delta distribution shifted by the vector $y\in {\mathbb R}^n$. For coefficient functions $b_i\in C^{\infty}_b$, where the latter function space $C^{\infty}_b$ denotes the function space of smooth functions with bounded derivatives,
consider an additional vector field
\begin{equation}\label{Bvvec}
V_{B}[v]:=\sum_{j=1}^n B_j(x)v_j\frac{\partial}{\partial x_j},
\end{equation}
\begin{defi}\label{defiell}
Let $n \geq 3$
\begin{equation}
D=\left\lbrace (x,y)\in {\mathbb R}^n\times {\mathbb R}^n|x=y\right\rbrace .
\end{equation}
We say that the function
\begin{equation}
\begin{array}{ll}
K^{\mbox{ell}}_n\in C^{\infty}
\left( \left( {\mathbb R}^n\times {\mathbb R}^n\right)\setminus D\right) 
\end{array}
\end{equation}
is an elliptic kernel if
\begin{equation}
K^{\mbox{ell}}_n\in O\left(|x-y|^{2-n}\right),~K^{\mbox{ell}}_{n,i}\in O\left(|x-y|^{1-n}\right).
\end{equation}
\end{defi}
It may be that kernels of linear elliptic equation satisfy stronger assumptions than those of (\ref{defiell}), but these assumptions represent what we need. In this paper we establish global scheme for the Cauchy problem
\begin{equation}\label{nsgenint}
\left\lbrace \begin{array}{ll}
\frac{\partial v_i}{\partial t}-\frac{1}{2}\sum_{j=0}^mV_{j}^2v_i+V_{B}\left[v \right] v_i\\
\\
=\int_{{\mathbb R}^n}\left( \frac{\partial}{\partial x_i}K^{\mbox{ell}}_n(x-y)\right) \sum_{j,k=1}^n\left(c_{jk} \frac{\partial v_k}{\partial x_j}\frac{\partial v_j}{\partial x_k}+\sum_{j,k=1}^n d_j\frac{\partial v_i}{\partial x_k}\right) (t,y)dy,\\
\\
\mathbf{v}(0,.)=\mathbf{h},
\end{array}\right.
\end{equation}
where $K^{\mbox{ell}}$ is an elliptic kernel. The treatment of this class of equations in (\ref{nsgenint}) can be applied for global schemes for incompressible Navier Stokes equations on manifolds. Furthermore the class represented in (\ref{nsgenint}) includes a class of incompressible Navier Stokes equations with variable viscosity. The degree of global regularity which we obtain depends on the degree of polynomial decay of the initial data $h_i$ in relation to certain polynomial growth behavior of a priori estimates of H\"{o}rmander diffusions. It seems that these a priori estimates were obtained in full generality in \cite{KS}. We note that 
the integral term may be extended but the form given in (\ref{nsgenint}) is an essential step, as it is possible to treat Navier Stokes equations on Riemannian manifolds and Navier-Stokes equations with variable viscosity or versions of compressible fluid models based on the global scheme for (\ref{nsgenint}) proposed in this paper. 
Next we recall the estimates in \cite{KS} and describe the global scheme for the system in (\ref{nsgenint}). In \cite{KS} the H\"{o}rmander diffusions are described in probabilistic terms. The result of \cite{KS} can be summarized as follows. 
\begin{thm}
	\label{stroock}
	Consider a $d$-dimensional difffusion process of the form
	\begin{equation}\label{stochm}
		\mathrm{d}X_t \ = \ \sum_{i=1}^d\sigma_{0i}(X_t)\mathrm{d}t + \sum_{j=1}^{d}\sigma_{ij}(X_t)\mathrm{d}W^j_t
	\end{equation}
with $X(0)=x\in {\mathbb R}^d$ with values in ${\mathbb R}^d$ and on a time interval $[0,T]$, and where $W_j,~1\leq j\leq n$ denotes a standard Brownian motion.
Assume that $\sigma_{0i},\sigma_{ij}\in C^{\infty}_{lb}$. Then the law of the process $X$ is absolutely continuous with respect to the Lebesgue measure, and the density $p$ exists and is smooth, i.e. 
\begin{equation}
\begin{array}{ll}
p:(0,T]\times {\mathbb R}^d\times{\mathbb R}^d\rightarrow {\mathbb R}\in C^{\infty}\left( (0,T]\times {\mathbb R}^d\times{\mathbb R}^d\right). 
\end{array}
\end{equation}
Moreover, for each nonnegative natural number $j$, and multiindices $\alpha,\beta$ there are increasing functions of time
\begin{equation}\label{constAB}
A_{j,\alpha,\beta}, B_{j,\alpha,\beta}:[0,T]\rightarrow {\mathbb R},
\end{equation}
and functions
\begin{equation}\label{constmn}
n_{j,\alpha,\beta}, 
m_{j,\alpha,\beta}:
{\mathbb N}\times {\mathbb N}^d\times {\mathbb N}^d\rightarrow {\mathbb N},
\end{equation}
such that 
\begin{equation}\label{pxest}
\begin{array}{ll}
{\Bigg |}\frac{\partial^j}{\partial t^j} \frac{\partial^{|\alpha|}}{\partial x^{\alpha}} \frac{\partial^{|\beta|}}{\partial y^{\beta}}p(t,x,y){\Bigg |}\\
\\
\leq \frac{A_{j,\alpha,\beta}(t)(1+x)^{m_{j,\alpha,\beta}}}{t^{n_{j,\alpha,\beta}}}\exp\left(-B_{j,\alpha,\beta}(t)\frac{(x-y)^2}{t}\right)
\end{array}
\end{equation}
Moreover, all functions (\ref{constAB}) and  (\ref{constmn}) depend on the level of iteration of Lie-bracket iteration at which the H\"{o}rmander condition becomes true.
\end{thm}

The theorem in (\ref{stroock})  is also sometimes formulated in a probabilistic manner. We note
\begin{cor}
	In the situation of (\ref{stroock}) above, solution $X_t^x$ starting at $x$ is in the standard Malliavin space $D^{\infty}$, and there are constants $C_{l,q}$ depending on the derivatives of the drift and dispersion coefficients such that for some constant $\gamma_{l,q}$
\begin{equation}\label{xprocessest}
|X_t^x|_{l,q}\leq C_{l,q}(1+|x|)^{\gamma_{l,q}}.
\end{equation}
Here $|.|_{l,q}$ denotes the norm where derivatives up to order $l$ are in $L^q$ (in the Malliavin sense).
\end{cor}
Note the polynomial dependence on $x$ of the factor
\begin{equation}
\frac{A_{j,\alpha,\beta}(t)(1+x)^{m_{j,\alpha,\beta}}}{t^{n_{j,\alpha,\beta}}}
\end{equation} compared to the case of constant viscosity, and it is a motivation for our definition of the control function above where we have the Laplacian of the controlled velocity function $v^{r,\rho,l}_i$ on the right side of the equation for the increment of the control at time step $l\geq 1$. For our purposes we need an additional observation which follows from the considerations of in \cite{H} and in \cite{KS}. We shall also consider this local behavior of spatial derivatives of H\"{o}rmander type densities in \cite{KH}. We remark that adjoint densities $p^*$ of densities $p$ satisfying a linear parabolic equation 
Local  adjoints of densities, and which satisfy
\begin{equation}
p(t,x;s,y)=p^*(s,y;t,x)
\end{equation}
can be constructed locally for H\"{o}rmander type densities as well. This follows from our construction in \cite{KH} as a Corollary to the theorem above. In our scheme we can make a similar use of this adjoints as we did in the case of strictly parabolic equations in \cite{KB2}, as there are similar weakly singular upper bounds for the density and its first spatial derivatives (also as a consequence of \cite{KH}).

Next we describe a global controlled scheme for the 
equation system in (\ref{nsgenint}). We describe the scheme incorporating the control function from the beginning.
We start with the description of the local scheme. We assume that $v^{r,\rho,l-1}_i(l-1,.)$. Locally on the domain $[l-1,l]\times {\mathbb R}^n$ and knowing $\mathbf{v}^{r,\rho,l-1}(l-1,.)$ we have to solve for $\mathbf{v}^{\rho,l}$ the equation
\begin{equation}\label{qparasystnav2hoerl}
\left\lbrace \begin{array}{ll}
\frac{\partial v^{\rho,l}_i}{\partial \tau}-\rho_l\frac{1}{2}\sum_{j=0}^mV_{j}^2v^{\rho,l}_i+\rho_lV_{B}\left[v^{\rho,l} \right] v^{\rho,l}_i\\
\\
=\rho_l\int_{{\mathbb R}^n}\sum_{j,m=1}^n \left(c_{jm} \frac{\partial v^{\rho,l}_m}{\partial x_j}\frac{\partial v^{\rho,l}_j}{\partial x_m}\right) (\tau,y)\frac{\partial}{\partial x_i}K^{\mbox{ell}}_n(x-y)dy,\\
\\
\mathbf{v}^{\rho,l}(l-1,.)=\mathbf{v}^{r,\rho,l-1}(l-1,.).
\end{array}\right.
\end{equation}  
We then denote $v^{r,\rho,l}_i=v^{\rho,l}_i+\delta r^l_i$ for $1\leq i\leq n$,
where the increment $\delta r^l_i$ is chosen once at the start of time step $l\geq 1$.
At the beginning of time step $l\geq 1$ we compute the series
\begin{equation}\label{funciv}
v^{\rho ,l }_i=v^{\rho,l,1}_i+\sum_{k=1}^{\infty} \delta v^{\rho,l, k+1}_i, 1\leq i\leq n,
\end{equation} 
where $v^{\rho,l,1}_i$ solves 
 \begin{equation}\label{scalparasystlin10v}
\left\lbrace \begin{array}{ll}
\frac{\partial v^{\rho,l,1}_i}{\partial \tau}-\rho_l\frac{1}{2}\sum_{j=0}^mV_{j}^2v^{\rho,l,1}_i=\\
\\ 
-\rho_lV_{B}\left[v^{\rho,l-1}(l-1,.)\right] v^{\rho,l-1}_i(l-1,.)+\\
\\
\rho_l\int_{{\mathbb R}^n}\sum_{j,m=1}^n \left(c_{jm} \frac{\partial v^{\rho,l-1}_m}{\partial x_j}(l-1,.)\frac{\partial v^{\rho,l-1}_j}{\partial x_m}(l-1,.)\right) (\tau,y)\frac{\partial}{\partial x_i}K^{\mbox{ell}}_n(x-y)dy,\\
\\
{\bf v}^{\rho,1,l}(l-1,.)={\bf v}^{r,\rho, l-1}(l-1,.).
\end{array}\right.
\end{equation} 
Furthermore, for $k\geq 1$ the functional increments $\delta v^{\rho,l,k+1}_i=v^{\rho,l,k+1}_i-v^{\rho,l,k}_i,~1\leq i\leq n$ solve
\begin{equation}\label{deltaurhok0*******}
\left\lbrace \begin{array}{ll}
\frac{\partial \delta v^{\rho,l,k+1}_i}{\partial \tau}-\rho_l\frac{1}{2}\sum_{j=0}^mV_{j}^2\delta v^{\rho,l,k+1}_i=\\
\\ 
-\rho_lV_{B}\left[v^{\rho,l,k}\right] \delta v^{\rho,l,k}_i-\rho_lV_{B}\left[\delta v^{\rho,l,k}\right] v^{\rho,l,k}_i\\ 
\\
\rho_l\int_{{\mathbb R}^n}K^{\mbox{ell}}_{n,i}(x-y){\Big (} \left( \sum_{j,m=1}^nc_{jm}\left( v^{\rho,l,k}_{m,j}+v^{\rho,l,k-1}_{m,j}\right)(\tau,y) \right)  \delta v^{\rho,l,k}_{j,m}(\tau,y) {\Big)}dy\\
\\
\mathbf{\delta v}^{\rho,l,k+1}(l-1,.)= 0,
\end{array}\right.
\end{equation}
and where $\delta v^{\rho,l,1}_j= v^{\rho,l,1}_j-v^{\rho,l,0}:=v^{\rho,l,1}_j-v^{r,\rho,l-1}_i(l-1,.)$. Note that $\delta v^{\rho,l,1}_j= v^{\rho,l,1}_j-h_j$ at the first time-step (if we choose $r^{0}_i\equiv 0$.
We shall prove a local contraction result for the increments of this scheme, where we generalise considerations in \cite{KB2} and \cite{KB3}. This leads to a local existence result of regular solutions. Note the at time step $l$ approximating solution can be represented in terms of the fundamental solution (or density) of
\begin{equation}
\frac{\partial v^{\rho,l}_i}{\partial \tau}-\rho_l\frac{1}{2}\sum_{j=0}^mV_{j}^2v^{\rho,l}_i=0.
\end{equation}
on $[l-1,l]\times {\mathbb R}^n$, which we denote by $G^l_H$.
The global scheme solves a system for the regular control function $\mathbf{r}=\left(r_1,\cdots ,r_n \right)^T:[0,\infty)\times {\mathbb R}^n\rightarrow {\mathbb R}^n$, and an equation for the controlled velocity function 
\begin{equation}
\mathbf{v}^{r}:=\mathbf{v}+\mathbf{r}.
\end{equation}
We need not solve the equations for this controlled velocity function directly, but it can be done. In any case the construction is done time-step by time step on domains $\left[l-1,l\right]\times {\mathbb R}^n,~l\geq 1$, where for $1\leq i\leq n$ the restriction of the control function component $r_i$ to $\left[l-1,l\right]\times {\mathbb R}^n$ is denoted by 
$r^l_i$. The local functions $v^{r,\rho,l}_i$ with $v^{r,\rho,l}_i(\tau,x)=v^{r,l}_i(t,x)$ are defined inductively on $\left[l-1,l\right]\times {\mathbb R}^n$ along with the control function $r^l$ via the Cauchy problem for
\begin{equation}
\mathbf{v}^{r,\rho,l}=\left(v^{\rho,l}_1+r^l_1,\cdots v^{\rho,l}_n+r^l_n\right)^T,
\end{equation}
which satisfies the equation   
\begin{equation}\label{Navleraycontrolledlintgen}
\left\lbrace \begin{array}{ll}
\frac{\partial v^{r,\rho,l}_i}{\partial \tau}-\rho_l\frac{1}{2}\sum_{j=0}^mV_{j}^2v^{r,\rho,l}_i
-\rho_lV_{B}\left[v^{r,\rho,l}\right] \ v^{r,\rho,l}_i=\\
\\
\frac{\partial r^l_i}{\partial \tau}-\rho_l\frac{1}{2}\sum_{j=0}^mV_{j}^2r^{l}_i
-\rho_lV_{B}\left[v^{r,\rho,l}\right] r^{l}_i-\rho_lV_{B}\left[r^{l}\right] v^{r,\rho,l}_i+\rho_lV_{B}\left[r^{l}\right] r^{l}_i\\
\\+\rho_l\int_{{\mathbb R}^n}\left( \frac{\partial}{\partial x_i}K^{\mbox{ell}}_n(x-y)\right) \sum_{j,k=1}^n\left( c_{jk}\frac{\partial v^{r,\rho,l}_k}{\partial x_j}\frac{\partial v^{r,\rho,l}_j}{\partial x_k}\right) (\tau,y)dy\\
\\
-2\rho_l\int_{{\mathbb R}^n}\left( \frac{\partial}{\partial x_i}K^{\mbox{ell}}_n(x-y)\right) \sum_{j,k=1}^n\left(c_{jk} \frac{\partial v^{r,\rho,l}_k}{\partial x_j}\frac{\partial r^l_j}{\partial x_k}\right) (\tau,y)dy\\
\\
-\rho_l\int_{{\mathbb R}^n}\left( \frac{\partial}{\partial x_i}K^{\mbox{ell}}_n(x-y)\right) \sum_{j,k=1}^n\left( c_{jk}\frac{\partial r^l_k}{\partial x_j}\frac{\partial r^l_j}{\partial x_k}\right) (\tau,y)dy,\\
\\
\mathbf{v}^{r,\rho,l}(l-1,.)=\mathbf{v}^{r,\rho,l-1}(l-1,.).
\end{array}\right.
\end{equation}
For $1\leq i\leq n$ the choice of the control function $r^l_i$ is mainly determined by the choice of two source functions
\begin{equation}\label{phivphir}
\begin{array}{ll}
\phi^{l,v}_i:\left[ l-1,l\right) \times {\mathbb R}^n\rightarrow {\mathbb R},\\
\\
\phi^{l,v}_i:\left[l-1,l\right) \times {\mathbb R}^n\rightarrow {\mathbb R}.
\end{array}
\end{equation}
\begin{rem}
It is a matter of taste whether we define the source function on the closed intervals $\left[ l-1,l\right]$ or on the half open intervals $\left[ l-1,l\right)$. The latter definition may be chosen in order to avoid 'overlaps'. However, since the functions involved are regularly bounded the time integral over the the closed interval and the half open interval lead to the same result.
\end{rem}

Similar as described in the case of the classical Navier Stokes equation described above these source functions are related to the source functions $\phi^l_i$, and there are different possibilities to introduce this relation. A direct approach is via
the equation
\begin{equation}\label{controllint2gen}
\left\lbrace \begin{array}{ll}
\frac{\partial r^l_i}{\partial \tau}-\rho_l\frac{1}{2}\sum_{j=0}^mV_{j}^2r^{l}_i
-\rho_lV_{B}\left[v^{r,\rho,l-1}\right] r^{l-1}_i\\
\\
-\rho_lV_{B}\left[r^{l-1}\right] v^{r,\rho,l-1}_i
+\rho_lV_{B}\left[r^{l-1}\right] r^{l-1}_i\\
\\+\rho_l\int_{{\mathbb R}^n}\left( \frac{\partial}{\partial x_i}K^{\mbox{ell}}_n(x-y)\right) \sum_{j,k=1}^n\left(c_{jk} \frac{\partial v^{r,\rho,l-1}_k}{\partial x_j}\frac{\partial v^{r,\rho,l-1}_j}{\partial x_k}\right) (l-1,y)dy\\
\\
-2\rho_l\int_{{\mathbb R}^n}\left( \frac{\partial}{\partial x_i}K^{\mbox{ell}}_n(x-y)\right) \sum_{j,k=1}^n\left( c_{jk}\frac{\partial v^{r,\rho,l-1}_k}{\partial x_j}(l-1,y)\frac{\partial r^{l-1}_j}{\partial x_k}(l-1,y)\right)dy\\
\\
-\rho_l\int_{{\mathbb R}^n}\left( \frac{\partial}{\partial x_i}K^{\mbox{ell}}_n(x-y)\right) \sum_{j,k=1}^n\left(c_{jk} \frac{\partial r^{l-1}_k}{\partial x_j}\frac{\partial r^{l-1}_j}{\partial x_k}\right) (l-1,y)dy=\phi^l_i,\\
\\
\mathbf{r}^l(l-1,.)=\mathbf{r}^{l-1}(l-1,.).
\end{array}\right.
\end{equation}
At this point where we have to determine or choose the source functions $\phi^l_i$ it is important to note that H\"{o}rmander type diffusions do not preserve a polynomial decay of a certain order in general. Note again the polynomial growth factor in (\ref{xprocessest}).

Again we emphasize that the local contraction results for the local higher order correction terms $\delta v^{r,\rho,l,k}_i$ lead us to define the control function such that it compensates the first increments $\delta v^{r,\rho,l,1}_i=v^{r,\rho,l,1}_i-v^{r,\rho,l-1}_i(l-1,.)$. This is indeed sufficient in order to define a global scheme, i.e., to get a linear upper bound of the Leray projection term for the controlled scheme on a transformed time scale. We have indicated the reasons for the classical Navier Stokes equation above. We shall show that the definition
\begin{equation}\label{control1def}
\delta r^l_i:=r^l_i-r^{l-1}_i(l-1,.)=-\delta v^{\rho,l,1}_i=-\left( v^{\rho,l,1}_i-v^{r,\rho,l-1}_i(l-1,.)\right) 
\end{equation}
leads to a global scheme for the generalized systems of equations considered in this paper, if the conditions of a certain local contraction result are satisfied. These conditions are a bit stronger than the conditions we needed for the local contraction result in \cite{KB2} and \cite{KB3}.
We then extend the definition in (\ref{control1def}), where we add source functions
\begin{equation}\label{control1def}
\begin{array}{ll}
\delta r^l_i:=r^l_i-r^{l-1}_i(l-1,.)=-\delta v^{\rho,l,1}_i=-\left( v^{\rho,l,1}_i-v^{r,\rho,l-1}_i(l-1,.)\right)\\
\\
+\int_{l-1}^{\tau}\phi^{l}_i(s,y)G_H(\tau-s,x-y)dyds,
\end{array}
\end{equation}
where $G_H$ is the fundamental solution $[l-1,l]\times {\mathbb R}^n$ of the H\"{o}rmander diffusion
\begin{equation}
	\label{hoer2}
	\begin{array}{ll}
		\frac{\partial u}{\partial t}=\frac{1}{2}\sum_{i=1}^mV_{i}^2u+V_{0}u.
	\end{array}
\end{equation}
We have remarked that the direct definition of a control function avoids a solution of an equation for the control function. We just define
\begin{equation}
\phi^l_i(\tau,x)=\phi^{l,v}_i(\tau,x)+\phi^{l,r}_i(\tau,x)
\end{equation}
where
\begin{equation}
\phi^{l,r}_i(\tau,.)=-\frac{r^{l-1}_i(l-1,.)}{C} \mbox{ for }\tau\in \left[ l-1,l\right),
\end{equation}
and
\begin{equation}
\phi^{l,v}_i(\tau,.)=-\frac{v^{r,\rho,l-1}_i(l-1,.)}{C^2} \mbox{ for }\tau\in \left[l-1,l\right).
\end{equation}

In order to prove convergence of the global scheme for Navier Stokes equation models with H\"{o}rmander diffusion we use function spaces of polynomial decay. This is due to the polynomial growth factor with respect to the spatial variables for a priori estimates of the density. This factor appears in the Kusuoka Stroock estimate and cannot be avoided. We say that a function $g\in C^{\infty}\left( {\mathbb R}^n\right) $ has polynomial decay of order $m>0$ up to derivatives of order $p>0$ at infinity if for all multiindices $\alpha=(\alpha_1,\cdots,\alpha_n)$ with order $|\alpha|:=\sum_{i=1}^n\alpha_i\leq p$ we have
\begin{equation}
|D^{\alpha}_xg(x)|\leq \frac{C_{\alpha}}{1+|x|^m}
\end{equation}
for some finite constants $C_{\alpha}$. The following existence result is closely related to the Gaussian a priori estimate in (\ref{pxest}). As we shall see in detail in the proof the reason is that for $\beta =0$ the estimate
\begin{equation}\label{pxestbeta0}
\begin{array}{ll}
{\Bigg |}\frac{\partial^j}{\partial t^j} \frac{\partial^{|\alpha|}}{\partial x^{\alpha}} p(t,x,y){\Bigg |}
\leq \frac{A_{j,\alpha,0}(\tau)(1+x)^{m_{j,\alpha,0}}}{t^{n_{j,\alpha,0}}}\exp\left(-B_{j,\alpha,0}(\tau)\frac{(x-y)^2}{\tau}\right)
\end{array}
\end{equation}
has, compared to usual Gaussian estimates of fundamental solution of operators with strictly elliptic spatial part, an additional factor $(1+x)^{m_{j,\alpha,0}}$, and any global scheme has to compensate this factor of polynomial growth. Note that this factor appears naturally in the estimates as we have shown in our alternative construction of the result in \cite{KH}.

\begin{thm}
Assume that the initial data $h_i\in C^{\infty},~1\leq i\leq n$ satisfy a polynomial decay condition, where for an integer $p\geq 2$ and for $|\alpha|\leq p$
\begin{equation}
{\big |}D^{\alpha}_xh_i(x){\big |}\leq \frac{C}{1+|x|^{q}}
\end{equation}
for
\begin{equation}
q\geq \max_{j,|\alpha|\leq p}\left\lbrace n_{j,\alpha,0},3m_{j,\alpha,0}\right\rbrace +2n+2
\end{equation}
Furthermore, assume that the vector fields $V_i,~0\leq i\leq n$ satisfy the H\"{o}rmander condition (\ref{Hoer}).
Then the Cauchy problem in (\ref{nsgenint})
has a global solution $\mathbf{v}=(v_1,\cdots ,v_n)^T$ with $v_i\in C^{p}\left( [0,\infty)\times {\mathbb R}^n\right) $.
\end{thm}

The structure of the proof of this theorem is as follows. The local contraction result stated in the next section and proved in the last section of this paper implies the existence of local regular solutions (as we shall observe at the end of this section). Then for different types of dynamically defined control functions listed above we get for the control function ii) a linear    
upper bound for the Leray projection term of the controlled Navier Stokes equation system and a linear upper bound for the control function, or for the control function i) a linear upper bound for controlled value functions of the Navier Stokes equation system and a linear upper bound of the control function. For the choice in iii) we can improve this in order to get a global upper bound which is independent of the time step number, and therefore a global uniformly bounded regular solution. Similar for the method in iv).

At the end of this section we consider the announced consequence of local regular existence of the local contraction result.

We consider the inductive construction of local regular solutions on $[l-1,l]\times {\mathbb R}^n$ by the local scheme above. At each time step $l\geq 1$ having constructed $v^{\rho,l-1}_i(l-1,.)\in C^m\cap H^m$ for $m\geq 2$ at time step $l-1$ (at $l=1$ these are just the initial data $h_i$),
as a consequence of local contraction below with respect to the norm $|.|_{C^1\left((l-1,l),  H^{m}\right) }$ for $m\geq 2$ we a have a time-local pointwise limit $v^{*,\rho,l}_i(\tau,.)=v^{*,\rho,l-1}_i(\tau,.)+\sum_{k=1}^{\infty}\delta v^{\rho,l,k}_i(\tau,.)\in H^m\cap C^m$ for all $1\leq i\leq n$, where for $n=3$ we have $H^2\subset C^{\alpha}$ uniformly in $\tau\in [l-1,l]$. For higher dimension the contraction has to be established at least for $m\geq \frac{n}{2}$ accordingly. Furthermore the functions of this series are even locally continuously differentiable with respect to $\tau\in [l-1,l]$ and hence H\"{o}lder continuous with respect to time.  Note that a local contraction below with respect to the norm $|.|_{C^0\left((l-1,l),  H^{m}\right) }$ is sufficient for our purposes as  we may prove that the first order time derivative $\frac{\partial}{\partial \tau}v^{\rho,l,k}_i(\tau,.)$  exist in $H^m$ as well for appropriate $m$ ($m>\frac{5}{2}$ is sufficient for $n=3$)  a consequence of the product rule for Sobolev spaces. We observe that $v^{\rho,l,k}_i(\tau,.)\in H^{m}$ can be obtained inductively for all $k$ for each given $m\geq 2$ and this leads to full local regularity of the limit function of the local scheme. 
If we plug in the approximating function $v^{\rho,l,k}_i(\tau,.)$ into the local incompressible highly degenerate Navier-Stokes equation system in its the Leray projection form in (\ref{nsgenint}), then from (\ref{deltaurhok0}) and from $\lim_{k\uparrow\infty}\delta v^{\rho,l,k}_j(\tau,x)=0$ and $\lim_{k\uparrow \infty}\frac{\partial \delta v^{\rho,l,k}_i}{\partial x_j}=0$ for all $(\tau,x)\in [l-1,l]\times {\mathbb R}^n$ pointwise by our local contraction result we get
\begin{equation}\label{deltaurhok0**}
\left\lbrace \begin{array}{ll}
\lim_{k\uparrow \infty}\frac{\partial \delta v^{\rho,l,k+1}_i}{\partial \tau}-\rho_l\frac{1}{2}\sum_{j=0}^mV_{j}^2\delta v^{\rho,l,k+1}_i=\\
\\ 
-\lim_{k\uparrow \infty}\rho_lV_{B}\left[v^{\rho,l,k}\right] \delta v^{\rho,l,k}_i-\rho_lV_{B}\left[\delta v^{\rho,l,k}\right] v^{\rho,l,k}_i\\ 
\\
\rho_l\lim_{k\uparrow \infty}\int_{{\mathbb R}^n}K^{\mbox{ell}}_{n,i}(x-y){\Big (} c_{jm}\left( \sum_{j,m=1}^n\left( v^{\rho,l,k}_{m,j}+v^{\rho,l,k-1}_{m,j}\right)(\tau,y) \right)  \delta v^{\rho,l,k}_{j,m}(\tau,y) {\Big)}dy=0\\
\\
\lim_{k\uparrow \infty}\mathbf{\delta v}^{\rho,l,k+1}(l-1,.)= 0,
\end{array}\right.
\end{equation}
which implies that $\lim_{k\uparrow \infty}\mathbf{\delta v}^{\rho,l,k+1}= 0$ and similar for spatial derivatives up to second order. Hence, the functions $v^{\rho,l}_i=v^{\rho,l}_i+\sum_{k=1}^{\infty}\delta v^{\rho,l,k}_i$ satisfy  
a local form of the equation in (\ref{nsgenint}) in a classical sense. Higher regularity of local solutions can be obtained then considering equations for the derivatives. This can be shown also directly by deriving equations for $v^{\rho,l,k}_i$ plugging this into the equation system in (\ref{nsgenint}), and estimating the deficit on the right side
by an expression in terms of functional increments which then go to zero as the local iteration index goes to infinity.

\section{Statement of local contraction result} 
It is essential to prove local contraction results with respect to the local norms
\begin{equation}
{\big |}f{\big |}^l_{C^{0}\left((l-1,l),H^{m}\right) }:=\sup_{\tau\in (l-1,l)}\sum_{|\alpha|\leq m}{\Big |}D^{\alpha}_xf(\tau,.){\Big |}_{L^2\left( {\mathbb R}^n \right) }
\end{equation}
for some $m\geq 2$. In the case of a generalized model we need to state the local contraction results with respect to the higher order correction terms, i.e., the terms $\delta v^{r,\rho,l,k}_i$ for $k\geq 2$. For the first order increment $\delta v^{\rho,l,1}_i$ we may loose some order of polynomial decay in the estimate due to natural estimates of the H\"{o}rmander density. We emphasize that we consider here the indirect approach: at each time step $l\geq 1$ we assume that the controlled functions $v^{r,\rho,l-1}_{i}$ are determined (hence especially the initial data $v^{r,\rho,l-1}_{i}(l-1,.)$ at time step $l\geq 1$ of our scheme, and we determine a local solution 
\begin{equation}
v^{\rho,l}_i=v^{r,\rho,l-1}_i(l-1,.)+\sum_{k=1}^{\infty}\delta v^{\rho,l,k}_i,
\end{equation}
where we have a contraction result for the higher order terms $\delta v^{\rho,l,k}_i$ for $k\geq 2$ and $1\leq i\leq n$. In the indirect approach we determine a the local solution of the incompressible Navier Stokes equation starting with controlled function data $v^{r,\rho,l-1}_i(l-1,.)$ but without further involvement of the control function, i.e., involvement of the control function $r^l_i(\tau,.)$ for $\tau>l-1$ in the first substep, i.e., we determine
\begin{equation}
v^{\rho,l}_i:=v^{\rho,l,1}_i+\sum_{j=2}^{\infty}\delta v^{\rho,l,k}_i
\end{equation}
where $\delta v^{\rho,l,1}_i$ solves (\ref{scalparasystlin10v23}) and
the functional increments $\delta v^{\rho,k+1,l}_i=v^{\rho,k+1,l}_i-v^{\rho,k,l}_i,~1\leq i\leq n$ then solve the equation (\ref{deltaurhok023}).
Then in this scheme we define
\begin{equation}
v^{r,\rho,l,1}_i:=v^{\rho,l,1}_i+\delta r^l_i,
\end{equation}
and in general for the approximation of order $k\geq 2$
\begin{equation}
v^{r,\rho,l,k}_i:=v^{\rho,l,k}_i+\delta r^l_i,
\end{equation}
as we have $r^{l-1}_i(l-1,.)$ in the definition of $v^{\rho,l,1}_i$ (via equation  (\ref{scalparasystlin10v23})), and
where the increment $\delta r^l_i=r^{l}_i-r^{l-1}_i$ is chosen at each time step such that the control function and the controlled value function have at most linear growth with respect to the time step number in transformed time coordinates. For the direct approach mentioned in the introduction we would have to establish local contraction results for controlled functions $v^{r,\rho,l,k}_i$, where the equation for $\delta v^{r,\rho,l,k}_i$ involves a relation of the control function and the subiteration index $k\geq 1$. This is much more cumbersome (although possible).
Note that we choose the control functions in general such that the increment of the first substep $\delta v^{r,\rho,l,1}_i$ is cancelled. The indirect construction mentioned allows us to establish a contraction result without referring to a control function and then use this contraction result in the controlled scheme. 
We need some assumption on the initial data. At time step $l\geq 1$ and for a given order of the norm $m\geq 2$ we assume that from the previous time step $l-1$ there is a constant $C^{l-1}$ such that 
\begin{equation}
\sum_{|\alpha|\leq m}\sup_{x\in {\mathbb R}^n}{\big |}D^{\alpha}_xv^{r,\rho,l-1}_i(l-1,x){\big |}\leq \frac{C^{l-1}}{1+|x|^{q}},
\end{equation}
i.e., the function $D^{\alpha}_x v^{r,\rho,l-1}_i(l-1,.)$ satisfy polynomial decay for $0\leq |\alpha|\leq m$ of order $q=3m_{0,\alpha,0}+2n+2$, where the former integer $m_{0,\alpha,0}$ is from the statement of the Kusuoka-Stroock a priori estimate above. 
\begin{rem}
For some schemes in our list, notably for the scheme in ii) the constant $C^l$ depends on the time step number $l\geq 1$, i.e., we have the inductive assumption
\begin{equation}\label{initialsizel}
|v^{r,\rho,l-1}_i(l-1,.)|^2_{H^2}\leq C^{l-1}=C+(l-1)C
\end{equation}
for some $C>0$. We shall observe that for the scheme iii) in our list $C^l$ can be chosen independent of the time step number $l\geq 1$.
\end{rem}
The strong assumption of polynomial decay and the inheritance of polynomial decay of the schemes observed in the next section simplify the reasoning for local contraction. At each local iteration step $k\geq 1$ at time step $l\geq 1$ we use upper bounds for classical representations of $\delta v^{\rho,l,k+1}_i$ which are convolutions. These convolutions may then be estimated by a Young inequality of the form
\begin{equation}
{\big |}f\ast g{\big |}_{L^r}
\leq {\big |}f{\big |}_{L^p}{\big |}g{\big |}_{L^q},
\end{equation}
where $1+r^{-1}=p^{-1}+q^{-1}$ for some $1\leq p,q,r\leq \infty$. We shall use inductive information about polynomial decay of the value function at the previous time step, i.e., information as in (\ref{contractionuseinf}) below. Some terms in the representation of $\delta v^{\rho,l,k+1}$ appear also in the associated multivariate Burgers equation. Next to estimates for the H\"{o}rmander density these terms are naturally estmated using the constant 
\begin{equation}\label{CB}
C^m_B\sim \sum_{|\alpha|\leq m}\max_{i\in \left\lbrace 1,\cdots ,n\right\rbrace }\sup_{y\in {\mathbb R}^n} {\big |}D^{\alpha}_yB_{i}(y){\big |}.
\end{equation}
Similarly, we define
\begin{equation}\label{CB}
C^m_{ij}\sim \sum_{|\alpha|\leq m}\max_{i\in \left\lbrace 1,\cdots ,n\right\rbrace }\sup_{y\in {\mathbb R}^n} {\big |}D^{\alpha}_yc_{ij}(y){\big |}.
\end{equation}
The Leray projection terms are a little more complicated. The upper bound we use are double convolutions. We may use local $L^1$-estimates for an upper bound of a truncated  H\"{o}rmander density, and this leads to the requirement of $L^2$-estimates of the convolution involving (first order derivatives) of the Laplacian kernel or its natural generalisation and products of local approximating value functions and functional increments. Upper bounds of the local H\"{o}rmander density $G^l_H$ leave us with $L^2$ estimates of convolutions involving first order derivatives of (generalised) Laplacian kernels and products of approximating value functions. As we have polynomial decay of the latter via inheritance of polynomial decay and the inductive assumption of polynomial decay, it is natural to estimate this 'inner' convolution by a combination of Young inequalities and weighted product estimates of $L^2$ norms. We may use estimates of the form where
for $s>\frac{n}{2}$ we have a constant $C_s>0$  such that for all $x\in {\mathbb R}^n$ the function
\begin{equation}
u(x):=\int_{{\mathbb R}^n}(1+|y|^2)^{-s/2}v(x-y)w(y)dy
\end{equation}
with functions $v,w\in L^2$ satisfies
\begin{equation}
|u|_{L^2}\leq C_s|v|_{L^2}|w|_{L^2}.
\end{equation}
Without loss of generality we may assume that $C_s\geq 1$. For $s=\frac{n}{2}+1$ this leads to the natural constant
\begin{equation}\label{CK}
C_K\sim\max_{i\in {1,\cdots,n}}\int_{{\mathbb R}^n}{\big |}K_{n,i}(.-y){\big |}\frac{1}{1+|y|^{n}}dy.
\end{equation}
The proportionality in $C_K$ is a finite constant dependent on dimension, the maximal order $m$ of derivatives considered.
As we indicated $L^1$-upper bounds of the density $G^l_H$ and its first order spatial derivatives are  related to another estimation constant  $C_G$, i.e.,
\begin{equation}\label{CG}
\begin{array}{ll}
C_G\mbox{ is related to }{\big |}G^l_{H,i}{\big |}_{L^1\times H^1}.
\end{array}
\end{equation}
Upper bounds of the right side in case of the local Gaussian are well known. For the H\"{o}rmander density consider our discussion below and in \cite{KH}. 
We have the following local contraction result.
\begin{thm}\label{mainthm1}
Let $n\geq 3$. Assume that  for $1\leq i\leq n$ and $m\geq 2$ and multiindices $\alpha$ with $|\alpha|\leq 2$ we have
such that for all $x\in {\mathbb R}^n$
\begin{equation}\label{contractionuseinf}
{\big |}D^{\alpha}_xv^{\rho,l-1}_i(l-1,x){\big |}\leq \frac{C^{l-1}}{1+|x|^{q}}
\end{equation}
for
\begin{equation}
q\geq \max_{|\alpha|\leq m}\left\lbrace 2,m_{j,\alpha,0}\right\rbrace +2n+2.
\end{equation}
Then we have local contraction results with respect to the $C^0\times H^{2m}$-norm \begin{equation}\label{contractionconstant}
\rho_l\leq \frac{1}{c(n)\left(  \left( 2C^m_BC_G+C_K\sum_{j,p=1}^nC^m_{jp}\right) 2\left( C^{l-1}+1\right)\right) }
\end{equation}
(along with  $C_G,C_K$ and $C_s$ defined above) for $k\geq 2$ we have 
\begin{equation}\label{contract}
\begin{array}{ll}
\max_{i\in \left\lbrace 1,\cdots ,n\right\rbrace }|\delta v^{\rho,l,k}_i|_{C^0\left( (l-1,l), H^m\right) }
\leq \frac{1}{2\sqrt{l}}\max_{i\in \left\lbrace 1,\cdots ,n\right\rbrace }|\delta v^{\rho,l,k-1}_i|_{C^0\left((l-1,l),H^m\right) },
\end{array} 
\end{equation}
and for $k=1$ and $\rho_l$ small enough we have
\begin{equation}\label{delta1remark}
\begin{array}{ll}
\max_{i\in \left\lbrace 1,\cdots ,n\right\rbrace }|\delta v^{\rho,l,1}_i|_{C^0\left((l-1,l),H^m\right) }\\
\\
=\max_{i\in \left\lbrace 1,\cdots ,n\right\rbrace }|v^{\rho,l,1}-v^{\rho,l-1}(l-1,.)|_{C^0\left((l-1,l),H^m\right) }\leq \frac{1}{4}.
\end{array}
\end{equation}
If
\begin{equation}
q\geq \max_{j\leq m,|\alpha|\leq 2m}\left\lbrace n_{j,\alpha},3m_{j,\alpha,0}\right\rbrace +2n+2.
\end{equation} then an analogous contraction result with respect to the $|.|_{H^{m,\infty}\times H^{2m}}$ norm holds, and with a time step size $\rho_l$ proportional to (\ref{contractionconstant}) holds, where the proportional constant depends only on the dimension, the order $2m$, and an additional constant related to estimation of products of functions by their factors in Sobolev spaces.
\end{thm}
\begin{rem}
 Note that the right side of (\ref{delta1remark}) is not zero even if we choose $r^l_i$ as in ii) (cf. our remark above that the meaning of the control function superscript depends on the time step number $l$).
\end{rem}

\section{Inheritance of polynomial decay for the higher order correction terms in the local scheme  } 
At each time step $l\geq 1$ having determined $v^{r,\rho,l-1}_i(l-1,.)$ we have to determine the increment
\begin{equation}
\delta v^{r,\rho,l}_i=\delta v^{\rho,l}_i+\delta r^l_i.
\end{equation}
This involves the increment $\delta v^{\rho,l}_i$ and the increment $\delta r^l_i$. The former is constructed by a local scheme and can be determined independently of the increment $\delta r^l_i$. The control function is designed in order to control the global growth properties of the scheme.  
Given $v^{r,\rho,l-1}_i(l-1,.),~1\leq i\leq n$ at time step $l-1$ the local solution function $v^{\rho,l}_i,1\leq i\leq n$ is constructed via the functional series
\begin{equation}\label{localfunc}
v^{\rho,l}_i=v^{r,\rho,l-1}_i(l-1,.)+\delta v^{\rho,l,1}_i+\sum_{k\geq 2}\delta v^{\rho,l,k}_i
\end{equation}
for $1\leq i\leq n$. Note the appearance of the control function in the first summand of this local series. The series is a controlled series but we suppressed the dependence on the control in order to keep the notation simple. The reason is that the dependence on the control function at time step $l$ concerns only the initial data $r^{l-1}_i(l-1,.)$ at that time step such that the structure of the local equation is exactly the same as the structure of the uncontrolled equation - just the data are different. The disadvantage is that we have a notation which equals the notation for local uncontrolled functional series, but having remarked this there should be no confusion.  We call the terms of the last sum in (\ref{localfunc}), i.e., the terms $v^{\rho,l,k}_i,k\geq 2$ the higher order correction terms, and for these terms we have inheritance of polynomial decay if the conditions of theorem \ref{mainthm1} are satisfied. These higher order terms satisfy the equation in (\ref{deltaurhok0}), which is identical to the equation of increments for higher order approximations of the local uncontrolled Navier Stokes equation. These terms depend only on the control function data $r^{l-1}_i(l-1,.)$, which appear in the equation for $v^{\rho,l,1}_i$ which solves the equation in (\ref{scalparasystlin10v}). This way we can avoid a more cumbersome analysis which involves the more complicated equations for the controlled value functions stated in the introduction (the analysis is analogous but there are a lot more terms with factor $\rho_l$ which have to be treated then). Here we take advantage of the fact that we choose a control function once at each time step $l\geq 1$, and solve the for the increment of the controlled value function independently of the increment of the control function (but not independently of the control function data $r^{l-1}_i(l-1,.)$ at time step $l\geq 1$. If $G^l_H$ denotes the fundamental solution of the equation
\begin{equation}
\frac{\partial G^l_H}{\partial \tau}-\rho_l\frac{1}{2}\sum_{j=0}^mV_{j}^2G^l_H=0
\end{equation}
on the domain $[l-1,l]\times {\mathbb R}^n$, then
 \begin{equation}\label{scalparasystlin10vher}
\begin{array}{ll}
v^{\rho,l,1}_i(\tau,x)=\int_{{\mathbb R}^n}v^{r,\rho,l-1}_i(l-1,y)G^l_H(\tau,x;s,y)dy\\
\\ 
-\rho_l\int_{l-1}^{\tau}\int_{{\mathbb R}^n}V_{B}\left[v^{r,\rho,l-1}(l-1,.)\right] v^{r,\rho,l-1}_i(l-1,y)G^l_H(\tau,x;s,y)dyds+\\
\\
\rho_l\int_{l-1}^{\tau}\int_{{\mathbb R}^n}\int_{{\mathbb R}^n}\sum_{j,m=1}^n \left(c_{jm} \frac{\partial v^{r,\rho,l-1}_m}{\partial x_j}(l-1,.)\frac{\partial v^{r,\rho,l-1}_j}{\partial x_m}(l-1,.)\right) (s,y)\times\\
\\
\times\frac{\partial}{\partial x_i}K^{\mbox{ell}}_n(z-y)G^l_H(\tau,x;s,z)dydzds,
\end{array}
\end{equation} 
and
\begin{equation}\label{deltaurhok0her}
\begin{array}{ll}
\delta v^{\rho,l,k+1}_i(\tau,x)=\\
\\ 
-\rho_l\int_{l-1}^{\tau}\int_{{\mathbb R}^n}\left( V_{B}\left[v^{\rho,l,k}\right] \delta v^{\rho,l,k}_i+V_{B}\left[\delta v^{\rho,l,k}\right] v^{\rho,l,k}_i\right) (s,y)\times\\
\\
\times G^l_H(\tau,x;s,y)dyds+\rho_l\int_{l-1}^{\tau}\int_{{\mathbb R}^n}\int_{{\mathbb R}^n}K^{\mbox{ell}}_{n,i}(z-y)\times\\
\\
{\Big (} c_{jm}\left( \sum_{j,m=1}^n\left( v^{\rho,l,k}_{m,j}+v^{\rho,l,k-1}_{m,j}\right)(s,y) \right)  \delta v^{\rho,l,k}_{j,m}(s,y) {\Big)}\times\\
\\
\times G^l_H(\tau,x;s,z)dydzds.
\end{array}
\end{equation}
The representation in (\ref{scalparasystlin10vher}) and the a priori estimates for $G_H$ show  that we may loose some order of polynomial decay at each time step for the uncontrolled scheme due to the first term in (\ref{scalparasystlin10vher}). On the other hand, the representation in (\ref{deltaurhok0her}) involves products of (spatial derivatives of) value functions with (spatial derivatives) of functional increments which both have a  polynomial decay of a certain order. Hence products have a higher order of polynomial decay which can compensate the polynomial growth factors of the densities we observe in the standard estimates. This is one motivation for the introduction of a control function $r^l_i=r^{l-1}_i(l-1,.)+\delta r^l_i$ along with $\delta r^l_i=-\delta v^{\rho,l,1}_i$ (our most simple choice of a control function), where we have
\begin{equation}
\begin{array}{ll}
v^{r,\rho,l}_i=v^{r,\rho,l-1}_i(l-1,.)+\delta v^{\rho,l,1}_i+\delta r^l_i+\sum_{k\geq 2}\delta v^{\rho,l,k}_i\\
\\
=v^{r,\rho,l}_i=v^{r,\rho,l-1}_i(l-1,.)+\sum_{k\geq 2}\delta v^{\rho,l,k}_i.
\end{array}
\end{equation}
Well the representation in (\ref{scalparasystlin10vher}) shows that inheritance polynomial decay is also preserved if we choose the simplified control function of iiia), i.e. the function
\begin{equation}\label{possiii}
\begin{array}{ll}
 \delta r^l_i(\tau,x)=-\int_{{\mathbb R}^n}v^{r,\rho,l-1}_i(l-1,y)G_l(\tau,x;l-1,y)dy\\
 \\
 + v^{r,\rho,l-1}_i(\tau,x).
 \end{array}
 \end{equation}
Anyway, for such types of controlled schemes we have preservation of polynomial decay of the controlled scheme if we have preservation of polynomial decay for the higher order correction terms. 
In this context (cf. \cite{KB3}) we say that
\begin{equation}
\begin{array}{ll}
v^{\rho,l,k}_i \mbox{ is of polynomial decay of order $m\geq 2$}\\
\\
\mbox{ for derivatives up to order $p\geq 0$ if for some finite  $C>0$}\\
\\
\sum_{|\alpha|\leq p}\sup_{\tau\in [l-1,l]}|D^{\alpha}_x v^{\rho,l,k}_i(\tau,y)|\leq \frac{C}{1+|y|^m}.
\end{array}
\end{equation}
Similarly for the functional increments $\delta v^{\rho,l,k}_i$.
The spaces of functions of polynomial decay of order $m\geq 2$ form an algebra. Especially, if $v^{\rho,l,k-1}_i,v^{\rho,l,k}_i$ are of polynomial decay of order $m\geq 2$ for derivatives up to order $p\geq 0$, then we have that functional increments $\delta v^{\rho,l,k}_i$ are of polynomial decay of order $m\geq 2$ and for derivatives up to order $p\geq 0$ . Moreover products of such functions have polynomial decay of order $2m$ for derivatives up to order $p$.
These considerations motivate the following definition (which we take from \cite{KB3} essentially).
\begin{defi}
Assume that for all $1\leq i\leq n$ and $l-1\geq 0$ the functions $v^{\rho,l-1,1}_i$ have polynomial decay of some order $m$ (which is a positive integer) for derivatives up to order $p$. We say that polynomial decay of order $m$ for derivatives up to order $p\geq 0$ is inherited by a controlled scheme (of type iii) or iiia) as described above) for the higher order correction terms $\delta v^{\rho,l,k}_i,~k\geq 2$, if for all $1\leq i\leq n$ these higher order terms have polynomial decay of order $m$ for derivatives up to order $p$.
\end{defi}
Next we prove inheritance of polynomial decay for the generalized controlled scheme.
Since we are interested in polynomial decay with respect to the spatial variables we use the standard a priori estimate of the density in (\ref{pxest}) for $j=0$ and $\beta=0$ and $\alpha\geq 0$, i.e., we use the estimate
\begin{equation}\label{pxest*}
\begin{array}{ll}
{\Bigg |}\frac{\partial^{|\alpha|}}{\partial x^{\alpha}} p(\tau,x,y){\Bigg |}\leq \frac{A_{0,\alpha,0}(t)(1+x)^{m_{0,\alpha,0}}}{t^{n_{0,\alpha,0}}}\exp\left(-B_{0,\alpha,0}(\tau)\frac{(x-y)^2}{\tau}\right).
\end{array}
\end{equation}
However, there is an additional difficulty here for the generalized scheme compared to the simple scheme with constant viscosity (and even compared to a scheme with operators with strictly elliptic spatial part). This additional difficulty consists in the polynomial growth factor 
\begin{equation}
(1+x)^{m_{0,\alpha,0}}
\end{equation}
in (\ref{pxest*}) which does not appear in the a priori estimates for operators with strictly elliptic spatial part.
 We have
\begin{lem}
Polynomial decay of order $q$ with
 \begin{equation}
q\geq \max_{|\alpha|\leq p}\left\lbrace n_{0,\alpha,0},m_{0,\alpha,0}\right\rbrace +n+1
\end{equation}
for derivatives up to order $p\geq 0$ is inherited by the higher order correction terms.
\end{lem}

\begin{proof}
Consider the representation of the higher order correction term $\delta v^{\rho,l,k+1}_i$
in (\ref{deltaurhok0her}). Since $G^l_H$ is a density for $\alpha=0$ we know that the representation
\begin{equation}\label{deltaurhok0her1}
\begin{array}{ll}
D^{\alpha}_x\delta v^{\rho,l,k+1}_i(\tau,x)=\\
\\ 
-\rho_l\int_{l-1}^{\tau}\int_{{\mathbb R}^n}\left( V_{B}\left[v^{\rho,l,k}\right] \delta v^{\rho,l,k}_i+V_{B}\left[\delta v^{\rho,l,k}\right] v^{\rho,l,k}_i\right) (s,y)\times\\
\\
\times D^{\alpha}_xG_H(\tau,x;s,y)dyds+\rho_l\int_{l-1}^{\tau}\int_{{\mathbb R}^n}\int_{{\mathbb R}^n}K^{\mbox{ell}}_{n,i}(z-y)\times\\
\\
{\Big (} c_{jm}\left( \sum_{j,m=1}^n\left( v^{\rho,l,k}_{m,j}+v^{\rho,l,k-1}_{m,j}\right)(s,y) \right)  \delta v^{\rho,l,k}_{j,m}(s,y) {\Big)}\times\\
\\
\times D^{\alpha}_xG_H(\tau,x;s,z)dydzds.
\end{array}
\end{equation}
holds. For $|\alpha|>0$ the representation can be justified by the fact that for each $x\in {\mathbb R}^n$ there exists $\epsilon >0$ and a ball $B_{\epsilon}(x)$ of radius $\epsilon$ around $x$ such that we get an integrable weakly singular upper bound. We then get an upper bound for ${\big |}D^{\alpha}_x\delta v^{\rho,l,k+1}_i(\tau,x){\big |}$ by the upper bounds of the modulus of the two summands
\begin{equation}\label{deltaurhok0her2}
\begin{array}{ll}
D^{\alpha}_x\delta v^{\rho,l,k+1}_i(\tau,x)=\\
\\ 
-\rho_l\int_{l-1}^{\tau}\int_{{\mathbb R}^n\setminus B_{\epsilon}(x)}\left( V_{B}\left[v^{\rho,l,k}\right] \delta v^{\rho,l,k}_i+V_{B}\left[\delta v^{\rho,l,k}\right] v^{\rho,l,k}_i\right) (s,y)\times\\
\\
\times D^{\alpha}_xG_H(\tau,x;s,y)dyds+\rho_l\int_{l-1}^{\tau}\int_{{\mathbb R}^n\setminus B_{\epsilon}(x)}\int_{{\mathbb R}^n}K^{\mbox{ell}}_{n,i}(z-y)\times\\
\\
{\Big (} c_{jm}\left( \sum_{j,m=1}^n\left( v^{\rho,l,k}_{m,j}+v^{\rho,l,k-1}_{m,j}\right)(\tau,y) \right)  \delta v^{\rho,l,k}_{j,m}(s,y) {\Big)}\times\\
\\
\times D^{\alpha}_xG_H(\tau,x;s,z)dydzds,
\end{array}
\end{equation}
and
\begin{equation}\label{deltaurhok0her3}
\begin{array}{ll}
D^{\alpha}_x\delta v^{\rho,l,k+1}_i(\tau,x)=\\
\\ 
-\rho_l\int_{l-1}^{\tau}\int_{B_{\epsilon}(x)}\left( V_{B}\left[v^{\rho,l,k}\right] \delta v^{\rho,l,k}_i+V_{B}\left[\delta v^{\rho,l,k}\right] v^{\rho,l,k}_i\right) (s,y)\times\\
\\
\times D^{\alpha}_xG_H(\tau,x;s,y)dyds+\rho_l\int_{l-1}^{\tau}\int_{B_{\epsilon}(x)}\int_{{\mathbb R}^n}K^{\mbox{ell}}_{n,i}(z-y)\times\\
\\
{\Big (} c_{jm}\left( \sum_{j,m=1}^n\left( v^{\rho,l,k}_{m,j}+v^{\rho,l,k-1}_{m,j}\right)(s,y) \right)  \delta v^{\rho,l,k}_{j,m}(s,y) {\Big)}\times\\
\\
\times D^{\alpha}_xG_H(\tau,x;s,z)dydzds.
\end{array}
\end{equation}
An upper bound for the first term is
\begin{equation}\label{deltaurhok0her4}
\begin{array}{ll}
{\big |}D^{\alpha}_x\delta v^{\rho,l,k+1}_i(\tau,x){\big |}=\\
\\ 
\rho_l\int_{l-1}^{\tau}\int_{{\mathbb R}^n\setminus B_{\epsilon}(x)}{\big |}\left( V_{B}\left[v^{\rho,l,k}\right] \delta v^{\rho,l,k}_i+V_{B}\left[\delta v^{\rho,l,k}\right] v^{\rho,l,k}_i\right) (s,y){\big |}\times\\
\\
\times {\Big |}\frac{A_{0,\alpha,0}(t)(1+x)^{m_{0,\alpha,0}}}{t^{n_{0,\alpha,0}}}\exp\left(-B_{0,\alpha,0}(\tau)\frac{(x-y)^2}{\tau}\right){\Big |}dyds\\
\\
+\rho_l\int_{l-1}^{\tau}\int_{{\mathbb R}^n\setminus B_{\epsilon}(x)}\int_{{\mathbb R}^n}{\big |}K^{\mbox{ell}}_{n,i}(z-y){\big |}\times\\
\\
\times {\Big |} c_{jm}\left( \sum_{j,m=1}^n\left( v^{\rho,l,k}_{m,j}+v^{\rho,l,k-1}_{m,j}\right)(s,y) \right)  \delta v^{\rho,l,k}_{j,m}(s,y) {\Big|}\times\\
\\
\times {\Big |}\frac{A_{0,\alpha,0}(t)(1+x)^{m_{0,\alpha,0}}}{t^{n_{0,\alpha,0}}}\exp\left(-B_{0,\alpha,0}(\tau)\frac{(x-z)^2}{\tau}\right){\Big |}dydzds,
\end{array}
\end{equation}
The products of type $\left( V_{B}\left[v^{\rho,l,k}\right] \delta v^{\rho,l,k}_i+V_{B}\left[\delta v^{\rho,l,k}\right] v^{\rho,l,k}_i\right) (s,y)$ have polynomial decay of order $2q$, and this implies that we have polynomial decay of order $q$ for the (double) convolutions involved. Here we use the 'ellipticity' assumption concerning the kernel $K^{\mbox{ell}}$ (which implies that we loose at most one order of polynomial decay). Our assumption that implies that via a very rough estimate we still have polynomial decay of order $2q-m_{0,\alpha,0}-1-n\geq q$ for the term in (\ref{deltaurhok0her4}). For the local integrals in (\ref{deltaurhok0her3}) we get the same conclusion.
\end{proof}

\section{Global regularity and growth behavior of the control function}
In this section we give several arguments for at most linear growth of several type of control functions where we argue with respect to function spaces and with respect to preservation of certain order of polynomial decay by the scheme. We also provide arguments for linear upper bounds without the local adjoint of H\"{o}rmander densities. 
Since we can prove local contraction results with respect to norms
\begin{equation}\label{normsas}
|.|_{C^0\left((l-1,l)\times H^m\left( {\mathbb R}^n \right)\right)} ,~|.|_{C^1\left((l-1,l)\times H^m\left( {\mathbb R}^n \right) \right) }
\end{equation}
for the controlled velocity functions we can measure the growth of these contolled velocity functions with respect to these norms. More precisely, if we have an inductive upper bound for $D^{\alpha}_xv^{r,*,\rho,l-1}_i(l-1,.)$ for multivariate spatial derivatives of order $|\alpha|\geq 0$ of the form, let's say
\begin{equation}
|D^{\alpha}_xv^{r,*,\rho,l-1}_i(l-1,.)|_{C^0\left((l-1,l)\times H^m\left( {\mathbb R}^n \right)\right)  }\leq C,~\mbox{for}~0\leq |\alpha|\leq m, 
\end{equation}
then a small step size $\rho_l>0$ and a choice of the control function increment
as in (\ref{controlincrementobservation}) below ensure that this upper bound for the controlled velocity function is preserved form time step $l-1$ to time step $l$. A global scheme with such an upper bound can be established for all time steps if we can prove a linear upper bound for the control function $r^l_i$ (linear with respect to the time step size). Then we may choose a step size $\rho_l\sim \frac{1}{l}$ and offset the growth of the control functions. Linear growth of the control function means constant growth (at most) at each time step.
Consider first a simple control function with the increment
\begin{equation}\label{controlincrementobservation}
\delta r^l_i(l,x)=\int_{l-1}^l\int_{{\mathbb R}^ n}\frac{-v^{r,*,\rho,l-1}_i(l-1,y)}{C}G_l(l,x;s,y)dyds.
\end{equation}
We can ensure (at most) constant growth with respect to norms as in (\ref{normsas}) by the estimation of convolutions with the Gaussian by the generalized Young inequality
\begin{equation}
|f\ast g|_r\leq |f|_p|g|_q,~\frac{1}{p}+\frac{1}{q}=1+\frac{1}{r},~1\leq p,q,r\leq 1.
\end{equation}
This is clear for the spatial part of the convolution, and for the time part we may use similar observations as in the proof of the local contraction result below. Similarly for spatial derivatives, where we may use the arguments which we considered in the proof of the local contraction results (including shifting derivatives). 

Can this be generalized to densites which occur in the solution scheme of highly degenerate Navier Stokes equations? For the control function increments themselves this holds since the density $G^l_H$ of a diffusion which satisfies the H\"{o}rmander condition is in $L^1$ for positive time (-difference), and the weakly singular behavior at time near zero changes only for derivatives. Similarly, natural derivative estimates of the density $G^l_H$ have an additional spatial factor of polynomial type, such that we may loose integrability properties. We have to avoid at least higher order (order $\geq 2$) derivatives on the local and any order of derivatives  (order $\geq 1$) on the global level. Well, this can be done using local adjoints in order to shift derivatives. The spatial regularity of the controlled velocity function from the previous time step depends on the dimension. This regularity can be used in order to get upper bounds for the control function increments. We recall this phenomenon first in the context of the classical Navier Stokes equation problem with constant positive viscosity. However, we want to add another twist to the discussion which concerns polynomial decay. We discuss it in the context of the classical Navier Stokes equation model first. The generalisation of the discussion will lead us to a more in-depth discussion of some properties of densities which satisfy the H\"{o}rmander condition. Although the natural upper bounds of these densities show that we may {\it not} use derivatives of these densities in order to estimate the growth behavior of the control functions, we shall see that it is sufficient to have some estimate of the local density itself in order to establish certain upper bounds of the growth control of the control functions. But let is turn to the classical model first and recall some thoughts of \cite{KB3}.

A natural question which may be posed is: why does the definition in (\ref{controlincrementobservation}) not reduce the polynomial decay of the control functions, and therefore, with a delay of one time step the polynomial decay of the controlled velocity functions as well ? The reason is that the regularity of the integrand can be used here. We cite and extend here the latest discussion in \cite{KB3}, and then show that these considerations can be extended to the highly degenerated systems. If the controlled velocity function $v^{r,*,\rho,l-1}_i(l-1,.)$ is smooth, i.e., if $v^{r,*,\rho,l-1}_i(l-1,.)\in C^{\infty}$, then this holds for the increment (\ref{controlincrementobservation}) as well. In the classical model where we have constant viscosity this follows easily from the fact that 
(\ref{controlincrementobservation}) is a convolution. For an arbitrary multivariate derivative $D^{\alpha}_x$ of order $|\alpha|\geq 0$ we may use the smoothness of the controlled velocity function and write
\begin{equation}\label{convobservation}
\begin{array}{ll}
D^{\alpha}_x\delta r^l_i(l,x)=\int_{l-1}^l\int_{{\mathbb R}^ n}\frac{-v^{r,*,\rho,l-1}_i(l-1,y)}{C}D^{\alpha}_xG_l(l,x;s,y)dyds\\
\\
=\int_{l-1}^l\int_{{\mathbb R}^ n}D^{\alpha}_y\frac{-v^{r,*,\rho,l-1}_i(l-1,y)}{C}G_l(l,x;s,y)dyds,
\end{array}
\end{equation}
which shows that the controlled function increment $\delta r^l_i(l,.)$ is smooth if the controlled velocity function $v^{r,*,\rho,l}_i(l-1,.)$ of the previous time step $l-1$ is smooth. Moreover, if the control function $r^{l-1}_i(l-1,.)$ is smooth it follows that the control function $r^l_i(l,.)$, the controlled velocity function $v^{r,*,\rho,l}_i$ and the original velocity function $v^{*,\rho,l}_i(l,.)$ inherit smoothness from the previous time step. This phenomenon has some consequences for the growth behavior of the simple control function considered so far. We have already observed that convolutions of the type considered in (\ref{convobservation}) may be estimated by breaking up the integral into a local part around $x$ and a complementary part. For a ball $B_{\epsilon}(x)$ of radius $\epsilon>0$ around $x$  we may write for $0<\mu<1$
\begin{equation}\label{convobservation2}
\begin{array}{ll}
{\Big |}D^{\alpha}_x\delta r^l_i(l,x){\Big |}\leq\\
\\
{\Big |}\int_{l-1}^l\int_{B_{\epsilon}(x)}D^{\alpha}_y\frac{-v^{r,*,\rho,l-1}_i(l-1,y)}{C}\frac{C}{(l-s)^{\mu}|x-y|^{n-2\mu}}dyds{\Big |}\\
\\
+{\Big |}\int_{l-1}^l\int_{{\mathbb R}^ n\setminus B_{\epsilon}(x)}D^{\alpha}_y\frac{-v^{r,*,\rho,l-1}_i(l-1,y)}{C}G_l(l,x;s,y)dyds{\Big |}.
\end{array}
\end{equation}
The second integral on the right side of (\ref{convobservation2}) is over a domain where $G_l$ is analytic. Moreover, polynomial decay any order $p>0$ of this term is inherited from $D^{\alpha}_y\frac{-v^{r,*,\rho,l-1}_i(l-1,.)}{C}$ as we may use the convolution rule to write the second integral on the right side (\ref{convobservation2}) for $|x|\neq 0$ in the form
\begin{equation}\label{convobservation2}
\begin{array}{ll}
{\Big |}\int_{l-1}^l\int_{{\mathbb R}^ n\setminus B_{\epsilon}(x)}D^{\alpha}_y\frac{-v^{r,*,\rho,l-1}_i(l-1,x-y)}{C}G_l(l,y;s,0)dyds{\Big |}\\
\\
\leq {\Big |} \int_{l-1}^l\int_{{\mathbb R}^ n\setminus B_{\epsilon}(x)}\frac{c}{1+|x-y|^p}G_l(l,y;s,0)dyds{\Big |}\\
\\
\leq {\Big |} \int_{l-1}^l\int_{{\mathbb R}^ n\setminus B_{\epsilon}(x)}\frac{c}{1+|x-y|^p}\frac{\tilde{c}}{1+|y|^n}dyds{\Big |}\leq\frac{\tilde{C}}{|x|^p}
\end{array}
\end{equation}
The first integral on the right side of (\ref{convobservation2}) may be estimated in polar coordinates $(r,\phi_1,\cdots,\phi_n)$ using iterated partial integration. For $\mu\in \left(\frac{1}{2},1\right)$ we need only $n-1$ partial integrations to get the upper bound   
\begin{equation}
\begin{array}{ll}
{\Big |}\int_{l-1}^l\int_{B_{\epsilon}(x)}D^{\alpha}_y\frac{-v^{r,*,\rho,l-1}_i(l-1,x-y)}{C}\frac{C}{(l-s)^{\mu}|y|^{n-2\mu}}dyds{\Big |}\\
\\
\leq (n-1)C_{\epsilon}+\\
\\
{\Big |}\int_{l-1}^l\int_{B_{\epsilon}(x)}D^{n-1}_r\left( D^{\alpha}_y\frac{-v^{r,*,\rho,l-1}_i(l-1,x-y)}{C}\right)_{\mbox{pol}} \frac{C}{(l-s)^{\mu}}|y|^{2\mu-1}dyds{\Big |},
\end{array}
\end{equation}
where $C_{\epsilon}$ is an upper bound for boundary terms which are integrals over the sphere $S^n_{\epsilon}(x)$, i.e., the boundary of $B_{\epsilon}(x)$. Here we understand that
\begin{equation}
\left( D^{\alpha}_y\frac{-v^{r,*,\rho,l-1}_i(l-1,x-y)}{C}\right)_{\mbox{pol}}
\end{equation}
is the $D^{\alpha}_y\frac{-v^{r,*,\rho,l-1}_i(l-1,.)}{C}$ written in polar coordinates. Summing up these conservations we get for $\mu>\frac{1}{2}$ an upper bound $\sim 1$ for the control function increment as desired. Note that the surface terms mentioned become small for small $\epsilon>0$ for this choice of $\mu>\frac{1}{2}$.

Next concerning extended control functions we come to a similar conclusion for the additional term
\begin{equation}
-\delta v^{r^{l-1},*,\rho,l,1}_i(\tau,x),
\end{equation}
for reasons discussed  in the previous section, and the additional source term can be estimated as above. Let us consider the simplified control function
\begin{equation}\label{controlincrementobservation22}
\delta r^{l,\mbox{simple}}_i(l,x)=\int_{l-1}^l\int_{{\mathbb R}^ n}\frac{-v^{r,\rho,l-1}_i(l-1,y)}{C}G^l_H(l,x;s,y)dyds.
\end{equation}
We have already remarked that we may replace $G^l_H$ by $G_l$, because which type of diffusion does not matter for the control function increment. The only reason to involve a density in the definition of the simplified control (\ref{controlincrementobservation22}) is the smoothing effect. So for this part we can argue as above. However, we could also use the local adjoint and argue as in \cite{KB2}. We explain the concept of a local adjoint for the first term on the right side of an extended control function increment in
\begin{equation}\label{controlincrementobservation22}
\delta r^{l,\mbox{full}}_i(l,x)=-\delta v^{r,\rho,l,1}_i(l,x)+\int_{l-1}^l\int_{{\mathbb R}^ n}\frac{-v^{r,\rho,l-1}_i(l-1,y)}{C}G^l_H(l,x;s,y)dyds.
\end{equation}
This term is essentially of the form
\begin{equation}\label{deltav1111}
\int_{{\mathbb R}^n}v^{r,\rho,l-1}_i(l-1,y)G^l_H(l,x;l-1,y)dy.
\end{equation}
Spatial derivatives have the ad hoc representation
\begin{equation}\label{deltav2222}
\int_{{\mathbb R}^n}v^{r,\rho,l-1}_i(l-1,y)D^{\alpha}_xG^l_H(l,x;l-1,y)dy.
\end{equation} 
For the latter expression (\ref{deltav2222}) the local adjoint may be useful in order to shift derivatives using the regularity of the 'data' $v^{r,\rho,l-1}_i(l-1,.)$. 
Let us consider the former expression (\ref{deltav1111}). 
For $j=0,\beta=0$ from the main estimate expressed in time-homogeneous form in $\ref{pxest}$ we get with $t^{n_{0,\alpha,0}}=(l-(l-1))^{n_{0,\alpha,0}}=1$
\begin{equation}\label{pxest00}
\begin{array}{ll}
{\Bigg |}\frac{\partial^{|\alpha|}}{\partial x^{\alpha}} G^l_H(l,x;l-1,y){\Bigg |}
\leq A^{sup,l}_{0,\alpha,0}(1+x)^{m_{0,\alpha,0}}\exp\left(-B^{inf,l}_{0,\alpha,0}(x-y)^2\right),
\end{array}
\end{equation}
where $B^{inf,l}_{0,\alpha,0}$ is a positive lower bound of a function of time corresponding to the function $B_{j,\alpha,\beta}$ for $j=0$ and $\beta=0$ in (\ref{pxest}), and, similarly, $A^{sup,l}_{0,\alpha,0}$ is an upper bound of a function of time corresponding to function $A_{j,\alpha,\beta}$ for $j=0$ and $\beta=0$ in (\ref{pxest}). Note that both constants are determined on a different time scale in our scheme than in the standard theorem above, but this gives only another constant factor related to the time step size $\rho_l$ at time step $l\geq 1$. Assuming that we have polynomial decay of order $m\geq 2$ from the previous time step $l-1$ we then have for any $p>1$ and some generic constant $C>0$
\begin{equation}\label{deltav22223}
\begin{array}{ll}
{\big |}\int_{{\mathbb R}^n}v^{r,\rho,l-1}_i(l-1,y)D^{\alpha}_xG^l_H(l,x;l-1,y)dy{\big |}\\
\\
\leq \int_{{\mathbb R}^n}\frac{C}{1+|y|^m}A^{sup,l}_{0,\alpha,0}(1+x)^{m_{0,\alpha,0}}\exp\left(-B^{inf,l}_{0,\alpha,0}(x-y)^2\right)\\
\\
\leq \int_{{\mathbb R}^n}\frac{C}{1+|y|^m}A^{sup,l}_{0,\alpha,0}(1+x)^{m_{0,\alpha,0}}\frac{C}{1+|x-y|^p}\\
\\
\leq \frac{\tilde{C}}{1+|x|^{m+p-n-m_{0,\alpha,0}}}.
\end{array}
\end{equation}
Hence for the choice $p\geq n+m_{0,\alpha,0}$ polynomial decay 
of the the first order approximation increments $-\delta v^{r,\rho,l,1}_i,~1\leq i\leq n$ is ensured, and therefore inheritance of polynomial decay holds for this part of the control function increments too. The argument which leads to (\ref{deltav22223}) hinges on the positive time difference $l-(l-1)$ of course. For the second term in (\ref{controlincrementobservation22}) we have to deal with the singularity as time difference become small (and zero in the limit). In this case we have already remarked that we can replace the second term in (\ref{controlincrementobservation22*}) by
\begin{equation}\label{controlincrementobservation22}
\int_{l-1}^l\int_{{\mathbb R}^ n}\frac{-v^{r,\rho,l-1}_i(l-1,y)}{C}G_l(l,x;s,y)dyds,
\end{equation}
and work as in the classical case. We could even keep the kernel $G^H_l$ in (\ref{controlincrementobservation22}) if we apply local adjoints with the construction in \cite{KH} and then argue as in in the prove of the local contraction result in \cite{KB2}. 
\section{Global linear upper bound of the Leray projection term for simple controlled schemes}
In this section we consider the simple possibility ii) in the list of control functions of the introduction. We show how this choice leads to a global linear bound of the Leray projection term. The alternative simple method via the choice ia) of that list is considered in the next section, where it serves as a step for a uniform bound. 
Concerning the global linear upper bound (on a time-transformed time scale) we first reconsider the reasoning outlined in the introduction for the classical Navier Stokes equation with constant viscosity. In a second step we shall show how and with which nuances this applies to the generalized system. Assume inductively (with respect to the time step number $l\geq 1$) that we have realized an upper bound proportional to the squareroot of the time step number for some time step number $l-1\geq 0$, i.e., that we have 
\begin{equation}
D^{\alpha}_xv^{r,\rho,l-1}_i(l-1,.)\sim \sqrt{l-1} \mbox{ for }|\alpha|\leq m
\end{equation}
for some $m\geq 2$ which is fixed in advance. We  may refine the local contraction result
\begin{equation}\label{loccontrupp1}
{\big |}\delta v^{r,\rho,l,k}_i{\big |}_{C^0\left((l-1,l),H^m\right)}\leq \frac{1}{2}{\big |}\delta v^{r,\rho,l,k-1}_i{\big |}_{C^0\left((l-1,l),H^m\right) },
\end{equation}
(for all $1\leq i\leq n$) or higher order local contraction result
\begin{equation}\label{loccontrupp2}
{\big |}\delta v^{r,\rho,l,k}_i{\big |}_{C^m\left((l-1,l),H^m\right) }\leq \frac{1}{2}{\big |}\delta v^{r,\rho,l,k-1}_i{\big |}_{C^m\left((l-1,l),H^m\right) }
\end{equation}
(for all $1\leq i\leq n$) a bit. In the form (\ref{loccontrupp1}) or (\ref{loccontrupp2}) it just ensures that the local limit 
\begin{equation}\label{funciv2}
\mathbf{v}^{\rho ,l }=\mathbf{v}^{r,\rho,l-1}+\sum_{k=1}^{\infty} \delta \mathbf{v}^{\rho,l,k}
=\mathbf{v}^{\rho,l,1}+\sum_{k=2}^{\infty} \delta \mathbf{v}^{\rho,l,k}
\end{equation} 
of the corresponding local functional series represents a local solution of the incompressible Navier Stokes equation on the domain $[l-1,l]\times {\mathbb R}^n$ (if the time step size $\rho_l$ is small enough).
\begin{rem}
Note that on the left side of (\ref{funciv2}) we use the notation $\mathbf{v}^{\rho ,l }$ such that $\mathbf{v}^{r,\rho ,l }=\mathbf{v}^{\rho ,l }+\delta \mathbf{r}^l$ accepting some notational ambiguity for the sake of notational simplicity according to our remarks above.
\end{rem}

We discussed this in \cite{KB2}, \cite{KB3}, and the argument transfers to the general scheme.
Now consider the first increment $\delta v^{r,\rho,l,1}_i$ for $1\leq i\leq n$. For the classical Navier Stokes equation with constant viscosity the functions
$v^{\rho,1,l}_i$ solve the equation in (\ref{scalparasystlin10v}) where the solution has the classical representation
 \begin{equation}\label{scalparasystlin10v*}
 \begin{array}{ll}
v^{\rho,1,l}_i(\tau,x)=\int_{{\mathbb R}^n}v^{r,\rho,l-1}_i(l-1,y)G_l(\tau,x-y)dy\\
\\ 
-\rho_l\int_{l-1}^{\tau}\int_{{\mathbb R}^n}\sum_{j=1}^n v^{r,\rho,l-1}_j(s,y)\frac{\partial v^{r,\rho,l-1}_i}{\partial x_j}(s,y)G_l(\tau-s,x-y)dyds\\
\\
+\rho_l\int_{l-1}^{\tau}\int_{{\mathbb R}^n}\int_{{\mathbb R}^n}\sum_{j,m=1}^n \left( \frac{\partial v^{r,\rho,l-1}_j}{\partial x_m}\frac{\partial v^{r,\rho,l-1}_m}{\partial x_j}\right) (l-1,y)\frac{\partial}{\partial x_i}K_n(z-y)\times\\
\\
\times G_l(\tau-s,x-z)dydzds.\\
\end{array}
\end{equation} 
Note that both summands in (\ref{scalparasystlin10v*}) are convolutions. 
\begin{rem}
Let us mention a notational convention: in the following the expression
\begin{equation}\label{not}
D^{\alpha}_xv^{r,\rho,l-1}_i(l-1,y)
\end{equation}
denotes the multivariate spatial derivative of order $\alpha$ evaluated at $y$ (some authors prefer to write $D^{\alpha}_yv^{r,\rho,l-1}_i(l-1,y)$ while others prefer to emphasize the difference between a variable and a value and prefer a notation as in (\ref{not}).
\end{rem}
Next let
\begin{equation}
\alpha^j:=\alpha-1_j=(\alpha_1-\delta_{1j}1,\alpha_2-\delta_{2j}\cdots ,\alpha_n-\delta_{nj})
\end{equation}
where for $1\leq i\leq n$ the symbol $\delta_{ij}$ denotes the Kronecker delta.

Hence, according to the convolution rule for the multivariate spatial derivative function $D^{\alpha}_xv^{\rho,1,l}_i$ we have for all $\tau\in [l-1,l]$ and all $x\in {\mathbb R}^n$ the representation
\begin{equation}\label{scalparasystlin10v*}
 \begin{array}{ll}
D^{\alpha}_xv^{\rho,1,l}_i(\tau,x)=\int_{{\mathbb R}^n}D^{\alpha}_xv^{r,\rho,l-1}_i(l-1,y)G_l(\tau,x-y)dy\\
\\ 
-\rho_l\int_{l-1}^{\tau}\int_{{\mathbb R}^n}\left( \sum_{j=1}^nD^{\alpha^j}_x\left(  v^{r,\rho,l-1}_j(s,y)\frac{\partial v^{r,\rho,l-1}_i}{\partial x_j}(s,y)\right) \right) G_{l,j}(\tau-s,x-y)dyds\\
\\
+2\rho_l\int_{l-1}^{\tau}\int_{{\mathbb R}^n}\int_{{\mathbb R}^n}\left( \sum_{j,m=1}^n \left( \left( D^{\alpha^m}_x\frac{\partial v^{r,\rho,l-1}_j}{\partial x_m}\right) \frac{\partial v^{r,\rho,l-1}_m}{\partial x_j}\right)\right)  (l-1,y)\frac{\partial}{\partial x_i}K_n(z-y) \times\\
\\
\times G_{l,m}(\tau-s,x-z)dydzds,\\
\end{array}
\end{equation} 
where for the last summand we have applied the convolution rule twice, and the factor $2$ in the last term is because of the symmetry in the product which is convoluted with the Laplacian kernel. 
We can conclude from this that for a time step size $\rho_l$ of order
\begin{equation}
\rho_l\sim \frac{1}{l},
\end{equation}
we have
\begin{equation}\label{dalphax}
D^{\alpha}_x\delta v^{\rho,l,1}_i=D^{\alpha}_xv^{\rho,l,1}_i-D^{\alpha}_xv^{r,\rho,l-1}_i(l-1,.)\sim 1.
\end{equation}
In order to get this conclusion we may argue via Fourier transform.
Fourier transformation ${\cal F}$ with respect to the spatial variable $x$ of the term in (\ref{firstrightns}) transfer convolutions into products, i.e., the term in (\ref{firstrightns}) equals
\begin{equation}\label{firstrightns*}
\begin{array}{ll}
{\cal F}\left( D^{\alpha}_xv^{r,\rho,l-1}_i(l-1,.)\right) {\cal F}\left( G_l(\tau,.)\right) -{\cal F}\left( D^{\alpha}_xv^{r,\rho,l-1}_i(l-1,.)\right)\\
\\
=-{\cal F}\left( D^{\alpha}_xv^{r,\rho,l-1}_i(l-1,.)\right)\left( 1- {\cal F}\left( G_l(\tau,.)\right)\right).
\end{array}
\end{equation}
The concrete expression of ${\cal F}\left( G_l(\tau,.)\right)$ depends on the definition of the Fourier transform which varies a little in the literature. If we define
\begin{equation}
{\cal F}(f)=\int_{{\mathbb R}^n} \exp(2\pi i\xi)f(x)dx,
\end{equation}
then the Fourier transform of the heat kernel looks like
\begin{equation}
{\cal F}\left( G_l(\tau,.)\right)=\exp\left(-4\pi\xi^2\rho_l\tau \right),
\end{equation}
such that the growth with respect to time (which is a parameter in this transformation)
(\ref{firstrightns*}) has the upper bound
\begin{equation}
\sup_{\xi\in {\mathbb R}^n}{\big |}{\cal F}\left( D^{\alpha}_xv^{r,\rho,l-1}_i(l-1,\xi)\right){\big |}{\big |}4\pi\xi^2\rho_l{\big |}\sim \sqrt{l-1}\frac{1}{l},
\end{equation}
and this growth behavior with respect to the time parameter is preserved surely if we transform back with to the spatial variables via inverse Fourier transform.

Note that the growth behavior in (\ref{dalphax}) implies that
\begin{equation}
D^{\alpha}_x v^{\rho,l,1}_i=D^{\alpha}_xv^{\rho,l,1}_i-D^{\alpha}_xv^{r,\rho,l-1}_i(l-1,.)\sim \sqrt{l-1}+1.
\end{equation}
We shall discuss this for the generalized scheme in more detail below, but the reason is essentially as follows. Since we are dealing with convolutions we know that the multivariate spatial derivatives of order $\alpha$ of the first term in (\ref{scalparasystlin10v*}) minus multivariate spatial derivatives of order $\alpha$ of the initial data at time step $l\geq 1$, i.e. the expression,
\begin{equation}\label{firstrightns}
\int_{{\mathbb R}^n}D^{\alpha}_xv^{r,\rho,l-1}_i(l-1,y)G_l(\tau,x-y)dy-D^{\alpha}_xv^{r,\rho,l-1}_i(l-1,.)
\end{equation}
becomes small for a small stepsize $\rho_l$. Furthermore, as $\rho_l\sim \frac{1}{l}$ and $v^{r,\rho,l-1}_i(l-1,.)\sim \sqrt{l-1}$ a small upper bound of the term in (\ref{firstrightns}) satisfies $\sim 1$, i.e., it is independent of the time step number $l\geq 1$. All the other summands in (\ref{scalparasystlin10v*}) are convolutions of $G_l$ with products of value functions of the form
\begin{equation}
v^{r,\rho,l-1}_j(s,y),\frac{\partial v^{r,\rho,l-1}_i}{\partial x_j}\sim \sqrt{l-1}
\end{equation}
where the product growth with respect to time of order $\sim (l-1)$ is compensated by the step size factor $\rho_l\sim \frac{1}{l}$. 
Note that this means that we can realise the bound
\begin{equation}\label{firstincrg}
D^{\alpha}_x\delta v^{\rho,l,1}_i\sim \sqrt{l-1}+1,~\mbox{ for }|\alpha|\leq m,
\end{equation}
and this has some consequence for a refinement of contraction of the construction result for the higher order approximations.
Again in the classical model, from (\ref{deltaurhok0}) we get the representation
\begin{equation}\label{deltaurhok0rep}
 \begin{array}{ll}
\delta v^{\rho,k+1,l}_i(\tau,x)=-\rho_l\int_{l-1}^{\tau}{{\mathbb R}^n}\sum_{j=1}^n v^{\rho,k-1,l}_j\frac{\partial \delta v^{\rho,k,l}_i}{\partial x_j}(s,y)G_l(\tau-s,x-y)dyds\\
\\
-\rho_l\int_{l-1}^{\tau}\int_{{\mathbb R}^n}\sum_j\delta v^{\rho,k,l}_j\frac{\partial v^{\rho,k,l}}{\partial x_j}(s,y)G_l(\tau-s,x-y)dyds+\\ 
\\
+\rho_l\int_{l-1}^{\tau}\int_{{\mathbb R}^n}\int_{{\mathbb R}^n}K_{n,i}(z-y){\Big (} \left( \sum_{j,m=1}^n\left( v^{\rho,k,l}_{m,j}+v^{\rho,k-1,l}_{m,j}\right)(s,y) \right)\times\\
\\
\times  \delta v^{\rho,k,l}_{j,m}(s,y) {\Big)}G_l(\tau-s,x-z)dydzds.
\end{array}
\end{equation}
Again, all summands in (\ref{deltaurhok0rep}) are convolutions and we may represent spatial derivatives of order $\alpha$ by convolution with spatial derivatives of first order of the fundamental heat equation solution $G_l$, and by derivatives of order at most $|\alpha|$ of the convoluted terms. Now consider the first term on the right side in (\ref{deltaurhok0rep}) for $k=1$, and observe the growth with respect to time or with respect to the time step number $l$. The first term on the right side of  (\ref{deltaurhok0rep}) looks like
\begin{equation}
\begin{array}{ll}
-\rho_l\int_{l-1}^{\tau}{{\mathbb R}^n}\sum_{j=1}^n v^{\rho,0,l}_j\frac{\partial \delta v^{\rho,1,l}_i}{\partial x_j}(s,y)G_l(\tau-s,x-y)dyds\\
\\
=-\rho_l\int_{l-1}^{\tau}{{\mathbb R}^n}\sum_{j=1}^n v^{r,\rho,l-1}_j(l-1,.)\frac{\partial \delta v^{\rho,1,l}_i}{\partial x_j}(s,y)G_l(\tau-s,x-y)dyds\\
\\
\sim \frac{1}{l}\sqrt{l-1}\sim \frac{1}{\sqrt{l}},
\end{array}
\end{equation}
because $\delta v^{\rho,1,l}_i\sim 1$ (as we have just observed). Similar for the second term on the right side in (\ref{deltaurhok0rep}). For $k=1$ we have
\begin{equation}
\begin{array}{ll}
-\rho_l\int_{l-1}^{\tau}\int_{{\mathbb R}^n}\sum_j\delta v^{\rho,1,l}_j\frac{\partial v^{\rho,1,l}}{\partial x_j}(s,y)G_l(\tau-s,x-y)dyds\\
\\
\sim\frac{1}{l}\sqrt{l}\sim \frac{1}{\sqrt{l}}.
\end{array}
\end{equation}
In both cases the convolution with the local fundamental solution has the effect of a constant in the upper bound (more details on that are given below in the case of H\"{o}rmander densities and in the proof of the local contraction result. The last term on the right side in (\ref{deltaurhok0rep}) is a double convolution where for $k=1$ we observe that
\begin{equation}
\begin{array}{ll}
\rho_l\int_{l-1}^{\tau}\int_{{\mathbb R}^n}\int_{{\mathbb R}^n}K_{n,i}(z-y){\Big (} \left( \sum_{j,m=1}^n\left( v^{\rho,1,l}_{m,j}+v^{r,\rho,l-1}_{m,j}\right)(s,y) \right)\times\\
\\
\times  \delta v^{\rho,1,l}_{j,m}(s,y) {\Big)}G_l(\tau-s,x-z)dydzds\\
\\
\sim\frac{1}{l}\sqrt{l}\sim \frac{1}{\sqrt{l}}.
\end{array}
\end{equation}
Since the value functions $\left( v^{\rho,1,l}_{m,j}+v^{r,\rho,l-1}_{m,j}\right)(s,y)$ and $\delta v^{\rho,1,l}_{j,m}(s,y)$ have polynomial decay of order $p\geq 2$ the convolution with the Laplacian kernel $K_{n,i}$ is finite and adds just another constant to an upper bound which is independent of time step number $l$. Again we shall observe this in more detail in the proof of the local contraction result below. The reasoning for multivariate derivative $D^{\alpha}_x\delta v^{\rho,l,k}_i$ uses the fact that the fundamental solution $G_l$ can take one spatial derivative and the other convolution terms involving the value function approximations take the other derivatives of order $|\alpha|-1$ via the convolution rule (as described above). Then we can proceed as before. Obviously at each approximation step we get at least one additional factor $\frac{1}{\sqrt{l}}$ (although we do not need this latter observation for the reasoning that the scheme is global - it is sufficient for the controlled scheme that all higher order terms satisfy the local contraction behavior and get a factor $\frac{1}{\sqrt{l}}$.
From these considerations it is clear that for $k\geq 2$ we get
\begin{equation}\label{secondincrgr}
D^{\alpha}_x\delta v^{r,\rho,l,k}_i\sim \left( \frac{1}{\sqrt{l}}\right)^{k-1},~\mbox{ for }|\alpha|\leq m,
\end{equation}
for the controlled scheme, since for $k\geq 2$ we have
\begin{equation}\label{identityr}
D^{\alpha}_x\delta v^{r,\rho,l,k}_i=D^{\alpha}_x\delta v^{\rho,l,k}_i,
\end{equation}
where we recall that the increments on the right side of (\ref{identityr}) are understood with respect to our local scheme which starts with $v^{r,\rho,l-1}_i(l-1,.)$, and on the right side we understand 
\begin{equation}
\begin{array}{ll}
\delta v^{r,\rho,l,k}_i=v^{r,\rho,l,k}_i-v^{r,\rho,l,k-1}_i
=v^{\rho,l,1}_i+\sum_{p=2}^k\delta v^{\rho,l,k}_i\\
\\
+\delta r^l_i-v^{\rho,l,1}_i-\sum_{p=2}^{k-1}\delta v^{\rho,l,k}_i-\delta r^l_i=\delta v^{\rho,l,k}_i
\end{array}
\end{equation}
 Note that in our notation
\begin{equation}
v^{\rho,l,1}_i=v^{r,\rho,l-1}_i(l-1)+\delta v^{\rho,l,1}_i
\end{equation}

As we said these observations motivate our definition of a control functions $r^l_i$ (or a part of the control function) in \cite{KB1} and \cite{KB2}, where we defined
\begin{equation}\label{deltarl}
\delta r^l_i=r^l_i-r^{l-1}_i(l-1,.)=-\delta v^{r,\rho,l,1}_i.
\end{equation}
This implies that we have
\begin{equation}\label{funciv3}
\begin{array}{ll}
v^{r,\rho ,l }_i=v^{r,\rho,l-1}_i+\sum_{k=1}^{\infty} \delta v^{r,\rho, l,k}_i\\
\\
=v^{r,\rho,l-1}_i+\delta v^{r,\rho,l,1}_i+\sum_{k=2}^{\infty} \delta v^{r,\rho, l,k}_i\\
\\
=v^{r,\rho,l-1}_i+\sum_{k=2}^{\infty} \delta v^{\rho, l,k}_i
,~1\leq i\leq n.
\end{array}
\end{equation}
Hence, we have
\begin{equation}
D^{\alpha}_xv^{r,\rho ,l }_i\sim \sqrt{l} \mbox{ for } |\alpha|\leq m.
\end{equation}
Furthermore, note that
\begin{equation}
D^{\alpha}_xr^l_i\sim l\mbox{ for } |\alpha|\leq m.
\end{equation}
Next we consider the situation of the generalized system. An analysis of the H\"{o}rmander estimates has the result that for each $x\in {\mathbb R}^n$ there is an $\epsilon >0$ and a ball $B_{\epsilon}(x)$ of radius $\epsilon >0$ around $x$ such that
\begin{equation}\label{apriorih0}
{\big |}1_{B_{\epsilon}(x)}G^l_H(\tau,x;s,y){\big |}\leq \frac{C}{(\tau-s)^{\alpha}(x-y)^{n-2\alpha}}
\end{equation}
for some $\alpha \in (0,1)$ and some constant $C>0$.
Here, $1_{B_{\epsilon}(x)}$ denotes the characteristic function which equals one on $B_{\epsilon}$ and is zero elsewhere. Furthermore, for the first order spatial derivatives we have
\begin{equation}\label{apriorih1}
{\big |}1_{B_{\epsilon}(x)}\frac{\partial}{\partial x_i}G^l_H(\tau,x;s,y){\big |}\leq \frac{C}{(\tau-s)^{\alpha}(x-y)^{n+1-2\alpha}}
\end{equation}
for some $\alpha \in (0,1)$ and some constant $C>0$. These estimates follow from the H\"{o}rmander estimates in \cite{H}. We shall give a detailed description in \cite{KH}, but cf. also our remarks at the end of the introduction of this paper.
Next for each $x\in {\mathbb R}^n$ we choose a ball $B_{\epsilon}(x)$ such that the estimates in (\ref{apriorih0}) and (\ref{apriorih1}) are satisfied. Then we consider
\begin{equation}
v^{\rho,l,1}_i(\tau,x)=v^{\rho,l,1}_{iB_{\epsilon}}(\tau,x)+v^{\rho,l,1}_{i(1-B_{\epsilon})}(\tau,x)
\end{equation}
where
\begin{equation}\label{scalparasystlin10vherB}
\begin{array}{ll}
v^{\rho,l,1}_{iB}(\tau,x):=\int_{B_{\epsilon}(x)}v^{r,\rho,l-1}_i(l-1,y)G^l_H(\tau,x,l-1,y)dy\\
\\ 
-\rho_l\int_{l-1}^{\tau}\int_{B_{\epsilon}(x)}V_{B}\left[v^{r,\rho,l-1}(l-1,.)\right] v^{r,\rho,l-1}_i(l-1,y)G^l_H(\tau,x;s,y)dyds+\\
\\
\rho_l\int_{l-1}^{\tau}\int_{B_{\epsilon}(x)}\int_{{\mathbb R}^n}\sum_{j,m=1}^n \left(c_{jm} \frac{\partial v^{r,\rho,l-1}_m}{\partial x_j}(l-1,.)\frac{\partial v^{r,\rho,l-1}_j}{\partial x_m}(l-1,.)\right) (\tau,y)\times\\
\\
\times\frac{\partial}{\partial x_i}K^{\mbox{ell}}_n(z-y)G^l_H(\tau,x;s,z)dydzds,
\end{array}
\end{equation}
and 
\begin{equation}\label{scalparasystlin10vher(1-B)}
\begin{array}{ll}
v^{\rho,l,1}_{i(1-B)}(\tau,x):=\int_{{\mathbb R}^n\setminus B_{\epsilon}(x)}v^{r,\rho,l-1}_i(l-1,y)G^l_H(\tau,x;l-1,y)dy\\
\\ 
-\rho_l\int_{l-1}^{\tau}\int_{{\mathbb R}^n\setminus B_{\epsilon}(x)}V_{B}\left[v^{r,\rho,l-1}(l-1,.)\right] v^{r,\rho,l-1}_i(l-1,y)G^l_H(\tau,x;s,y)dyds+\\
\\
\rho_l\int_{l-1}^{\tau}\int_{{\mathbb R}^n\setminus B_{\epsilon}(x)}\int_{{\mathbb R}^n}\sum_{j,m=1}^n \left(c_{jm} \frac{\partial v^{r,\rho,l-1}_m}{\partial x_j}(l-1,.)\frac{\partial v^{r,\rho,l-1}_j}{\partial x_m}(l-1,.)\right) (\tau,y)\times\\
\\
\times\frac{\partial}{\partial x_i}K^{\mbox{ell}}_n(z-y)G^l_H(\tau,x;s,z)dydzds.
\end{array}
\end{equation} 
The latter term can be treated as before in the case of constant viscosity by using the Kusuoka-Stroock estimates. Since $|x-y|\geq \epsilon$ in the integrals of (\ref{scalparasystlin10vher(1-B)}) we can differentiate the kernel in order to get representations of the spatial derivatives $D^{\alpha}_xv^{\rho,l,1}_i$ of the first approximation at time step $l\geq 1$, i.e., we have for all $|\alpha|\leq m$ the representation
\begin{equation}\label{scalparasystlin10vher(1-B)alpha}
\begin{array}{ll}
D^{\alpha}_xv^{\rho,l,1}_{i(1-B)}(\tau,x):=\int_{{\mathbb R}^n\setminus B_{\epsilon}(x)}v^{r,\rho,l-1}_i(l-1,y)D^{\alpha}_xG^l_H(\tau,x;l-1,y)dy\\
\\ 
-\rho_l\int_{l-1}^{\tau}\int_{{\mathbb R}^n\setminus B_{\epsilon}(x)}V_{B}\left[v^{r,\rho,l-1}(l-1,.)\right] v^{r,\rho,l-1}_i(l-1,y)D^{\alpha}_xG^l_H(\tau,x;s,y)dyds+\\
\\
\rho_l\int_{l-1}^{\tau}\int_{{\mathbb R}^n\setminus B_{\epsilon}(x)}\int_{{\mathbb R}^n}\sum_{j,m=1}^n \left(c_{jm} \frac{\partial v^{r,\rho,l-1}_m}{\partial x_j}(l-1,.)\frac{\partial v^{r,\rho,l-1}_j}{\partial x_m}(l-1,.)\right) (\tau,y)\times\\
\\
\times\frac{\partial}{\partial x_i}K^{\mbox{ell}}_n(z-y)D^{\alpha}_xG^l_H(\tau,x;s,z)dydzds.
\end{array}
\end{equation} 
We get upper bounds of ${\big |}D^{\alpha}_xv^{\rho,l,1}_i(\tau,x){\big |}$ by estimating 
the modulus of multivariate spatial derivatives of the density ${\big |}D^{\alpha}_xG^l_H(\tau,x;s,z){\big|}$ by the Kusuoka-Stroock estimates or by the estimates we provide in \cite{KH}, and estimate the modulus of the other integrands similar as in the case of constant viscosity using inductive information of time growth $D^{\alpha}_xv^{r,\rho,l-1}_i(l-1,.)\sim \sqrt{l-1}$ for multiindices $|\alpha|\leq m$ and some $m\geq 2$. Note that the upper bound that we get is a sum of convolutions. As in the case of constant viscosity we may use Fourier transformation with respect to thespatial variables in order to estmate the growth of the first summand with respect to the time step number, adn we may use inductive information of time growth and the choice of $\rho_l\sim \frac{1}{l}$ for the other summands. Hence, similar as in the case of constant viscosity described above we get
\begin{equation}\label{firstorder(1-B)}
{\big |}D^{\alpha}_xv^{\rho,l,1}_{i(1-B)}(\tau,x)-D^{\alpha}_xv^{r,\rho,l-1}_{i(1-B)}(l-1,x){\big |}\sim 1.
\end{equation}
For the complement term considered in (\ref{scalparasystlin10vherB}) an additional step is needed. For the summand of first order approximation at time step $l$, i.e., the function $v^{\rho,l,1}_{iB}$, and its first order spatial derivatives we may use the estimates in (\ref{apriorih0}) and (\ref{apriorih1}) in
\begin{equation}\label{scalparasystlin10vherBmod}
\begin{array}{ll}
{\big |}D^{\beta}_xv^{\rho,l,1}_{iB}(\tau,x){\big |}:=\int_{B_{\epsilon}(x)}{\big |}v^{r,\rho,l-1}_i(l-1,y){\big |}{\big |}G^l_H(\tau,x,l-1,y){\big |}dy\\
\\ 
-\rho_l\int_{l-1}^{\tau}\int_{B_{\epsilon}(x)}{\big |}V_{B}\left[v^{r,\rho,l-1}(l-1,.)\right] v^{r,\rho,l-1}_i(l-1,y){\big |}{\big |}G^l_H(\tau,x;s,y){\big |}dyds+\\
\\
\rho_l\int_{l-1}^{\tau}\int_{B_{\epsilon}(x)}\int_{{\mathbb R}^n}\sum_{j,m=1}^n {\big |}\left(c_{jm} \frac{\partial v^{r,\rho,l-1}_m}{\partial x_j}(l-1,.)\frac{\partial v^{r,\rho,l-1}_j}{\partial x_m}(l-1,.)\right) (\tau,y){\big |}\times\\
\\
\times{\big |}\frac{\partial}{\partial x_i}K^{\mbox{ell}}_n(z-y){\big |}{\big |}G^l_H(\tau,x;s,z){\big |}dydzds,
\end{array}
\end{equation}
where the multiindices $\beta$ satisfy $0\leq |\beta|\leq 1$. For spatial derivatives of order $\alpha$ with $|\alpha|\geq 2$ we need the local adjoint $G^{l,B_{\epsilon}(x),*}_H(\tau,x;s,z)$(cf. \cite{KH} and the remark below) and use a representation via the adjoint an estimate via
\begin{equation}\label{scalparasystlin10vherBmod}
\begin{array}{ll}
{\big |}D^{\alpha}_xv^{\rho,l,1}_{iB}(\tau,x){\big |}:=\int_{B_{\epsilon}(x)}{\big |}D^{\gamma}_xv^{r,\rho,l-1}_i(l-1,y){\big |}{\big |}G^{l,B_{\epsilon}(x),*}_H(\tau,x,l-1,y){\big |}dy\\
\\ 
-\rho_l\int_{l-1}^{\tau}\int_{B_{\epsilon}(x)}{\big |}V_{B}\left[v^{r,\rho,l-1}(l-1,.)\right] v^{r,\rho,l-1}_i(l-1,y){\big |}{\big |}G^{l,B_{\epsilon}(x),*}_H(\tau,x;s,y){\big |}dyds+\\
\\
\rho_l\int_{l-1}^{\tau}\int_{B_{\epsilon}(x)}\int_{{\mathbb R}^n}\sum_{j,m=1}^n {\big |}\left(c_{jm} \frac{\partial v^{r,\rho,l-1}_m}{\partial x_j}(l-1,.)\frac{\partial v^{r,\rho,l-1}_j}{\partial x_m}(l-1,.)\right) (\tau,y){\big |}\times\\
\\
\times{\big |}\frac{\partial}{\partial x_i}K^{\mbox{ell}}_n(z-y){\big |}{\big |}G^{l,B_{\epsilon}(x),*}_H(\tau,x;s,z){\big |}dydzds .
\end{array}
\end{equation}
We get
\begin{equation}
{\big |}D^{\alpha}_xv^{\rho,l,1}_{iB}(\tau,x)-D^{\alpha}_xv^{r,\rho,l-1}_{iB}(l-1,x){\big |}\sim 1,
\end{equation}
and together with (\ref{firstorder(1-B)}) we get
\begin{equation}
{\big |}D^{\alpha}_xv^{\rho,l,1}_{i}(\tau,x)-D^{\alpha}_xv^{r,\rho,l-1}_{i}(l-1,x){\big |}\sim 1.
\end{equation}
\begin{rem}
For parabolic equation with strictly spatial elliptic operators each density $p$ has a adjoint $p^*$ (which solves a parabolic adjoint equation) such that
\begin{equation}
p(t,x;s,y)=p^*(s,y;t,x) \mbox{ and }~D^{\alpha}_xp(t,x;s,y)=D^{\alpha}_yp^*(s,y;t,x)
\end{equation}
For H\"{o}rmander diffusion we can define a local adjoint, i.e., for each argument $x\in {\mathbb R}^n$ there is a ball $B_{\epsilon}(x)$ of radius $\epsilon>0$ around $x$ such that the H\"{o}rmander density has a local adjoint on this ball. We give the details for this in \cite{KH}. The main reason is this: the H\"{o}rmander condition encodes infinitesimal rotations and shifts caused by the drift term with diffusions caused by the second order terms. This leads to the possibility of local expansions of the density and the construction of local adjoints (cf. \cite{KH}).
\end{rem}

Next we consider the higher order terms looking for a refinement of the local contraction result regarding the dependence of the contraction constant and the time step number of the scheme. In order to have a global linear bound for the controlled scheme (of type iiia)) it is essential to have
\begin{equation}
\sum_{k=2}^{\infty} \delta v^{\rho,l,k}_i\sim \frac{1}{\sqrt{l}}
\end{equation}
for all $1\leq i\leq n$.
The essential step is to establish such a growth behavior with respect to the time step number for the functional increments $\delta v^{\rho,l,k}_i$. Note that we have
\begin{equation}
\sum_{k=2}^{\infty} \delta v^{\rho,l,k}_i=\sum_{k=2}^{\infty} \delta v^{r,\rho,l,k}_i
\end{equation}
for all controlled schemes proposed in the introduction, such that this result for the uncontrolled scheme transfers to any of the controlled schemes directly. Again we start with multivariate spatial derivatives of order $0\leq |\beta|\leq 1$ and split the representation 
for $k=1$ in two summands choosing for each $x\in {\mathbb R}^n$ and $\epsilon >0$ and a ball $B_{\epsilon}(x)$ of radius $\epsilon>0$ around $x$ such that the a priori estimates (\ref{apriorih0}) and (\ref{apriorih1}) hold. We get the representation
\begin{equation}
\delta v^{\rho,l,2}_i(\tau,x)=\delta v^{\rho,l,2}_{iB}(\tau,x)+\delta v^{\rho,l,2}_{i(1-B)}(\tau,x)
\end{equation}
for all $(\tau,x)\in [l-1,l]\times {\mathbb R}^n$, where
\begin{equation}\label{deltaurhok0herk2(1-B)}
\begin{array}{ll}
D^{\beta}_x\delta v^{\rho,l,2}_i(\tau,x)=\\
\\ 
-\rho_l\int_{l-1}^{\tau}\int_{{\mathbb R}^n\setminus B_{\epsilon}(x)}\left( V_{B}\left[v^{\rho,l,1}\right] \delta v^{\rho,l,1}_i+V_{B}\left[\delta v^{\rho,l,1}\right] v^{\rho,l,1}_i\right) (s,y)\times\\
\\
\times D^{\beta}_xG^l_H(\tau-s,x-y)dyds+\rho_l\int_{l-1}^{\tau}\int_{{\mathbb R}^n\setminus B_{\epsilon}(x)}\int_{{\mathbb R}^n}K^{\mbox{ell}}_{n,i}(z-y)\times\\
\\
{\Big (} c_{jm}\left( \sum_{j,m=1}^n\left( v^{\rho,l,1}_{m,j}(\tau,y)+v^{r,\rho,l-1}_{m,j}(l-1,y)\right) \right)  \delta v^{\rho,l,1}_{j,m}(s,y) {\Big)}\times\\
\\
\times D^{\beta}_xG^l_H(\tau-s,x-z)dydzds,
\end{array}
\end{equation}
and
\begin{equation}\label{deltaurhok0herk2B}
\begin{array}{ll}
D^{\beta}_x\delta v^{\rho,l,2}_{iB}(\tau,x)=\\
\\ 
-\rho_l\int_{l-1}^{\tau}\int_{B_{\epsilon}(x)}\left( V_{B}\left[v^{\rho,l,1}\right] \delta v^{\rho,l,1}_i+V_{B}\left[\delta v^{\rho,l,1}\right] v^{\rho,l,1}_i\right) (s,y)\times\\
\\
\times D^{\beta}_xG^l_H(\tau-s,x-y)dyds+\rho_l\int_{l-1}^{\tau}\int_{B_{\epsilon}(x)}\int_{{\mathbb R}^n}K^{\mbox{ell}}_{n,i}(z-y)\times\\
\\
{\Big (} c_{jm}\left( \sum_{j,m=1}^n\left( v^{\rho,l,1}_{m,j}(\tau,y)+v^{\rho,l-1}_{m,j}(l-1,y)\right) \right)  \delta v^{\rho,l,1}_{j,m}(s,y) {\Big)}\times\\
\\
\times D^{\beta}_xG^l_H(\tau-s,x-z)dydzds.
\end{array}
\end{equation}
For the term in (\ref{deltaurhok0herk2(1-B)}) we may use the Kusuoka-Stroock estimates in order to get a  upper bound for the modulus ${\big |}D^{\beta}_x\delta v^{\rho,l,2}_i(\tau,x){\big |}$ for $0\leq |\beta|\leq 1$, and this type of upper bound can be extended to higher order derivatives for ${\big |}D^{\alpha}_x\delta v^{\rho,l,2}_i(\tau,x){\big |}$ and any multiindex $\alpha$ with $|\alpha|\geq 0$ just by using the Kusuoka-Stroock estimates. The involved constants are surely independent of the time step number $l\geq 1$ as far as the upper bounds of the fundamental solution are concerned. For the local terms in (\ref{deltaurhok0herk2B}) we may use the local a priori estimates for H\"{o}rmander diffusions.

\section{Global bound of the Leray projection term, local and global solutions}
  
We have observed that the functions
\begin{equation}\label{gbcontrol1}
l\rightarrow |v^{r,\rho,l}_j(l,.)|^2_{H^m}
\end{equation}
for $m\geq 2$ of a controlled scheme with control function $r^l_i$ of type iiia) or iiib) have linear linear growth with respect to the time step number $l$. Moreover, the control function $r^l_i$ themselves satisfy
\begin{equation}\label{gbcontrol2}
r^l_i(l,.)\sim l.
\end{equation}
This means that we have obtained a global regular solution $\mathbf{v}=(v_1,\cdots,v_n)$ which is defined in transformed time coordinates $\tau=\rho_l t$ on the domains $[l-1,l]\times {\mathbb R}^n$ by
\begin{equation}\label{classicalrep}
v^{\rho,l}_i(\tau,x)=v^{r,\rho,l}_i(\tau,x)-r^l_i(\tau,x)
\end{equation}
On each domain $\left[l-1,l\right]\times {\mathbb R}^n$ the function $v^{\rho,l}_i\in C^{1,2}\left( \left(l-1,l\right)\times {\mathbb R}^n\right) $ is a local classical solution
of the generalized (highly degenerate) incompressible Navier Stokes equation, and as we have shown that
\begin{equation}
\sup_{l\in {\mathbb N}}|v^{r,\rho,l}_i(l,.)| \mbox{ is bounded }
\end{equation}
for all $1\leq i\leq n$, and we have a global linear upper bound of the control functions $r^l_i$ with respect to the time step number $l\geq 1$, i.e.,
\begin{equation}
\sup_{l\in {\mathbb N}}|r^l_i(l,.)|\leq Cl
\end{equation}
for some constant $C>0$ and all $1\leq i\leq n$, we have also a global linear upper bound with respect to the time step number $l\geq 1$ of the value functions $v_i$, i.e., we have
\begin{equation}
\sup_{l\in {\mathbb N}}|v^{\rho,l}_i(l,.)|\leq Cl.
\end{equation}
As both summands on the right side of (\ref{classicalrep}) are locally $C^{1,2}$ on the domains $\left(l-1,l\right)\times{\mathbb R}^n$, this is also true for $v^{\rho,l}_i,~1\leq i\leq n$. Furthermore, if  $v^{r,\rho}_i,~1\leq i\leq n$ denotes the global controlled velocity function on $[0,\infty)\times {\mathbb R}^n$ with $v^{r,\rho}_i(\tau,x)=v^{r,\rho,l}_i/\tau,x)$ for $\tau\in [l-1,l]\times {\mathbb R}^n$, and $r_i:[0,\infty)\times {\mathbb R}^n\rightarrow {\mathbb R}$ denote the global control functions for $1\leq i\leq n$ with $r_i(\tau,x)=r^l_i(\tau,x)$ for $\tau\in [l-1,l]\times {\mathbb R}^n$, then we have by construction (and our argument above) that the functions $v^{r,\rho}_i,~1\leq i\leq n$ and $r_i,~ 1\leq i\leq n$ are globally Lipschitz on the whole domain $\left[0,\infty\right)\times {\mathbb R}^n$. Hence $v^{\rho}_i:=v^{r,\rho}_i-r_i$ is globally Lipschitz on the whole domain $[0,\infty)\times {\mathbb R}^n$ and such that for all time step number $l$ the restrictions of $v^{\rho}_i$ are in $C^{1,2}\left( (l-1,l)\times {\mathbb R}^n\right) $. It follows then by standard regularity theorems that $v^{\rho}_i$ is a global classical solution of the generalized incompressible Navier Stokes equation in time 
transformed coordinates. Hence the globally defined functions $v_i,~1\leq i\leq n$ with $v_i(t,x)=v_i(\tau,x),~1\leq i\leq n$ with $1\leq i\leq n$ and $t=\rho_l \tau$ for all $[l-1,l)\times {\mathbb R}^n$ and $l\geq $ is a global classical solution of the original generalized incompressible Navier Stokes equation. We can sharpen this result a bit and prove the existence of global upper bounds which are completely independent of the  time step number.   
Indeed, next we consider controlled scheme which lead to a global bound of the Leray projection term which is independent of the time step number. We have found several different arguments for this conclusion. First consider the controlled scheme with
\begin{equation}
v^{r,\rho,l}_j=v^{\rho,l}_j+r^{l}_j
\end{equation}
for some functions $r^l_j$, where
\begin{equation}\label{controll100}
r^{l}_j-r^{l-1}_j=-\left( v^{\rho,l,1}_j-v^{r,\rho,l-1}_j(l-1,.)\right)+\int_{l-1}^{\tau}\phi^{l}_j(s,y)G^l_H(\tau-s,x-y)dyds,
\end{equation}
and where the source term in (\ref{controll1}) is of the form
\begin{equation}\label{sourceleray00}
\phi^{l}_j(s,y)=-\frac{v_j^{r,\rho,l-1}}{C}(l-1,.)-\frac{r^{l-1}_i(l-1,.)}{C^2},
\end{equation}
At time step $l=1$ we are free to choose the 'data' $r^{l-1}_i(l-1,.)=r^{0}_i(0,.)$ for $1\leq i\leq n$. We may choose them such that they 'have the same signs' as the data $h_i$ of the Cauchy problem, i.e., we choose
\begin{equation}
r^0_i(0,.)=\frac{h_i(.)}{C}~\mbox{ for all $1\leq i\leq n$.} 
\end{equation}
The reasoning is then as follows. Assume we have computed $v^{r,\rho,l-1}_i$ and $r^{l-1}_i$ for all $1\leq i\leq n$ and for $l\geq 2$. In our controlled scheme we first compute the local solution of the generalised incompressible Navier Stokes equation with data $v^{r,\rho,l-1}_i(l-1,.)$ for all $1\leq i\leq n$ via the functional series
\begin{equation}
v^{r^{l-1},\rho,l}_i=v^{r,\rho,l-1}_i(l-1,.)+\delta v^{\rho,l,1}_i+\sum_{k\geq 2}\delta v^{\rho,l,k}_i,
\end{equation}
where we indicate with the superscript $r^{l-1}$ that we compute the local solution of the generalised (but otherwise uncontrolled) incompressible Navier Stokes equation with respect to the data $v^{r,\rho,l-1}_i(l-1,.)$ which include the information of the control function from the previous time step. Then looking at the controlled function (including the control function at time step $l$ we add the increment $\delta r^l_i$ defined in (\ref{controll100}) and get the representation
\begin{equation}
v^{r,\rho,l}_i=v^{r^{l-1},\rho,l}_i+\delta r^l_i=v^{r,\rho,l-1}_i(l-1,.)+\sum_{k\geq 2}\delta v^{\rho,l,k}_i+\int_{l-1}^{\tau}\phi^{l}_j(s,y)G^l_H(\tau-s,x-y)dyds.
\end{equation}

We have observed that the subtraction of the first order increments $\delta v^{\rho,l,1}_i$ via the control function increments $\delta r^l_i$ is useful in order to preserve polynomial decay of the controlled velocity value functions. For the classical incompressible Navier Stokes equation (even with variable strictly elliptic viscosity) we do not need this.
Now as $\rho_l$ becomes small the higher order correction
\begin{equation}
v^{r,\rho,l-1}_i(l-1,.)+\sum_{k\geq 2}\delta v^{\rho,l,k}_i(l,.)
\end{equation}
become small compared to the source term
\begin{equation}
\int_{l-1}^{l}\phi^{l}_j(s,y)G^l_H(\tau-s,x-y)dyds
\end{equation}
where for $\rho_l$ small the diffusion effect of the kernel $G^l_H$ is small (similar as in the scheme for classical model for $G_l$). For this reason if we defined the control function increment via (\ref{controll1}) below, we would get a global  linear bounds
\begin{equation}\label{addlast}
|v^{r,\rho,l}_i(l,.)|\leq Cl,~|r^l_i(l,.)|\leq Cl
\end{equation}
for some $C>0$ by arguments similar as in the preceding sections. This looks a little worse
than what we would get if we defined the control function via (\ref{controll1}) below, because in that case we get a uniform bound for the controlled velocity functions in addition.  Well, the estimates in (\ref{addlast}) and similar estmates for derivatives lead us still to the conclusion of the existence of a global classical solution. However the extended scheme leads to the slight improvement that we have global uniform bounds for the controlled velocity functions and for the control functions. Since this is only a slight improvement and the main theorem of a global classical solution is achieved without these additional observations, we only sketch the argument.

Consider an argument $x\in {\mathbb R}^n$. As $l\geq 1$ varies and as long as 
\begin{equation}\label{addlast2}
v^{r,\rho,l}_i(l,x) \mbox{~and~}r^l_i(l,x)
\end{equation}
have the same sign, we observe that we have a uniform upper bound for both function independent of the time step number $l\geq 1$. Now, if $l_0$ is the first times step number such that
\begin{equation}\label{addlast2}
v^{r,\rho,l}_i(l_0,x) \mbox{~and~}r^l_i(l_0,x)
\end{equation}
have different signs, then we observe that the function
\begin{equation}\label{addlast3}
l\rightarrow |v^{r,\rho,l}_i(l,x)+r^l_i(l,x)|,~l\geq l_0
\end{equation}
has - for time step sizes $\rho_l>0$ of order $\rho_l\sim \frac{1}{C^4}$ the tendency to fall forever or up to the time step number $l_e$ where both functions in (\ref{addlast2}) have the same sign again in the following sense: either we have that
\begin{equation}
v^{r,\rho,l}_i(l,x),r^l_i(l,x)\sim \frac{1}{C^2},
\end{equation}
the function in (\ref{addlast}) is decreasing after finitely many time steps in the sense that for any argument $l_0\leq l\leq l_e$  or $l\geq l_0$ (if $l_e=\infty$) we find a $l'\geq l$ such that 
\begin{equation}
|v^{r,\rho,l}_i(l',x)+r^l_i(l',x)|\leq |v^{r,\rho,l}_i(l,x)+r^l_i(l,x)|.
\end{equation}
Many cases have to be considered for this argument and in order to make it fully precise we have also to show how exactly the time step sizehas to be chosen. However, since this is only a slight improvement of the general argument of a global linear bound (which is enough in order to prove our main theorem) we shall not provide all the details here.

Let us make some additional remarks. 
We make some additional observation concerning ia). The schemes consdiered are schemes with bounded controlled velocity functions and with control functions which are linearly bounded. First we consider a scheme
\begin{equation}
v^{r,\rho,l,k}_j=v^{\rho,l,k}_j+r^{l,0}_j
\end{equation}
for some functions $r^l_j$, where
\begin{equation}\label{controll1}
r^{l,0}_j-r^{l-1,0}_j=-\left( v^{\rho,l,1}_j-v^{r,\rho,l-1}_j(l-1,.)\right)+\int_{l-1}^{\tau}\phi^{l,0}_j(s,y)G_l(\tau-s,x-y)dyds,
\end{equation}
and where the source term in (\ref{controll1}) is of the form
\begin{equation}\label{sourceleray}
\phi^{l,0}_j(s,y)=-\frac{v_j^{r,\rho,l-1}}{C}(l-1,.),
\end{equation}
and show that this is a global scheme in case of simple models with constant viscosity. This scheme has the advantage that it is possible to choose a uniform time step size. Note that the scheme in ia) is without the increment $-\left(v^{\rho,l,1}_j-v^{\rho,l-1}_j(l-1,.)\right)$, but the proof is similar and it would be cumbersome to list all variations of argument. The choice in (\ref{controll1}) has the advantage that it works also for the generalized degenerate model as we shall observe. In a second step, and in order to get a uniform global bound we consider control functions with 
\begin{equation}\label{controll1}
r^{l}_j-r^{l-1}_j=-\left( v^{\rho,l,1}_j-v^{r,\rho,l-1}_j(l-1,.)\right)+\int_{l-1}^{\tau}\phi^{l}_j(s,y)G_l(\tau-s,x-y)dyds,
\end{equation}  
and where the source term is as in \cite{KB3}, i.e., 
\begin{equation}\label{sourceleray*}
\phi^{l}_j(s,y)=-\frac{v_j^{r,\rho,l-1}}{C}(l-1,.)-\frac{r^{l-1}_j}{C^2}(l-1,.).
\end{equation} 
Let us consider the simple model with constant viscosity first. First we note that it is a major step to show that for given $m\geq 2$ and for all $1\leq i\leq n$ and all multiindices $\alpha$ with $|\alpha|\leq m$ we have the implication
\begin{equation}\label{boundedimpl}
\sup_{x\in {\mathbb R}^n}{\big |}D^{\alpha}_xv^{r,\rho,l-1}_i(l-1,.){\big |}\leq C\Rightarrow\sup_{x\in {\mathbb R}^n}{\big |}D^{\alpha}_xv^{r,\rho,l}_i(l,.){\big |}\leq C
\end{equation}
for some constant $C>0$ and all $l\geq 1$. If (\ref{boundedimpl}) holds, and we have a linear bound
\begin{equation}
|D^{\alpha}_x r^l_i(l,.)|\sim l,
\end{equation}
then we have a global linear bound for the velocity function $v^{\rho,l}_i=v^{r,\rho,l}_i-r^l_i$.
Note that the increment (or 'decrement') of the controlled value function at time step $l$ satisfies for all $1\leq i\leq n$
\begin{equation}\label{incrementcontrol}
\begin{array}{ll}
\delta v^{r,\rho,l}_i=v^{r,\rho,l}_i(l,.)-v^{r,\rho,l-1}_i(l-1,.)\\
\\
=\sum_{k\geq 2}\delta v^{r,\rho,l,k}_i(l,.)+\int_{l-1}^{\tau}\phi^{l,0}_j(s,y)G_l(\tau-s,x-y)dyds\\
\\
=\sum_{k\geq 2}\delta v^{\rho,l,k}_i(l,.)+\int_{l-1}^{l}\phi^{l,0}_j(s,y)G_l(\tau-s,.-y)dyds,
\end{array}
\end{equation}
where the higher order increments satisfy a local contraction with contraction factor $\frac{1}{2}$ such that
\begin{equation}\label{obgub1}
\sum_{k\geq 2}{\big |}\delta v^{\rho,l,k}_i(l,.){\big |}_{C^1((l-1,l), H^{2m})}
\leq {\big |}\delta v^{\rho,l,1}_i(l,.){\big |}_{C^1((l-1,l), H^{2m})}.
\end{equation}
As the time step size $\rho_l>0$ becomes small we know that
\begin{equation}\label{obgub2}
{\big |}\delta v^{\rho,l,1}_i(l,.){\big |}_{C^1((l-1,l)\times H^{2m}}\leq \frac{1}{4},
\end{equation}
while $G_l$ is close to identity. Consider first the controlled velocity value function themselves, i.e., consider $\alpha =0$. If $|v^{r,\rho,l-1}_i(l-1,x)|\in\left[\frac{3C}{4},C \right]$ for some $x\in {\mathbb R}^n$ then for small $\rho_l>0$ we have
\begin{equation}\label{alpha0}
\begin{array}{ll}
{\big |}\int_{l-1}^{l}\phi^{l,0}_j(s,y)G_l(\tau-s,x-y)dyds{\big |}\\
\\
\geq {\big |}\int_{l-1}^{l}\left( -\frac{v_j^{r,\rho,l-1}}{C}(l-1,.)(s,y)\right) G_l(\tau-s,x-y)dyds{\big |}\\
\\
\geq \frac{1}{2}(l-(l-1)=\frac{1}{2},
\end{array}
\end{equation}
such that with the observations in (\ref{obgub1}) and in (\ref{obgub2}) we get indeed  (\ref{boundedimpl}) for $\alpha=0$. Similar for $\alpha >0$ where we note that we may use a convolution rule in order to have the estimate for derivatives in \ref{alpha0}. As we have
\begin{equation}
\delta D^{\alpha}_xv^{\rho,l,1}_i(l,.)\sim 1
\end{equation}
and
\begin{equation}
D^{\alpha}_x\int_{l-1}^{l}\phi^{l,0}_j(s,y)G_l(\tau-s,x-y)dy\sim 1
\end{equation}
we have
\begin{equation}
D^{\alpha}_x\delta r^{l,0}_i(l,.)\sim 1, \mbox{ whence }D^{\alpha}_xr^l_i(l,.)\sim l,
\end{equation}
such that the control functions are linearly bounded.

Note that in the application of the local contraction result we used
\begin{equation}
\sum_{k\geq 2}\delta v^{r,*,\rho,l,k}_i=\sum_{k\geq 2}\delta v^{*,\rho,l,k}_i.
\end{equation}
We also used the inheritance of polynomial spatial decay of our scheme in order to conclude from (\ref{boundedimpl}) that a global bound exists with respect to the $|.|_{C^1((l-1,l),H^{2m})}$-norm. 
For the growth of the first order increment and its multivariate spatial derivatives we may use
\begin{equation}\label{scalparasystlin10v*uniglobalupperbound}
 \begin{array}{ll}
D^{\alpha}_x\delta v^{\rho,1,l}_i(\tau,x)=\int_{{\mathbb R}^n}D^{\alpha}_xv^{r,\rho,l-1}_i(l-1,y)G_l(\tau,x-y)dy-v^{r,\rho,l-1}_i(l-1,x)\\
\\ 
-\rho_l\int_{l-1}^{\tau}\int_{{\mathbb R}^n}\left( \sum_{j=1}^nD^{\alpha^j}_x\left(  v^{r,\rho,l-1}_j(s,y)\frac{\partial v^{r,\rho,l-1}_i}{\partial x_j}(s,y)\right) \right) G_{l,j}(\tau-s,x-y)dyds\\
\\
+2\rho_l\int_{l-1}^{\tau}\int_{{\mathbb R}^n}\int_{{\mathbb R}^n}\left( \sum_{j,m=1}^n \left( \left( D^{\alpha^m}_x\frac{\partial v^{r,\rho,l-1}_j}{\partial x_m}\right) \frac{\partial v^{r,\rho,l-1}_m}{\partial x_j}\right)\right)  (l-1,y)\frac{\partial}{\partial x_i}K_n(z-y) \times\\
\\
\times G_{l,m}(\tau-s,x-z)dydzds,
\end{array}
\end{equation} 
where the reasoning is similar as in the preceding section. The next step is to show that we can get a uniform bound from this for an extended control function (as in iv) in our list), i.e., an upper bound which does not depend on the time step number $l\geq 1$. First we reconsider the argument for (\ref{boundedimpl}) 
for some large constant $C>2$ and all $l\geq 1$ in case of a control function $r^l_i$ which have a asymmetry build in. Assume that (\ref{boundedimpl}) holds for the extended control function and for $|\alpha|\leq m$ and assume in addition that
\begin{equation}
|D^{\alpha}_x r^l_i(l,x)|\in \left[2C^2,2C^2+1\right].
\end{equation}
In this case we know that for small $\rho_l>0$ 
\begin{equation}\label{incrementcontrol}
\begin{array}{ll}
{\big |}\int_{l-1}^{\tau}\phi^{l}_j(s,y)G_l(\tau-s,x-y)dyds{\big |}\\
\\
={\big |}\int_{l-1}^{l}\left( -\frac{v_j^{r,\rho,l-1}}{C}(l-1,.)-\frac{r^{l-1}_j}{C^2}(s,y)\right) G_l(\tau-s,.-y)dyds{\big |}\\
\\
={\big |}\int_{l-1}^{l}\left( -\frac{v_j^{r,\rho,l-1}}{C}(l-1,.)-\frac{r^{l-1}_j}{C^2}(s,y)\right) G_l(\tau-s,.-y)dyds{\big |}\\
\\
\geq {\big |}\int_{l-1}^{l}\left( -1+2\right) \frac{1}{2}ds{\big |}\geq \frac{1}{2},
\end{array}
\end{equation}
while
\begin{equation}\label{incrementcontrol}
\begin{array}{ll}
|\delta v^{r,\rho,l}_i|=|v^{r,\rho,l}_i(l,.)-v^{r,\rho,l-1}_i(l-1,.)|\leq \frac{1}{4}.
\end{array}
\end{equation}
Similar for multivariate spatial derivative of order $|\alpha| >0$. Hence we conclude that we have
\begin{equation}
|D^{\alpha}_x v^{r,\rho,l}_i(l,.)|\in [0,C]~\mbox{ and }|D^{\alpha}_x r^l_i(l,.)|\in [0,2C^2+1],
\end{equation}
and this implies that for the uncontrolled velocity functions that
\begin{equation}
|D^{\alpha}_x v^{\rho,l}_i(l,.)|\in [0,2C^2+1]
\end{equation}
for $|\alpha|\leq m$. Hence we need to establish the upper bound for the controlled value function in the case of an extended control function. The weaker weight of the extension $-\frac{r^{l-1}_i}{C^2}(l-1,.)$ helps. Given some $x\in {\mathbb R}^n$, if
\begin{equation}
v^{r,\rho,l-1}_i(l-1,x)\in [-C,C]~\mbox{ and }~r^{l-1}_i(l-1,x)\in \left[-2C^2-1,2C^2+1\right] 
\end{equation}
have the same sign then $-\frac{r^{l-1}_i}{C^2}(l-1,.)\in \left\lbrace -2+\frac{1}{C},2+\frac{1}{C}\right] $ and $-\frac{v^{r,\rho,l-1}_i}{C}(l-1,x)\in [-1,1]$, and for small $\rho_l$ 
(\ref{boundedimpl}) is satisfied by construction. Indeed for small $\rho_l\sim \frac{1}{C^3}$ except for the additional source term itself related to the
 term
 \begin{equation}\label{termincontrol}
\int_{l-1}^{\tau}\phi^{l}_j(s,y)G_l(\tau-s,x-y)dyds
\end{equation}
 all additional terms in the controlled equation for $v^{r,\rho,l}_i$ which are related to (\ref{termincontrol}) have a factor $\rho_l$. Hence, we observe that
\begin{equation}\label{incrementcontrol}
\begin{array}{ll}
\delta v^{r,\rho,l}_i=v^{r,\rho,l}_i(l,.)-v^{r,\rho,l-1}_i(l-1,.)\\
\\
=\sum_{k\geq 2}\delta v^{r,\rho,l,k}_i(l,.)+\int_{l-1}^{\tau}\phi^{l}_j(s,y)G_l(\tau-s,x-y)dyds+\rho_l\left(\cdots \right) \\
\\
\leq \frac{1}{C}+\sum_{k\geq 2}\delta v^{\rho,l,k}_i(l,.)+\int_{l-1}^{l}\phi^{l}_j(s,y)G_l(\tau-s,.-y)dyds\\
\\
\in \left[-C,C \right]. 
\end{array}
\end{equation}
for appropriate $C>2$. Similar for spatial derivatives of order $|\alpha|\leq m$.
Next, given some $x\in {\mathbb R}^n$, assume that
\begin{equation}
v^{r,\rho,l-1}_i(l-1,x)\in [-C,C]~\mbox{ and }~r^{l-1}_i(l-1,x)\in \left[-2C^2-1,2C^2+1\right] 
\end{equation}
have different signs. In this case we observe that the modulus of the control function decreases more that the uncontrolled value functions can grow at one time step with small time step size $\rho_l\sim \frac{1}{C^3}$, i.e., we have  
\begin{equation}
{\big |}r^l_i(l,.){\big |}-\sum_{k=1}^{\infty}\delta v^{\rho,l,k}_i(l,.)\geq \frac{1}{C},
\end{equation} 
and similar for spatial derivatives of of order $|\alpha|\leq m$. Furthermore, we observe that for small step site $\rho_l$ we still have $r^{l-1}_i(l-1,x)\in \left[-2C^2-1,2C^2+1\right]$.
Concerning generalization to the highly degenerate Navier Stokes equation model, there is only one element which we need to change in the argument above. Note that we have inductively
\begin{equation}
{\big |}D^{\alpha}_xv^{r,\rho,l-1}_i(l-1,.){\big |}\leq \frac{C}{1+|x|^q}
\end{equation}
for all $1\leq i\leq n$ and all multiindices $\alpha$ with $|\alpha|\leq m$ for $q$ as in the statement of the local contraction theorem, and by inheritance of polynomial decay and local contraction we have
\begin{equation}
{\big |}D^{\alpha}_xv^{r,\rho,l,k}_i(\tau,.){\big |}\leq \frac{2C}{1+|x|^q}
\end{equation}
for all $k\geq 1$, $\tau\in [l-1,l]$, and for all $1\leq i\leq n$ and all multiindices $\alpha$ with $|\alpha|\leq m$ for $q$ as well. 
The additional problem is that the standard estimates of the H\"{o}rmander diffusion have an additional polynomial growth factor with respect to the spatial argument $x$. This is no problem for our scheme as the control function has the term $-\left( v^{\rho,l,1}_j-v^{r,\rho,l-1}_j(l-1,.)\right)$ built in in its definition, and all other terms contain products of approximating value functions as factors which offset this additional polynomial growth factor. However, we do not have this effect for the source terms $\phi^l_i$ in our dynamic definition of the control functions $r^l_i$. Therefore we define
\begin{equation}\label{controllgen}
\begin{array}{ll}
r^{l}_j-r^{l-1}_j=-\left( v^{\rho,l,1}_j-v^{r,\rho,l-1}_j(l-1,.)\right)\\
\\
+\int_{l-1}^{\tau}\frac{2C}{1+|y|^q}\phi^{l}_j(s,y)G^l_H(\tau-s,x-y)dyds,
\end{array}
\end{equation}  
where the source term is of the the same form as in (\ref{sourceleray*}), i.e. the value function and the control function now just refer two the value functions and control functions of the general scheme. Note that the convolution is useful in this respect as the additional factor in (\ref{controllgen}) does not change the sign as we consider spatial derivatives. The argument for a global bound of the Leray projection term is then analogous as in the classical model.

\section{Proof of local contraction result}
We consider the essential case of contraction results for derivatives up to order $|\alpha|\leq 2$. The extension to order $m>2$ is straightforward.
The inheritance of polynomial decay of the local higher order correction terms and the inductive assumption of polynomial decay described above facilitates the proof of the local contraction result, because we have upper bounds of approximating value functions $v^{r,\rho,l,k}_i$ which have polynomial decay of a certain order and this leads to the simple definition of the constants $C_G$ and $C_B$ above which play a natural role in our contraction estimate via classical representations of functional increments $\delta v^{\rho,l,k}_i$. We emphasize again that the first approximation
$v^{\rho,l,1}_i$ at time step $l\geq 1$
solves an equation with data $v^{r,\rho,l-1}_i(l-1,.)$,
i.e., we consider the local construction $v^{\rho,l}_i=v^{\rho,l,1}_i+\sum_{k=2}^{\infty}\delta v^{\rho,l,k}_i$ and then we add at each time step the control function increments $\delta r^l_i$ in order to estimate the growth with respect to the time step number $l\geq 1$. This is different to a direct approach which involves the control function in the local construction. In the following we denote
\begin{equation}
v^{\rho,l,k}_i=v^{\rho,l,1}_i+\sum_{p=2}^{k}\delta v^{\rho,l,p}_i,
\end{equation}
keeping in mind that the controlled function $v^{r,\rho,l-1}_i(l-1,.)$ are part of the Cauchy problem which defines $v^{\rho,l,1}_i$.
Compared to the local contraction result for classical Navier Stokes equation models with constant viscosity for the degenerate Navier-Stokes equation models we have to consider two additional aspects. One of these aspects is the different standard a priori estimate for H\"{o}rmander densities, which includes an additional polynomial factor with respect to the spatial variables. We have to take care of this aspect for the $L^2$- and $H^1$-contraction estimates. The second new aspect is that we need a local adjoint of densities in order to deal with $H^m$ estimates for $m\geq 2$. Actually the first order estimates are essential since they are with respect to differentiable functions where the functions themselves and their first order spatial derivatives vanish at spatial infinity. Such spaces are closed, but we consider higher order Sobolev spaces as well.  We emphasize the essential differences to the classical model. For the classical model with constant viscosity the result may be obtained with weaker assumptions concerning the order of polynomial decay. We refer to our notes in \cite{KB2,KB3} for the discussion of local contraction in the case of the classical model. 
We consider the essential $|.|_{C^0\left((l-1,l), H^1\right) }$-estimates first. We have observed that the functional increments $\delta v^{\rho,l,k+1}_i=v^{\rho,l,k+1}_i-v^{\rho,l,k}_i,~1\leq i\leq n$ solve
(\ref{deltaurhok0}). Furthermore, if $G^l_H$ denotes the fundamental solution of the equation
\begin{equation}
\frac{\partial G^l_H}{\partial \tau}-\rho_l\frac{1}{2}\sum_{j=0}^mV_{j}^2G^l_H=0
\end{equation}
on the domain $[l-1,l]\times {\mathbb R}^n$, then we have the representations
 \begin{equation}\label{scalparasystlin10vhercr}
\begin{array}{ll}
\delta v^{\rho,l,1}_i(\tau,x)=\int_{{\mathbb R}^n}v^{r,\rho,l-1}_i(l-1,y)G^l_H(\tau,x;s,y)dy
-v^{\rho,l-1}_i(l-1,x)\\
\\ 
-\rho_l\int_{l-1}^{\tau}\int_{{\mathbb R}^n}V_{B}\left[v^{r,\rho,l-1}(l-1,.)\right] v^{r,\rho,l-1}_i(l-1,y)G^l_H(\tau,x;s,y)dyds+\\
\\
\rho_l\int_{l-1}^{\tau}\int_{{\mathbb R}^n}\int_{{\mathbb R}^n}\sum_{j,m=1}^n \left(c_{jm} \frac{\partial v^{r,\rho,l-1}_m}{\partial x_j}(l-1,.)\frac{\partial v^{r,\rho,l-1}_j}{\partial x_m}(l-1,.)\right) (\tau,y)\times\\
\\
\times K^{\mbox{ell}}_{n,i}(z-y)G^l_H(\tau,x;s,z)dydzds,
\end{array}
\end{equation} 
and
\begin{equation}\label{deltaurhok0hercr}
\begin{array}{ll}
\delta v^{\rho,l,k+1}_i(\tau,x)=\\
\\ 
-\rho_l\int_{l-1}^{\tau}\int_{{\mathbb R}^n}\left( V_{B}\left[v^{\rho,l,k}\right] \delta v^{\rho,l,k}_i+V_{B}\left[\delta v^{\rho,l,k}\right] v^{\rho,l,k}_i\right) (s,y)\times\\
\\
\times G^l_H(\tau,x;s,y)dyds+\rho_l\int_{l-1}^{\tau}\int_{{\mathbb R}^n}\int_{{\mathbb R}^n}K^{\mbox{ell}}_{n,i}(z-y)\times\\
\\
{\Big (} \left( \sum_{j,m=1}^n c_{jm}\left( v^{\rho,l,k}_{m,j}+v^{\rho,l,k-1}_{m,j}\right)(s,y) \right)  \delta v^{\rho,l,k}_{j,m}(s,y) {\Big)}\times\\
\\
\times G^l_H(\tau,x;s,z)dydzds.
\end{array}
\end{equation}
The representation of the first functional increment $\delta v^{\rho,l,1}_i$ shows that we loose some order of spatial polynomial decay in the first approximation step. However, in the representation of the higher order approximation increments $\delta v^{\rho,l,k}_i$ we have products of approximation functions (of lower approximation order) which implies that for the higher order approximation increments polynomial decay of a certain order is preserved if it is larger enough.
The assumption of polynomial decay in the statement of the local contraction theorem of the data $v^{r,\rho,l-1}_i(l-1,.)$ at time step $l\geq 1$ in the statement of the local contraction theorem  implies that for $m\geq 2$ for small $\rho_l>0 $ we have
\begin{equation}\label{firstincrementup}
{\big |}D^{\alpha}_xv^{r,\rho,l,1}_i{\big |}\leq \frac{1}{4}.
\end{equation}
Note here, that we may assume that $C^{l-1}$ is chosen large enough such that $C^{l-1}>1$ and (\ref{firstincrementup}) is satisfied for
\begin{equation}
\rho_{l}\leq \frac{1}{C^{l-1}}.
\end{equation}
Furthermore, we may assume w.l.o.g. that $C_B,C_G>1$.
We have observed this above in the case of the classical model and using the assumption of polynomial decay of the data we get this by a similar reasoning starting from (\ref{scalparasystlin10vhercr}).
We shall observe that for the higher order correction terms we have for $ k\geq 2$ and for appropriate $\rho_l>0$ a spatial decay
\begin{equation}
\max_{i\in \left\lbrace 1,\cdots ,n\right\rbrace} \sup_{\tau\in [l-1,l]}\sum_{|\alpha|\leq 1}{\big |}D^{\alpha}_xv^{\rho,l,k}_i(\tau,x){\big |}\leq \frac{2C^{l-1}}{1+|x|^q}
\end{equation}
for $q\geq 3\max{|\alpha|\leq 2}m_{0,\alpha,0}$ as in the statement of the local contraction theorem.
Outside a ball $B_{\epsilon}(x)$ of radius $\epsilon$ around $x$ we can estimate classical representations of increments $\delta v^{r,\rho,l,k}_i(\tau,x)$ via Kusuoka-Stroock or H\"{o}rmander estimates. Inside a local ball we have local integrability of the H\"{o}rmander density $G^l_H$ and its first order derivatives. Therefore we split up   
the fundamental solution $G^l_H$ with a partion of unity $\phi^x_{\epsilon},~(1-\phi^x_{\epsilon})$ where $\phi^x_{\epsilon}$ is supported on $B_{\epsilon}(x)$ and satisfies $\phi^x_{\epsilon}(x)=1$. Furthermore, we may choose $\phi^x_{\epsilon}\in C^{\infty}(B_{\epsilon}(x))$ with bounded derivatives (use the standard elements of partitions of unity).
We write
\begin{equation}
G^l_H=\phi^x_{\epsilon}G^l_H+(1-\phi^x_{\epsilon})G^l_H.
\end{equation}
and 
\begin{equation}
G^l_{H,i}=\phi^x_{\epsilon}G^l_{H,i}+(1-\phi^x_{\epsilon})G^l_{H,i}.
\end{equation}
We may write the integral in (\ref{deltaurhok0hercr}) accordingly for each given $x\in {\mathbb R}^n$  with two summands, and then use (\ref{apriorih0}) and the Kusuoka Stroock a priori estimates for the respective summands in order to get upper bounds. The estimate of the local integral around $x$ (supported in $B_{\epsilon}(x)$ may gives a certain constant which we absorb in the definition of $C_G$ above. Note that inductively we have
\begin{equation}
{\big |}D^{\alpha}_xv^{\rho,l-1}_i(l-1,.){\big |}\leq \frac{C^{l-1}}{1+|x|^q}
\end{equation}
for some $q\geq \max_{|\alpha|\leq 2}3m_{0,\alpha,0}+2n+2$ for all multiindices $\alpha$ with $|\alpha|\leq m$, where we know that this behavior is inherited by the higher order local approximations $v^{\rho,l,k}_i$ for $k\geq 2$. Hence the constants $C_B,C_G$ and $C_K$ are well-defined in section 2 above. 
First for $|\alpha|\leq 1$ we have
\begin{equation}\label{deltaurhok0hercr1}
\begin{array}{ll}
{\big |}D^{\alpha}_x\delta v^{\rho,l,k+1}_i{\big |}_{C^0\left( (l-1,l),L^{\infty}\right) }\leq \\
\\ 
{\big |}\rho_l\int_{{\mathbb R}^n}\left( V_{B}\left[v^{\rho,l,k}\right] \delta v^{\rho,l,k}_i+V_{B}\left[\delta v^{\rho,l,k}\right] v^{\rho,l,k}_i\right) (s,y)\times\\
\\
\times D^{\alpha}_xG^l_H(.,.;s,y)dyds{\big |}_{C^0\left( (l-1,l),L^{\infty}\right)}+{\Big |}\rho_l\int_{l-1}^{\tau}\int_{{\mathbb R}^n}\int_{{\mathbb R}^n}K^{\mbox{ell}}_{n,i}(z-y)\times\\
\\
{\Big (} \left( \sum_{j,m=1}^nc_{jm}\left( v^{\rho,l,k}_{m,j}+v^{\rho,l,k-1}_{m,j}\right)(s,y) \right)  \delta v^{\rho,l,k}_{j,m}(s,y) {\Big)}\times\\
\\
\times D^{\alpha}_xG^l_H(.,.;s,z)dydzds{\Big |}_{C^0\left( (l-1,l),L^{\infty}\right)}\\
\\
\leq {\big |}\rho_l\left( V_{B}\left[v^{\rho,l,k}\right] \delta v^{\rho,l,k}_i+V_{B}\left[\delta v^{\rho,l,k}\right] v^{\rho,l,k}_i\right) \times \\
\\
\times \left( 1+|.|^{q-2\max_{|\alpha|\leq 2}m_{0,\alpha,0}-2n-2}\right) {\big |}_{C^0\left( (l-1,l),L^{\infty}\right)}\\
\\
+{\Big |}\rho_l\int_{l-1}^{\tau}\int_{{\mathbb R}^n}K^{\mbox{ell}}_{n,i}(.-y)\times\\
\\
{\Big (} \left( \sum_{j,m=1}^n|c_{jm}|\left( v^{\rho,l,k}_{m,j}+v^{\rho,l,k-1}_{m,j}\right)(.,y) \right)  \delta v^{\rho,l,k}_{j,m}(.,y) {\Big)}\\
\\
\left( 1+|y|^{q-2\max_{|\alpha|\leq 2}m_{0,\alpha,0}-2n-2}\right){\Big |}_{C^0\left( (l-1,l),L^{\infty}\right)}
\end{array}
\end{equation}
The second term of the right side of (\ref{deltaurhok0hercr1}) has another spatial convolution with the generalized Laplacian kernel $K^{\mbox{ell}}$. However, we have
\begin{equation}
\int_{{\mathbb R}^n}\frac{1}{1+|y|^{n+2}}K^{\mbox{ell}}_{,i}(.-y)dy\in L^1
\end{equation}
for all $1\leq i\leq n$ by the Young inequality, because locally the $i$th partial spatial derivative  $K^{\mbox{ell}}_{,i}$ is in $L^1$ and outside a ball it is in $L^2$. Hence, we have
\begin{equation}\label{deltaurhok0hercr1term}
\begin{array}{ll}
{\Big |}\rho_l\int_{{\mathbb R}^n}K^{\mbox{ell}}_{n,i}(.-y)
{\Big (} \left( \sum_{j,m=1}^n|c_{jm}|\left( v^{\rho,l,k}_{m,j}+v^{\rho,l,k-1}_{m,j}\right)(s,y) \right)  \delta v^{\rho,l,k}_{j,m}(.,y
 {\Big)}\\
\\\left( 1+|y|^{q-2\max_{|\alpha|\leq 2}m_{0,\alpha,0}-2n-2}\right){\Big |}_{C^0\left( (l-1,l),L^{\infty}\right)}\\
\\
\leq {\Big |}\rho_lC_K
{\Big (} \left( \sum_{j,m=1}^n|c_{jm}|\left( v^{\rho,l,k}_{m,j}+v^{\rho,l,k-1}_{m,j}\right) \right)  \delta v^{\rho,l,k}_{j,m} {\Big)}\times\\
\\
\times\left( 1+|.|^{q-2\max_{|\alpha|\leq 2}m_{0,\alpha,0}-n}\right) {\Big |}_{C^0\left( (l-1,l),L^{\infty}\right)}
\end{array}
\end{equation}
Hence, for $|\alpha|\leq 1$ we have
\begin{equation}\label{deltaurhok0hercr1intermediate}
\begin{array}{ll}
{\big |}D^{\alpha}_x\delta v^{\rho,l,k+1}_i{\big |}_{L^{\infty}\times L^{\infty}}\leq \\
\\
\leq {\big |}\rho_lC_{G}\left( V_{B}\left[v^{\rho,l,k}\right] \delta v^{\rho,l,k}_i+V_{B}\left[\delta v^{\rho,l,k}\right] v^{\rho,l,k}_i\right)\left( 1+|.|^{q-2\max_{|\alpha|\leq 2}m_{0,\alpha,0}-2n-2}\right){\big |}_{C^0\left( (l-1,l),L^{\infty}\right)}\\
\\
+{\Big |}\rho_lC_KC_G
{\Big (} \left( \sum_{j,m=1}^n|c_{jm}|\left( v^{\rho,l,k}_{m,j}+v^{\rho,l,k-1}_{m,j}\right) \right)  \delta v^{\rho,l,k}_{j,m} {\Big)}\left( 1+|.|^{q-2\max_{|\alpha|\leq 2}m_{0,\alpha,0}-n}\right){\Big |}_{C^0\left( (l-1,l),L^{\infty}\right)}
\end{array}
\end{equation}
This means that for some constant $c(n)$ which depends only on the dimension $n$ and on the order $m\geq 2$ (which determines the number of terms involved in our estimates) we have
\begin{equation}\label{deltaurhok0hercr1}
\begin{array}{ll}
\max_{j\in \left\lbrace 1,\cdots ,n\right\rbrace }\sup_{\tau\in [l-1,l],x\in {\mathbb R}^n}{\big |}D^{\alpha}_x\delta v^{\rho,l,k+1}_i(\tau,x){\big |}\leq\\ 
\\
\rho_lc(n)\left(  2C_BC_GC^l_k+C_K\sum_{j,m=1}^n|c_{jm}|\left( C^l_k+C^l_{k-1}\right)\right)\left( 1+|.|^{q-2\max_{|\alpha|\leq 2}m_{0,\alpha,0}-n}\right)\times\\
\\
\times \max_{j\in \left\lbrace 1,\cdots ,n\right\rbrace }\sum_{|\alpha|\leq 1}\sup_{\tau\in [l-1,l],x\in {\mathbb R}^n}{\big |}D^{\alpha}_x\delta v^{\rho,l,k}_j(\tau,x){\big |}.
\end{array}
\end{equation}
This means that an upper bound for the contraction constant for 
$$\sum_{|\alpha|\leq 1}\sup_{\tau\in [l-1,l],x\in {\mathbb R}^n}\left(1+|x|^n \right) {\big |}D^{\alpha}_x\delta v^{\rho,l,k}_j(\tau,x){\big |}$$ is
\begin{equation}
\rho_lc(n)\left(  2C_BC_GC^l_k+C_K\sum_{j,p=1}^nC^m_{jp}\left( C^l_k+C^l_{k-1}\right)\right)\left(1+|x|^{q-2\max_{|\alpha|\leq 2}m_{0,\alpha,0}} \right)
\end{equation}
where we may use the constants $C^m_{jp}$ defined in the section on the statement of the contraction result. Note that $m$ denotes here the order of spatial derivatives considered and at this point $m=1$ is sufficient. For higher order estimates $m\geq 2$ has to be adapted accordingly.
Since
\begin{equation}
v^{\rho,l,k}_i=v^{\rho,l,1}_i+\sum_{p=2}^{k}\delta v^{\rho,l,k}_i
\end{equation}
and we observe that the order of spatial polynomial decay is inherited by the higher order correction terms, we know that the order of spatial decay at infinity of $v^{\rho,l,k}_i$ is less or the the same as the order of spatial decay behavior of $v^{\rho,l,1}_i$.
Hence, we have
\begin{equation}
C^l_k\lesssim C^l_1\lesssim \frac{1}{1+|.|^{q-\max_{|\alpha|\leq 2}m_{0,\alpha,0}}},
\end{equation}
and
\begin{equation}
\begin{array}{ll}
\sum_{|\alpha|\leq 1}\sup_{\tau\in [l-1,l]}{\big |}D^{\alpha}_x\delta v^{\rho,l,k}_j(\tau,.){\big |}\\
\\
\lesssim \sum_{|\alpha|\leq 1}\sup_{\tau\in [l-1,l]}{\big |}D^{\alpha}_x\delta v^{\rho,l,1}_j(\tau,.){\big |}\lesssim \frac{1}{1+|.|^{q-\max_{|\alpha|\leq 2}m_{0,\alpha,0}}}
\end{array}
\end{equation}
Hence for $|\alpha|\leq 1$ we get the contraction result for
\begin{equation}\label{contractionconstant}
\rho_l\leq \frac{1}{c(n)\left(  \left( 2C_BC_G+C_K\sum_{j,p=1}^nc_{jp}\right) 2\left( C^{l-1}+1\right)\right) }
\end{equation}

In order to estimate $\sum_{|\alpha|\leq m}\sup_{\tau\in [l-1,l],x\in {\mathbb R}^n}{\big |}D^{\alpha}_x\delta v^{\rho,l,k}_j(\tau,x){\big |}$ for $m\geq 2$ we need the local adjoint for the truncations of the densities $G^l_h$ for local estimates around the argument $x$. We may then shift one derivative to the integrand involving the approximations of the value functions of order $k$ as explained above. Generic adaption of the constant $c(n)$ (depending only on dimens on and the number of terms involved in the representations of the increments $\delta v^{\rho,l,k}_i$ and their derivatives) leads to the same constant as in (\ref{contractionconstant}). For the stronger norms we also need the following additional consideration.
For given $x\in {\mathbb R}^n$ we write
\begin{equation}
\begin{array}{ll}
G^l_H(\tau,x;s,y)=\phi^x_{\epsilon}(x)G^l_H(\tau,x;s,y)+(1-\phi^x_{\epsilon}(x))G^l_H(\tau,x;s,y)\\
\\
=\phi^x_{\epsilon}(x)
G^{l,*}_H(s,y;\tau,x)+(1-\phi^x_{\epsilon}(x))G^l_H(\tau,x;s,y),
\end{array}
\end{equation}
where for small $\epsilon >0$ we know that a local adjoint $G^{l,*}_H$ exists.
Spatial derivatives of order $\alpha$ of the summand $(1-\phi^x_{\epsilon})G^l_H$ can be estimated by the Kusuoka Stroock estimates.
This is also true for multiindices $\alpha$ with $|\alpha|\geq 2$. For the other summand $\phi^x_{\epsilon}G^l_H$ we only have local integrability for derivatives up to first order. For this summand we can use the adjoint and shift  spatial derivatives to the approximating value functions $v^{\rho,l,k}_i,~\delta v^{\rho,l,k}_i$ and $v^{\rho,l-1}_i(l-1,.),~\delta v^{\rho,l-1}_i(l-1,.)$. We may then use the representations
 \begin{equation}\label{scalparasystlin10vhercrend1}
\begin{array}{ll}
\delta v^{\rho,l,1}_i(\tau,x)=\\
\\
\int_{{\mathbb R}^n}v^{r,\rho,l-1}_i(l-1,y)\left( \phi^x_{\epsilon}(x)G^{l,*}_H(s,y;\tau,x)+(1-\phi^x_{\epsilon}(x))G^l_H(\tau,x;s,y)\right) dy\\
\\
-v^{r,\rho,l-1}_i(l-1,x)\\
\\ 
-\rho_l\int_{l-1}^{\tau}\int_{{\mathbb R}^n}V_{B}\left[v^{r,\rho,l-1}(l-1,.)\right] v^{r,\rho,l-1}_i(l-1,y)\times\\
\\
\times\left( \phi^x_{\epsilon}(x)G^{l,*}_H(s,y;\tau,x)+(1-\phi^x_{\epsilon}(x))G^l_H(\tau,x;s,y)\right) dyds+\\
\\
\rho_l\int_{l-1}^{\tau}\int_{{\mathbb R}^n}\int_{{\mathbb R}^n}\sum_{j,m=1}^n \left(c_{jm} \frac{\partial v^{r,\rho,l-1}_m}{\partial x_j}(l-1,.)\frac{\partial v^{r,\rho,l-1}_j}{\partial x_m}(l-1,.)\right) (\tau,y)\times\\
\\
\times\frac{\partial}{\partial x_i}K^{\mbox{ell}}_n(z-y)\left( \phi^x_{\epsilon}(x)G^{l,*}_H(s,y;\tau,x)+(1-\phi^x_{\epsilon}(x))G^l_H(\tau,x;s,y)\right) dydzds,
\end{array}
\end{equation} 
and
\begin{equation}\label{deltaurhok0hercrend2}
\begin{array}{ll}
\delta v^{\rho,l,k+1}_i(\tau,x)=\\
\\ 
-\rho_l\int_{l-1}^{\tau}\int_{{\mathbb R}^n}\left( V_{B}\left[v^{\rho,l,k}\right] \delta v^{\rho,l,k}_i+V_{B}\left[\delta v^{\rho,l,k}\right] v^{\rho,l,k}_i\right) (s,y)\times\\
\\
\times \left( \phi^x_{\epsilon}(x)G^{l,*}_H(s,y;\tau,x)+(1-\phi^x_{\epsilon}(x))G^l_H(\tau,x;s,y)\right)dyds\\
\\
+\rho_l\int_{l-1}^{\tau}\int_{{\mathbb R}^n}\int_{{\mathbb R}^n}K^{\mbox{ell}}_{n,i}(z-y)
{\Big (} c_{jm}\left( \sum_{j,m=1}^n\left( v^{\rho,l,k}_{m,j}+v^{\rho,l,k-1}_{m,j}\right)(s,y) \right)  \delta v^{\rho,l,k}_{j,m}(s,y) {\Big)}\times\\
\\
\times \left( \phi^x_{\epsilon}(x)G^{l,*}_H(s,y;\tau,x)+(1-\phi^x_{\epsilon}(x))G^l_H(\tau,x;s,y)\right)dydzds.
\end{array}
\end{equation}
Spatial derivatives are then treated as described. Here we may use Leibniz rule where we note that for $\tau-s>0$ we have 
\begin{equation}
\begin{array}{ll}
D^{\alpha}_x\phi^x_{\epsilon}(x)G^{l,*}_H(s,y;\tau,x)=\sum_{|\beta|\leq |\alpha|}\binom{\alpha}{\beta}D^{\alpha-\beta}\phi^x_{\epsilon}(x)D^{\beta}G^{l}_H(\tau,x;s,y)\\
\\
=\sum_{|\beta|\leq |\alpha|}\binom{\alpha}{\beta}D^{\alpha-\beta}\phi^x_{\epsilon}(x)D^{\beta}_yG^{l,*}_H(s,y;\tau,x)
\end{array}
\end{equation}
The derivatives of order $|\beta|>1$ are the shifted by partial integration. Then we can proceed as before in the case of $C^0\times H^1$-norms. 
 \footnotetext[1]{
\texttt{{kampen@wias-berlin.de, kampen@mathalgorithm.de}}.}



\begin{thebibliography}{19}
\baselineskip=12pt


\bibitem{BS}
{\sc Bredberg, I. ,; Strominger, A.}: 
{Black Holes as incompressible fluids on the sphere}, arXiv 1106.3084[hep-th], 2011.



\bibitem{CFNT}
	{\sc Constantin, P., Foias, C., Nicolenko, B., Temam, R.}: {\em Integral manifolds and inertial manifolds for dissipative partial differential equations.} Springer, 1989.
	
	
\bibitem{FJRT}
	{\sc Foias, C., Manley, O., Rosa, R., Temam, R.}: {\em Navier-Stokes equations and turbulence.} CUP, 2001.
	

\bibitem{FJKT}
	{\sc Foias, C., Jolly, M., Kravchenko, M., Tity, E.}: {\em Navier-Stokes equation, determing modes, dissipative dynamical systems.} arXiv1208,5434v1, August 2012.


	
	
\bibitem{H}
	\textsc{H\"{o}rmander, L.}: {\em Hypoelliptic second order differential equations}, {\em Acta Math.,} Vol. 119,  147-171, 1967.
	
	
	
\bibitem{KNS}
	{\sc Kampen, J\"org}: {\em Constructive analysis of the Navier-Stokes equation.} arXiv10044589v6, Juli 2012.

\bibitem{K3}
 {\sc Kampen, J.,} {\em A global scheme for the incompressible Navier-Stokes equation on compact Riemannian manifolds}, arXiv: 1205.4888v4,  June 2012. (in revision)


 \bibitem{KB1}
{\sc Kampen, J.} {\em On the multivariate Burgers equation and the incompressible Navier-Stokes equation (part I)}, arXiv:0910.5672v5  [math.AP], Mar. 2011.	
	
\bibitem{KB2}
{\sc Kampen, J.} {\em On the multivariate Burgers equation and the incompressible Navier-Stokes equation (part II)}, arXiv:1206.6990v6  [math.AP], Jan. 2013.

\bibitem{KB3}
{\sc Kampen, J.} {\em On the multivariate Burgers equation and the incompressible Navier-Stokes equation (part III)}, arXiv:1210.2602v9 [math.AP] (Aug, 2013) 




\bibitem{KE2}
{\sc Kampen, J.} {\em
On a class of singular solutions to the incompressible 3-D Euler equation}, arXiv:1209.6250 [math.AP], Oct. 2012. (in revision)

\bibitem{KIF2}
{\sc Kampen, J.} {\em Some infinite matrix analysis, a Trotter product formula for dissipative operators, and an algorithm for the incompressible Navier-Stokes equation}, arXiv:1212.2403v3  [math.AP], Aug. 2013. 


\bibitem{KH}
{\sc Kampen, J.} {\em Density estimates for differential equations of second order satisfying a weak H\"{o}rmander condition}, arXiv 1307.5399v1, July 2013. (extended version in preperation)

\bibitem{KT}
{\sc Koch, H.} {\em Well-posedness for the Navier Sokes equations}, Adv. Math. 157, no. 1, 22-35, 2001.

\bibitem{KS} 
{\sc Kusuoka, S., Stroock, D.}: {\em Application of Malliavin calculus II} J. Fac. Sci. Univ. Tokio, Sect. IA, Math. 32, p. 1-76, 1985.

\bibitem{L}
{\sc Leray, J.} {\em Sur le Mouvement d'un Liquide Visquex Emplissent l'Espace}, Acta Math. J. (63), 193-248, 1934.


\bibitem{Tao}
{\sc Tao, T.} {\em Localisation and compactness properties of the Navier Stokes global regularity problem}, arXiv:1108.1165v4  [math.AP], 2011. 



 \end{thebibliography}
\end{document}